\documentclass[12pt]{amsart}

\usepackage{todonotes}

\usepackage{graphicx}
\input xy
\xyoption{all}
\usepackage[pdftex]{hyperref}
\hypersetup{colorlinks,citecolor=blue,linktocpage,hyperindex=true,backref=true}
\usepackage{tikz}
\usepackage[arrow, matrix, curve]{xy}
\usepackage{amssymb,amscd}

\newtheorem{theorem}{Theorem}[section]
\newtheorem{lemma}[theorem]{Lemma}
\newtheorem{proposition}[theorem]{Proposition}
\newtheorem{corollary}[theorem]{Corollary}
\theoremstyle{definition}
\newtheorem{definition}[theorem]{Definition}

\newtheorem{example}[theorem]{Example}

\newtheorem{remark}[theorem]{Remark}

\newtheorem{num}[theorem]{}

\numberwithin{equation}{section}

\newcommand\rank{\operatorname{rank}}
\newcommand\tor{\operatorname{tor}}
\newcommand\ab{\operatorname{ab}}

\newcommand\abpt{\ab(PT)}
\newcommand\aperf{\cA_6^{\rm perf}}
\newcommand\avor{\cA_6^{\rm vor}}
\newcommand\asat{\cA_6^{\rm sat}}
\newcommand\agperf{\cA_g^{\rm perf}}
\newcommand\agvor{\cA_g^{\rm vor}}
\newcommand\agsat{\cA_g^{\rm sat}}
\newcommand{\ratmap}{\dashrightarrow}
\newcommand{\acts}{\curvearrowright}
\newcommand{\onto}{\twoheadrightarrow}
\newcommand{\into}{\hookrightarrow}
\newcommand{\lra}{\longrightarrow}

\newcommand\vor{^{\rm vor}}
\newcommand\vapath{\operatorname{path}}
\newcommand\Hur{\operatorname{Hur}}
\newcommand{\wt}{\widetilde}

\newcommand{\Sym}{\operatorname{Sym}}

\newcommand{\R}{\operatorname{R}}
\newcommand{\codim}{\operatorname{codim}}
\newcommand{\im}{\operatorname{Im}}
\newcommand{\inv}{^{-1}}

\newcommand{\ocM}{\overline{\cM}}

\DeclareMathOperator{\Aut}{Aut}
\DeclareMathOperator{\Pic}{Pic}

\DeclareMathOperator{\PGL}{PGL}

\DeclareMathOperator{\rk}{rk}

\newcommand\wD{\widetilde D}
\newcommand{\bC}{\mathbb C}

\newcommand{\bH}{\mathbb H}
\newcommand{\bP}{\mathbb P}
\newcommand{\bQ}{\mathbb Q}
\newcommand{\bR}{\mathbb R}
\newcommand{\bZ}{\mathbb Z}
\newcommand{\cA}{\mathcal A}
\newcommand{\cC}{\mathcal C}
\newcommand{\cD}{\mathcal D}
\newcommand{\cF}{\mathcal F}

\newcommand{\cL}{\mathcal L}
\newcommand{\cM}{\mathcal M}
\newcommand{\cO}{\mathcal O}
\newcommand{\cP}{P}
\newcommand{\cR}{\mathcal R}
\newcommand{\cT}{\mathcal T}
\newcommand{\cV}{\mathcal V}
\newcommand\wcA{\widetilde\cA}

\newcommand{\of}{\overline{f}}

\newcommand{\ophi}{\overline{\varphi}}

\newcommand{\oq}{\overline{q}}

\newcommand{\tC}{\widetilde{C}}
\newcommand{\tE}{\widetilde{E}}

\newcommand{\ra}{\rightarrow}

\newcommand{\inj}{\hookrightarrow}

\newcommand{\Alb}{\operatorname{Alb}}

\newcommand{\hh}{{\overline{\mathcal H}}}
\newcommand{\hhu}{\overline{\Hur}}
\newcommand{\mm}{{\overline{\mathcal M}}}
\newcommand{\rr}{\overline{\mathcal R}}

\newcommand\hur{\operatorname{Hur}}
\newcommand\ohur{\overline{\hur}}

\begin{document}
\title[The uniformization of $\cA_6$]
{The uniformization of the moduli space\\ of principally polarized abelian 6-folds}

\author[V. Alexeev]{Valery Alexeev}
\address{Valery Alexeev: Department of Mathematics, University of Georgia
\hfill \newline\texttt{}
 \indent Athens GA 30602, USA}
\email{{\tt valery@math.uga.edu}}

\author[R. Donagi]{Ron Donagi}
\address{Ron Donagi: Department of Mathematics, University of Pennsylvania
\hfill \newline\texttt{}
\indent 209 South 33rd Street,
Philadelphia, PA 19104-6395, USA} \email{{\tt donagi@math.upenn.edu}}

\author[G. Farkas]{Gavril Farkas}
\address{Gavril Farkas: Institut f\"ur Mathematik, Humboldt-Universit\"at zu Berlin \hfill \newline\texttt{}
\indent Unter den Linden 6,
10099 Berlin, Germany}
\email{{\tt farkas@math.hu-berlin.de}}

\author[E. Izadi]{Elham Izadi}
\address{Elham Izadi: Department of Mathematics, University of California, San Diego \hfill
\indent \newline\texttt{}
\indent La Jolla, CA 92093-0112, USA}
 \email{{\tt eizadi@math.ucsd.edu}}

\author[A. Ortega]{Angela Ortega}
\address{Angela Ortega: Institut f\"ur Mathematik, Humboldt-Universit\"at zu Berlin \hfill \newline\texttt{}
\indent Unter den Linden 6, 10099 Berlin, Germany}
\email{{\tt ortega@math.hu-berlin.de}}

\maketitle
\tableofcontents

\section*{Introduction}

It is a classical idea that general principally polarized abelian varieties
(ppavs) and their moduli spaces are hard to understand, and that one
can use algebraic curves to study some special classes, such as
Jacobians and Prym varieties.  This works particularly well in small dimension, where in this way one reduces the study of all abelian varieties to the rich and concrete theory of
curves. For $g\le3$, a general ppav is a
Jacobian, and the Torelli map $\cM_g\to \cA_g$ between the moduli spaces
of curves and ppavs respectively, is birational. For $g\le5$, a general ppav is a
Prym variety by a classical result of Wirtinger
\cite{wirtinger1895untersuchungen-uber}.  In particular, for $g=5$, this gives a
uniformization of $\cA_5$ by curves, as follows. We denote by $\cR_g$
the Prym moduli space of pairs $[C,\eta]$ consisting of
a smooth curve $C$ of genus $g$ and a non-trivial $2$-torsion point $\eta \in \Pic^0(C)$.  By Donagi-Smith
\cite{donagi1981the-structure-of-the-prym}, the Prym map
$P\colon \cR_6\to \cA_5$ is generically of degree $27$, with fibers
corresponding to the configuration of the $27$ lines on a cubic surface.

The uniformization  of $\cA_g$ for $g\leq 5$ via the Prym map $P:\cR_{g+1}\rightarrow \cA_{g}$ has been used for many problems concerning ppav of small dimension.  Important applications of the Prym uniformization include the proof of Clemens and Griffiths \cite{ClemensGriffiths} respectively Mumford \cite{M74} of the irrationality of smooth cubic threefolds, which rely on the distinctions between Pryms and Jacobians, the proofs of the general Hodge conjecture for the theta divisors of general ppav, see \cite{IvS} and \cite{ITW}, or the detailed study of the cohomology and stratification of $\mathcal{A}_5$ in terms of singularities of theta divisors, see for instance \cite{CF05} or \cite{FGSMV}. The Prym map $P:\cR_6\to \cA_5$ has been also used to determine the birational type of $\mathcal{A}_5$. It has been proven in \cite{Don} that $\mathcal{R}_6$ (and hence $\mathcal{A}_5$) is unirational. Other proofs followed in  \cite{MM83} and \cite{V84}.

\vskip 4pt

The purpose of this paper is to prove a similar uniformization result
for the moduli space $\cA_6$ of principally polarized abelian varieties of dimension $6$. The idea of this construction is due to Kanev
\cite{kanev1989spectral-curves} and it uses the geometry of the $27$ lines on a cubic surface. Suppose $\pi \colon C\to \bP^1$ is a cover of degree $27$  whose monodromy group equals the Weyl group
$W(E_6)\subset S_{27}$ of the $E_6$ lattice. In particular, each smooth fibre of $\pi$ can be identified with the set of $27$ lines on an abstract cubic surface and, by monodromy, this identification carries over from one fibre to another. Assume furthermore that $\pi$ is branched over $24$ points and that over each of them the local monodromy of $\pi$ is given by a reflection in $W(E_6)$. A prominent example of such a covering $\pi:C\rightarrow \bP^1$ is given by the \emph{curve of lines} in the cubic surfaces of a Lefschetz pencil of hyperplane sections of a cubic threefold $X\subset \bP^4$, see \cite{Kan1}, as well as Section \ref{sec:Kanev-construction} of this paper. Since $\mbox{deg}(X^{\vee})=24$, such a pencil contains precisely $24$ singular cubic surfaces, each having exactly one node.

By the Hurwitz formula, we find that each such $E_6$-cover $C$ has genus $46$. Furthermore, $C$ is endowed with a symmetric correspondence $\widetilde{D}$ of degree $10$, compatible with the covering $\pi$ and defined using the intersection form on a cubic surface. Precisely, a  pair $(x,y)\in C\times C$ with $x\neq y$ and $\pi(x)=\pi(y)$ belongs to $\widetilde{D}$ if and only if the lines corresponding to the points $x$ and $y$ are incident. The correspondence $\widetilde{D}$ is disjoint from the diagonal of $C\times C$. The associated endomorphism $D:JC\rightarrow JC$ of the Jacobian  satisfies the
quadratic relation $(D-1)(D+5)=0$. Using this, Kanev \cite{Kan87} showed that the associated \emph{Prym-Tyurin-Kanev} or \emph{PTK
variety} $$PT(C,D):=\mbox{Im}(D-1)\subset JC$$ of this pair is a $6$-dimensional ppav of exponent $6$. Thus, if $\Theta_C$ denotes the Riemann theta divisor on $JC$, then $\Theta_{C | P(C,D)} \equiv6\cdot \Xi$, where $\Xi$ is a principal polarization on $P(C,D)$.

\vskip 3pt

Since the map $\pi$ has $24$ branch points
corresponding to choosing $24$ roots in $E_6$ specifying the local monodromy at each branch point, the Hurwitz scheme
$\Hur$ parameterizing degree $27$ covers $\pi\colon C\to \bP^1$ with $W(E_6)$ monodromy as above is $21$-dimensional (and
also irreducible, see \cite{kanev2006hurwitz-spaces}). The geometric construction described above induces the \emph{Prym-Tyurin}
map $$PT\colon\Hur\to \cA_6$$
between two moduli spaces  of the same dimension. The following theorem answers a conjecture raised by Kanev, see also \cite[Remark 5.5]{lange2008a-galois-theoretic-approach}:

\begin{theorem}\label{thmmain1} The Prym-Tyurin map $PT:\Hur\rightarrow \cA_6$ is generically finite. It follows that the general principally polarized abelian variety of dimension $6$ is a Prym-Tyurin-Kanev (PTK) variety of exponent $6$ corresponding to a $W(E_6)$-cover $C\rightarrow \bP^1$.
\end{theorem}

This result, which is the main achievement of this paper, gives a structure theorem for general abelian varieties of dimension $6$ and offers a uniformization for $\cA_6$ by curves with additional discrete data. Just like the classical Prym map $P:\mathcal{R}_6\rightarrow \cA_5$, it is expected that the Prym-Tyurin map $PT$ will open the way towards a systematic study of abelian $6$-folds and their moduli space. What is essential is less the fact that a general $6$-dimensional ppav is a PTK variety, but rather the rich geometric structure that Theorem \ref{thmmain1} provides, which is then of use for other applications  presented in Sections \ref{sec-Hurglobal}-\ref{sec:Petri}. An immediate consequence of Theorem \ref{thmmain1} is the following:

\begin{corollary}\label{minclass}
For every ppav $[A, \Theta]\in \cA_6$, the class $6\cdot \theta^5/5!\in H^{10}(A, \mathbb Z)$ is represented by an effective curve.
\end{corollary}

It is expected that for a general $[A, \Theta]\in \cA_6$, the minimal cohomology class $\theta^5/5!$ is not even algebraic. Coupled with Corollary \ref{minclass}, this would mean that $[A, \Theta]$ should not admit any Prym-Tyurin realization of exponent relatively prime to $6$.


\vskip 4pt

The main idea of the proof of Theorem \ref{thmmain1} is to study degenerations of
PTK varieties as the branch locus $(\bP^1, p_1+\cdots+p_{24})$ of the cover $\pi\colon C\to\bP^1$ approaches a maximally degenerate
point of $\overline{\cM}_{0,24}$. The map $PT$ becomes toroidal and its
essential properties can be read off a map of fans. Then, to show that
$PT$ is dominant, it is sufficient to show that the rays in the fan
describing the image span a $21$-dimensional vector space, i.e. that a
certain $(21\times 21)$-matrix has full rank. This can be done by an
explicit computation, once the general theory is in place. The theory of degenerations of Jacobians  \cite{alexeev2004compactified-jacobians} and  Prym
varieties in \cite{alexeev2002degenerations-of-prym} is known. One of the main goals of the present paper is an extension of the theory to
the case of PTK varieties. For our purposes we do not require
the answer to the more delicate problem of understanding the indeterminacy locus of the
period map.

\vskip 3pt

The remainder of this work focuses on several birational problems that
are related to the geometry of $\cA_6$ by Theorem \ref{thmmain1}, and on several quite
non-obvious parallels between the Prym map and the Prym-Tyurin map $PT$. Consider the space $\mathcal{H}$ classifying
$E_6$-covers $[\pi:C\rightarrow \bP^1, p_1, \ldots, p_{24}]$ together with a labeling of the set of their $24$ branch points.
In view of the structure Theorem \ref{thmmain1}, it is of compelling interest to
understand the birational geometry of this space. It admits a compactification $\hh$ which is the moduli space of \emph{twisted stable maps} from curves of genus zero into the classifying stack $\mathcal B W(E_6)$, that is, the normalization of the stack of admissible covers with monodromy group $W(E_6)$ having as source a nodal curve of genus $46$ and as target a stable $24$-pointed curve of genus $0$  (see Section \ref{sec-Hurglobal} for details). One has a finite morphism $$\mathfrak{b}:\hh\rightarrow \overline{\cM}_{0,24}.$$
In Section \ref{sec:pos}, we show that the canonical class of $\hh$ is big (Theorem \ref{kodhh}). From the point of view of $\mathcal{A}_6$, it is more interesting to study the global geometry of the  quotient space $$\hhu:=\hh/S_{24},$$ compactifying the Hurwitz space $\Hur$ of $E_6$-covers (without a labeling of the branch points). The Prym-Tyurin map $PT$ extends to a regular morphism
$PT^{\mathrm{Sat}}:\hhu\rightarrow \cA_6^{\mathrm{Sat}}$ to the Satake compactification $\cA_6^{\mathrm{Sat}}$ of $\cA_6$. Denoting by $\overline{\cA}_g:=\cA_g^{\mathrm{perf}}$ the perfect cone (first Voronoi) compactification of $\cA_g$, we establish the following result on the birational geometry of $\hhu$, which we regard as a compact master space for ppav of dimension $6$:

\begin{theorem}\label{thmmain2}
There exists a boundary divisor $E$ of $\hhu$ that is contracted by the Prym-Tyurin map $PT:\hhu\dashrightarrow \overline{\mathcal{A}}_6$, such that $K_{\hhu}+E$ is a big divisor class.
\end{theorem}

The proof of Theorem \ref{thmmain2} is completed after numerous preliminaries at the end of Section \ref{sec:WP}.

\vskip 4pt

In the course of proving Theorem \ref{thmmain2}, we establish numerous facts concerning the geometry of the space $\hhu$. One of them is a surprising link between the splitting of the rank $46$ Hodge bundle $\mathbb E$ on $\hhu$ into Hodge eigenbundles and the Brill-Noether theory of $E_6$-covers, see Theorem \ref{incarnations}. For a point $[\pi:C\rightarrow \bP^1]\in \hhu$, we denote by $D:H^0(C,\omega_C)\rightarrow H^0(C,\omega_C)$ the map induced at the level of cotangent spaces by the Kanev endomorphism and by $$H^0(C, \omega_C)=H^0(C,\omega_C)^{(+1)}\oplus H^0(C, \omega_C)^{(-5)},$$ the decomposition into the $(+1)$ and the $(-5)$-eigenspaces of holomorphic differentials respectively. Setting $L:=\pi^*(\mathcal{O}_{\bP^1}(1))\in W^1_{27}(C)$, for a general point $[\pi:C\rightarrow \bP^1]\in \Hur$, we show that the following canonical identifications hold:
$$
H^0(C,\omega_C)^{(+1)}=H^0(C,L)\otimes H^0(C,\omega_C\otimes L^{\vee}) \ \ \  $$
and
$$   H^0(C,\omega_C)^{(-5)}=\left(\frac{H^0(C,L^{\otimes 2})}{\mbox{Sym}^2 H^0(C,L)}\right)^{\vee}\otimes \bigwedge^2 H^0(C,L).
$$
In particular, the $(+1)$-Hodge eigenbundle is fibrewise isomorphic to the image of the Petri map $\mu(L): H^0(C,L)\otimes H^0(C, \omega_C\otimes L^{\vee})\rightarrow H^0(C, \omega_C)$, whenever the Petri map is injective (which happens generically along $\Hur$, see Theorem \ref{petri}). The identifications above are instrumental in expressing in Section \ref{sec:WP} the class of the  $(-5)$-Hodge eigenbundle $\mathbb E^{(-5)}$ on a partial compactification $\mathcal{G}_{E_6}$ of $\Hur$ in terms of boundary divisors. The moduli space $\mathcal{G}_{E_6}$ differs from $\hhu$ only along divisors that are contracted under the Prym-Tyurin map. Note that the class $\lambda^{(-5)}=c_1(\mathbb E^{(-5)})$ is equal to the pull-back $PT^*(\lambda_1)$ of the Hodge class $\lambda_1$  on $\overline{\cA}_6$. The explicit realization of the class $\lambda^{(-5)}$ is then used to establish positivity properties of the canonical class $K_{\hhu}$.

\vskip 4pt

An obvious question is to what extent the geometry of $\hhu$ can be used to answer the notorious problem on the Kodaira dimension of $\cA_6$. Recalling that $PT:\hhu\dashrightarrow \overline{\cA}_6$  denotes the extension of the Prym-Tyurin map outside a codimension $2$ subvariety of $\hhu$,  the pull-back divisor $PT^*(\partial \overline{\cA}_6)$ contains a unique boundary divisor $D_{E_6}$ of $\hhu$ that is not contracted by $PT$. The statement that $\cA_6$ is of general type is then equivalent to the bigness of the divisor class $7\lambda^{(-5)}-[D_{E_6}]$ on $\hhu$ (see Corollary \ref{koda6} for a more precise statement). Theorem \ref{thmmain1} implies that $\lambda^{(-5)}$ is a big class on $\hhu$, which is a weaker result. Note that it has been established in \cite{FV14} that the boundary divisor $\partial \overline{\cA}_6$ of the perfect cone compactification $\overline{\cA}_6$ is unirational.

\vskip 4pt

We are also able to describe the ramification divisor of the Prym-Tyurin map in terms of the geometry of the \emph{Abel-Prym-Tyurin curve}
$\varphi_{(-5)}=\varphi_{H^0(\omega_C)^{(-5)}}: C\rightarrow \bP^5$ given by the linear system of $(-5)$-invariant holomorphic forms on $C$.

\begin{theorem}\label{rampt}
An $E_6$-cover $[\pi:C\rightarrow \bP^1]\in \Hur$  such that the Petri map $\mu(L)$ is injective lies in the ramification divisor of the map $PT:\Hur \rightarrow \cA_6$ if and only if the Abel-Prym-Tyurin curve $\varphi_{(-5)}(C)\subset \bP^5$ lies on a quadric.
\end{theorem}

The conclusion of Theorem \ref{rampt} can be equivalently formulated as saying that  the  map
$$\mathrm{Sym}^2 H^0(C,\omega_C)^{(-5)}\longrightarrow H^0(C,\omega_C^{\otimes 2})$$
given by multiplication of sections is not injective. Note the striking similarity between this description of the ramification divisor of the Prym-Tyurin map and that of the classical Prym map $P:\cR_{g+1}\rightarrow \cA_{g}$, see \cite{BeauvilleSchottkyPrym}: A point $[C, \eta]\in \cR_{g+1}$ lies in the ramification divisor of $P$ if and only if the multiplication map for the Prym-canonical curve
$$\mbox{Sym}^2 H^0(C,\omega_C\otimes \eta)\rightarrow H^0(C, \omega_C^{\otimes 2})$$ is not injective. An important difference must however be noted. While the general Prym-canonical map $\varphi_{\omega_C\otimes \eta}:C\rightarrow \bP^{g-2}$ is an embedding when $g\geq 5$, the Abel-Prym-Tyurin map $\varphi_{(-5)} : C \ra \bP^5$ sends the ramification points lying over a branch point of the cover $\pi:C\rightarrow \bP^1$ to the same point of $\bP^5$ (see Section \ref{sec:ramif} below).

\vskip 3pt

It is natural to ask in what way the Prym-Tyurin-Kanev (PTK) varieties considered in this paper generalize classical Prym varieties. It is classical \cite{wirtinger1895untersuchungen-uber} that the Prym variety of the \emph{Wirtinger cover} of a $1$-nodal curve of genus $g$ is the Jacobian of its normalization. Thus, if $\Delta_{0}^{''}\subset \overline{\cR}_{g+1}$ is the boundary divisor of such covers and $P:\rr_{g+1}\dashrightarrow \overline{\cA}_g$ is the extension of the Prym map outside a codimension $2$ subvariety of $\rr_g$, then $P(\Delta_0^{''})$ contains the closure of the Jacobian locus in $\overline{\cA}_g$. In particular,  Jacobians arise as limits of Prym varieties. We generalize this situation and explain how ordinary Prym varieties appear as limits of PTK varieties.

\vskip 4pt

Via the Riemann Existence Theorem, a general $E_6$-cover $\pi:C\rightarrow \bP^1$ is determined by a branch divisor $p_1+\cdots+p_{24}\in \mbox{Sym}^{24}(\bP^1)$ and discrete data involving a collection of roots $r_1, \ldots, r_{24}\in E_6$ which describe the local monodromy of $\pi$ at the points $p_1, \ldots, p_{24}$. Letting two branch points, say $p_{23}$ and $p_{24}$, coalesce such that $r_{23}=r_{24}$, whereas the reflections in the remaining roots $r_1, \ldots, r_{22}$ span the Weyl group $W(D_5)\subset W(E_6)$, gives rise to a boundary divisor $D_{D_5}$ of $\hhu$. We show in Section \ref{sec:PvPTK} that the general point of $D_{D_5}$ corresponds to the following geometric data:

\vskip 4pt

\noindent (i) A genus $7$ Prym curve $[Y, \eta]\in \cR_7$, together with a degree $5$ pencil $h:Y\rightarrow \bP^1$ branched simply along the divisor $p_1+\cdots+p_{22}$; the unramified double cover $F_1\rightarrow Y$ gives rise to a degree $10$ map $\pi_{1}:F_1\rightarrow \bP^1$ from a curve of genus $13$.

\noindent (ii) A genus $29$ curve $F_2\subset F_1^{(5)}$, which is pentagonally related to $F_1$, and is thus completely determined by $F_1$. Precisely, $F_2$ is one of the two irreducible components of the locus  $$\bigl\{x_1+ \cdots+ x_5\in F_1^{(5)}: \pi_1(x_1)=\cdots =\pi_1(x_5)\bigr\}$$ inside the symmetric power $F_1^{(5)}$ of $F_1$. One has a degree $16$ cover $\pi_2:F_2\rightarrow \bP^1$ induced by $\pi_1$.

\noindent (iii) A distinguished point $q_1+\cdots+q_5\in F_2$, which determines $5$ further pairs of points $$\bigl(q_i, q_1+\cdots+\iota(q_i)+\cdots+q_5\bigr)\in F_1\times F_2$$ for $i=1, \ldots, 5$, which get identified. To $F_2$ we attach a rational curve $F_0$ at the point $q_1+\cdots+q_5$. The resulting nodal curve $C_1=F_0\cup F_1\cup F_2$ has genus $46$ and admits a map $\pi:C_1\rightarrow \bP^1$ of degree $27$ with $\pi_{|F_i}=\pi_i$ for $i=0,1,2$, where $\pi_0$ is an isomorphism. The map $\pi$ can easily be turned into an $E_6$-admissible cover having as source a curve stably equivalent to $C_1$. A general point of the divisor $D_{D_5}$ is realized in this way.

\vskip 4pt

We show in Section \ref{sec:PvPTK} that $PT([C_1, \pi])=P([F_1/Y])=P([Y, \eta]) \in \cA_6$;  furthermore, the general $6$-dimensional Prym variety from $P(\cR_7)\subset \cA_6$ appears in this way. We summarize the above discussion, showing that the restriction $PT_{D_{D_5}}$ of the Prym-Tyurin map factors via the (generically injective) Prym map $P:\rr_7\dashrightarrow \overline{\cA}_6$ in the following way.

\begin{theorem}\label{prym6}
If $D_{D_5}\subset \hhu$ is the boundary divisor of $W(D_5)$-covers defined above, one has the following commutative diagram:
\begin{equation}
\xymatrix{
              D_{D_5} \ar[r] \ar@{-->}[d]_{PT_{D_5}}  & \hhu  \ar@{-->}[d]^{PT}\\
              {\overline{\cR}_7} \ar@{-->}[r]^{P} &  {\overline{\cA}_6}
    }
\end{equation}
The fibre   $PT_{{D_5}}^{-1}\bigl(P[F_1/Y]\bigr)$ of the Prym-Tyurin map $PT_{D_5}:D_{D_5}\dashrightarrow \rr_7$ over a general genus $7$ Prym curve $[F_1/Y]\in \rr_7$ is the fibration over the curve $W^1_5(Y)$ of degree $5$ pencils on $Y$ with fibre over a pencil $A\in W^1_5(Y)$ the curve $F_2$ obtained by applying the $5$-gonal construction to $A$.
\end{theorem}

We close the introduction by discussing the structure of the paper. In
Section \ref{sec:Kanev-construction} we discuss Kanev's construction, whereas in Section \ref{sec:e6-lattice} we
collect basic facts about the $E_6$ lattice and the group $W(E_6)$
that are used throughout the paper. After recalling the theory of
degenerations of Jacobians and ordinary Prym varieties in Section \ref{sec:degs-jacobians},
we complete the proof of Theorem \ref{thmmain1} in Section \ref{sec:degPTK}, by
describing the Prym-Tyurin map in the neighborhood of a maximally
degenerate point of the space $\hhu$ of $E_6$-admissible
covers. Sections \ref{sec-Hurglobal} and \ref{sec:pos} are devoted to the birational geometry of this
Hurwitz space. The most important result is Theorem \ref{lam}
describing the Hodge class $\lambda$ on $\hhu$ in terms of boundary
divisors.  In Section \ref{secboundary} we completely describe the extended Prym-Tyurin map
$PT:\hhu \dashrightarrow \overline{\cA}_6$ to the perfect cone (first Voronoi)
toroidal compactification of $\cA_6$
at the level of divisors
and show that only three boundary divisors of $\hhu$, namely
$D_{E_6}, D_{\mathrm{syz}}$ and $D_{\mathrm{azy}}$ are not contracted
by the map PT (Theorem \ref{contractions}). After proving Theorem
\ref{prym6} in Section \ref{sec:PvPTK}, we complete in Section \ref{sec:WP} the proof of
Theorem \ref{thmmain2} after a detailed study of the divisors
$D_{\mathrm{azy}}$ and $D_{\mathrm{syz}}$ of azygetic and syzygetic
$E_6$-covers respectively on a partial compactification
$\mathcal{G}_{E_6}$ of $\Hur$. The ramification divisor of the
Prym-Tyurin map is described in Section \ref{sec:ramif}. Finally, in Section \ref{sec:Petri}, we
prove by degeneration a Petri type theorem on $\Hur$.

\vskip 4pt

\noindent {\bf Acknowledgments:} We owe a great debt to the work of
Vassil Kanev, who first constructed the Prym-Tyurin map $PT$ and
raised the possibility of uniformizing $\cA_6$ in this way.
The authors acknowledge partial support by the NSF:
VA under grant DMS 1200726,
RD under grant DMS 1603526, EI under grant DMS-1103938/1430600.
The work of GF and AO has been partially supported by the DFG Sonderforschungsbereich 647 ``Raum-Zeit-Materie".

\section{Kanev's construction and Prym-Tyurin varieties of $E_6$-type}
\label{sec:Kanev-construction}

Consider a cubic threefold $X\subset\bP^4$ and a smooth hyperplane
section $S\subset X$. The cubic surface $S$ contains a set of
$27$ lines $\Lambda:=\{\ell_s\}_{1\leq s\leq 27}$
forming a famous classical
configuration, which we shall review below in Section~\ref{sec:e6-lattice}.
Consider the lattice $\bZ^{\Lambda}=\bZ^{27}$ with the standard basis corresponding to
$\ell_s$'s, and let $\deg\colon \bZ^{\Lambda}\to \bZ$ be the degree
homomorphism, so that $\deg(\ell_s) =1$ for all $s=1, \ldots, 27$.

\begin{num}\label{subsec-D_L} By assigning to  each line $\ell_s$ the sum $\sum_{\{s':\ \ell_s\cdot \ell_{s'}=1\}} \ell_{s'}$ of
the $10$ lines on $S$ intersecting $\ell_s$, we define a homomorphism $D'_{\Lambda}\colon
\bZ^{27}\to \bZ^{27}$ of degree $10$. It is easy to check that $D'_{\Lambda}$
satisfies the following quadratic equation:
\begin{displaymath}
  (D'_{\Lambda}+5)(D'_{\Lambda}-1) = 5 \left( \sum_{s=1}^{27} \ell_s \right) \cdot \deg
\end{displaymath}
The restriction $D_{\Lambda}$ of $D'_{\Lambda}$ to the subgroup
$\mbox{Ker}(\deg)$ satisfies the equation $(D_{\Lambda}+5)(D_{\Lambda}-1)=0$.

Consider a generic pencil $\{S_t \}_{t\in\bP^1}$ of cubic
hyperplane sections of $X$. This defines:
\begin{itemize}
\item a degree $27$ smooth curve cover $\pi\colon C\to\bP^1$; the points in
  the fiber $\pi\inv(t)$ correspond to the lines lying on $S_t$;
\item a symmetric \emph{incidence} correspondence $\wt D\subset C\times C$. Let $p_i\colon \wt D\to C$ denote the two projections. Then
$\wt D$ has degree $\mbox{deg}(p_1)=\mbox{deg}(p_2)=10$;
\item a homomorphism $D'=p_{2*}\circ p_1^* \colon \Pic(C)\to \Pic(C)$
satisfying the following quadratic equation (see also \cite{kanev1989spectral-curves}):

  \begin{math}
    \ \ (D'+5)(D'-1) = 5 \pi\inv(0) \cdot \deg
  \end{math};
\item the restriction $D$ of $D'$ to $JC=\Pic^0(C)$, satisfying
  $(D+5)(D-1)=~0$.
\end{itemize}

For a generic such pencil the map $\pi\colon
C\to \bP^1$ has $24$ branch points on $\bP^1$, corresponding to singular
cubic surfaces in the pencil, each with one node. Over each of
the $24$ points, the fibre consists of $6$ points of multiplicity two and
$15$ single points. By the Riemann-Hurwitz formula, we compute $g(C)=46$.
\end{num}

\begin{num} We refer to \cite{kanev1989spectral-curves,
  lange2008a-galois-theoretic-approach} for the following facts.  The
cover $\pi\colon C\to\bP^1$ is not Galois. The Galois group of its Galois closure
is $W(E_6)$, the reflection group of
the $E_6$ lattice. As we shall review in Section~\ref{sec:e6-lattice}, the
lattice $E_6$ appears as the lattice $K_S^\perp \subset\Pic (S)$.  The
27 lines can be identified with the $W(E_6)$-orbit of the fundamental
weight $\omega_6$, and one has a natural embedding $W(E_6)\subset
S_{27}$.  The intermediate non-Galois cover $C\to\bP^1$ is associated
with the stabilizer subgroup of $\omega_6$ in $W(E_6)$, that is, with the
subgroup $W(E_6)\cap S_{26} \cong W(D_5)$.
\end{num}

\begin{num} By Riemann's Existence Theorem, a $27$-sheeted cover $C\to\bP^1$ ramified over 24
points is defined by a choice of $24$ elements $w_i\in S_{27}$
satisfying $w_1 \cdots w_{24}=1$. For a cover coming from a pencil
of cubic surfaces, each $w_i\in W(E_6)$ is a reflection in a root of
the $E_6$. It is a \emph{double-six}, that is, viewed as an element of $S_{27}$, it is a product of $6$ disjoint transpositions.
\end{num}

\begin{definition}
Let $\Hur$ be the Hurwitz space parametrizing irreducible smooth
Galois $W(E_6)$-covers $\tC\to\bP^1$ ramified in 24 points, such that
the monodromy over each point is a reflection in a root of the $E_6$ lattice.
\end{definition}

\begin{num} Note that points in the space $\Hur$ correspond to covers where we do not choose a labeling of the branch points. The data for the cover $\tC$ consists of the
branch divisor $p_1+\ldots+p_{24}$ on $\bP^1$, and, for each of
these points, the monodromy $w_i\in W(E_6)$ given by a reflection in a root, once a base point $p_0\in \mathbb P^1$ and a system of arcs
$\gamma_i$ in $\pi_1(\bP^1\setminus \{p_1,\dotsc,p_{24}\}, p_0)$ with
$\gamma_1\cdots \gamma_{24}=1$ has been  chosen. The elements $\{w_i\}_{i=1}^{24}$
generate $W(E_6)$ and satisfy the relation $w_1\cdots w_{24}=1$.
The monodromy data being finite, the space $\Hur$ comes with a
finite unramified cover $$\mathfrak{br}: \Hur\to \cM_{0,24}/S_{24}$$ to the moduli space of $24$
unordered points on $\bP^1$. Thus $\dim(\Hur) =21$.
An important fact about this space is the following result of Kanev \cite{kanev2006hurwitz-spaces}:
\end{num}

\begin{theorem}
For any irreducible root system $R$, the Hurwitz scheme
parameterizing Galois $W(R)$-covers such that the monodromy around any
branch point is a reflection in $W(R)$, is irreducible.
\end{theorem}

\begin{num} In particular, the space $\Hur$ is irreducible. If $[\widetilde{\pi}:\tC\to\bP^1]\in \Hur$, let $\pi:C\to\bP^1$ be an intermediate
non-Galois cover of degree $27$, that is, the quotient of $\tC$ by a subgroup $W(E_6)\cap
S_{26} \cong W(D_5)$ in $S_{27}$. Since $W(E_6)$ acts transitively on the set
$\{1,\dotsc,27\}$, the $27$ subgroups $S_{26}\subset S_{27}$ are
conjugate, and the corresponding curves $C$ are isomorphic. Thus,
$\Hur$ is also a coarse moduli space for degree $27$ non-Galois covers $\pi\colon
C\to\bP^1$, branched over $24$ points such that the monodromy at each branch point is a reflection of $W(E_6)$.
\end{num}

\begin{num} Let $\pi\colon C\to\bP^1$ be an $E_6$-cover as above. Each fiber of $\pi$ can be
identified consistently with the set of $27$ lines on a cubic surface.
The incidence of lines, in the same way as for the correspondence $D_{\Lambda}$ in \ref{subsec-D_L},
induces a symmetric correspondence $\widetilde D\subset C\times C$ of degree
$10$, which is disjoint from the diagonal $\Delta\subset C\times C$.  In turn, $\widetilde D$ induces a homomorphism $D'\colon \Pic(C) \to \Pic(C)$,
whose restriction $D:JC\rightarrow JC$ to the degree zero part $JC:=\mbox{Pic}^0(C)$ satisfies the  quadratic relation
\begin{equation}\label{quadkanev}
(D-1)(D+5)=0\in \mbox{End}(JC).
\end{equation}

\begin{definition}
The \emph{Prym-Tyurin-Kanev (PTK) variety} $PT(C,D)$ is defined as the connected
component of the identity $PT(C,D):=\big(\mbox{Ker} (D+5)\big)^0 = \mbox{Im}
(D-1)\subset JC$.
\end{definition}
\end{num}

\begin{num} Using \cite{Kan87}, Equation \eqref{quadkanev} implies that the restriction of the principal polarization $\Theta_C$ of $JC$ to
$PT(C,D)$ is a multiple of a principal polarization. Precisely, $\Theta_{C| PT(C,D)}=6\cdot \Xi$, where $(PT(C,D), \Xi)$ is a ppav.
Since
\[
0=\widetilde{D} \cdot \Delta=2\mbox{deg}(\widetilde{D})-2\mbox{tr}\bigl\{D:H^0(C,\omega_C)\rightarrow H^0(C,\omega_C)\bigr\},
\]
we obtain that
\begin{equation}\label{dimpt1}
\mbox{dim } PT(C,D)=\frac{1}{6}\Bigl(g(C)-\mbox{deg}(\widetilde{D})\Bigr)=\frac{1}{6}(46-10)=6,
\end{equation}
see also \cite[Proposition 5.3]{lange2008a-galois-theoretic-approach}. We have the morphism of moduli stacks
\[
\begin{array}{cccc}
PT \colon & \Hur & \lra & \cA_6 \\
& [C,D] & \longmapsto & [PT(C,D), \Xi].
\end{array}
\]
Both stacks are irreducible and $21$-dimensional.
The main result of
this paper (Theorem~\ref{thmmain1}) is that $PT$ is a dominant, i.e., generically finite, map.
\end{num}

\begin{num} Our main concrete examples of $E_6$-covers of $\bP^1$ are the \emph{curves of lines} in Lefschetz pencils of cubic surfaces.
The subvariety $\cT\subset \Hur$ corresponding to pencils $\{S_t\}_{t\in \mathbb P^1}$ of hyperplane
sections of cubic $3$-folds $X\subset\bP^4$ has expected dimension
\begin{displaymath}
  {7 \choose 3}-1 + \dim\operatorname{Gr}(2,5) - \dim\PGL_5 = (35-1) + 6 - (25-1) = 16.
\end{displaymath}
\end{num}

\begin{num} We now describe the restriction of the map $PT$ to the locus $\cT\subset \Hur$ parametrizing such covers. Let $V$ be a $5$-dimensional vector space over
$\bC$ whose projectivization contains $X$ and let $F\in \Sym^3(V^{\vee})$ be a defining equation for $X$. Denote by $\cF:=\cF(X)$ the Fano variety of lines in
$X$. Let $JX := H^{2,1}(X)^{\vee} / H_3 (X, \bZ)$ be the intermediate Jacobian of $X$. It is well known  \cite{ClemensGriffiths} that the Abel-Jacobi map defines
an isomorphism $JX \cong \Alb \cF$, where $\Alb \cF$ is the Albanese variety of $\cF$. Let $\Lambda$ be a Lefschetz pencil of hyperplane sections of $X$ and
 denote by $E$ its base curve. The curve $C$ classifying the lines lying on the surfaces contained in $\Lambda$ lives naturally in $\cF$. The map sending a line
 to its point of intersection with $E$ induces a degree $6$ cover $C\rightarrow E$. Furthermore, the choice of a base point of $C$ defines a map $C\ra JX$. So
 we obtain a well-defined induced map $JC\ra E\times JX$. The transpose $E \times \Pic^0(\cF) = E \times JX \ra JC$ of this map is given by pull-back on divisors
 on each of the factors, using the map $C\ra E$ and the embedding $C\inj \cF$ respectively. On the locus $\cT$ we can explicitly determine the PTK
 variety:

\begin{lemma}\label{lemP(C)}
The map $JC\ra E\times JX$ (or its transpose $E \times JX \ra JC$) induces an isomorphism of ppav $PT(C,D)\stackrel{\cong}\rightarrow E\times JX$.
\end{lemma}

\begin{proof}
We first show that the correspondence $D$ restricts to multiplication by $(-5)$ on both factors $E$ and $JX$. For $\ell\in C$, let $\widetilde{D}(\ell)$ be the sum of
the lines incident to $\ell$ and $E$ inside $X$. We denote by $H_{\ell}$ the hyperplane spanned by $E$ and $\ell$ and put $S_{\ell} := H_{\ell} \cap X$. The lines
incident to $E$ and $\ell$  form $5$ pairs $(\ell_1, \ell_1'), \ldots , (\ell_5, \ell_5')$, with $\ell+ \ell_i + \ell_i'\in |-K_{S_{\ell}}|$ for $i=1, \ldots, 5$.

Consider first the intermediate Jacobian $JX$. We have
\[
\widetilde{D}(\ell) = \sum_{i=1}^5 (\ell_i + \ell_i') \equiv 5 |-K_{S_{\ell}}| - 5\ell,
\]
where $\equiv$ denotes linear equivalence in $S_{\ell}$. Since $|-K_{S_{\ell}}|$ is constant as $\ell$ varies, it follows that $D$ restricts to multiplication
by $(-5)$ on $JX$.

Consider the elliptic curve $E$. Then $\widetilde{D}(\ell)$ in $E$ is the sum of the intersection points of $\ell_i, \ell_i'$ with $E$. Note that $(\ell + \ell_i +\ell_i' )|_E$ is
also the intersection of the plane $\Pi_i:=\langle \ell, \ell_i, \ell_i' \rangle$ with $E$. Hence $\sum_{i=1}^5 (\ell + \ell_i +\ell_i' ) |_E$ is the intersection of the $5$
planes $\Pi_1 , \ldots , \Pi_5$ with $E$. Projecting from $\ell$, we see that the union of these planes is the intersection of $H_{\ell}$ with the inverse image $Q$ of
the plane quintic in $\bP^2 = \bP (V /\ell)$ parametrizing singular conics (the discriminant curve for the projection of $X$ from $\ell$). Therefore $\sum_{i=1}^5
 (\ell + \ell_i + \ell_i' ) |_E$ is contained in the intersection $Q \cap E$ and since the two divisors have the same degree, we obtain that $\sum_{i=1}^5 (\ell + \ell_i
 + \ell_i' ) |_E = Q \cap E$ is constant. This implies that $D$ is multiplication by $(-5)$ on $E$ as well.

So the PTK variety is isogenous to $E\times JX$. To show that they are isomorphic, we show that the pull-back of the polarization of $JC$ to $E\times JX$ is $6$ times a principal polarization. This is immediate on the factor $E$, since the map $C \ra E$ has degree $6$. To see it on the $JX$ factor as well, we again use the Abel-Jacobi embedding $C\inj\cF\inj JX$ and recall the fact \cite{ClemensGriffiths} that one model of the theta divisor in $JX$ is the image of the degree $6$ difference map $\psi:\cF \times \cF \ra \Alb \cF = JX$, defined by $\psi(\ell,\ell')=\ell-\ell'$.
\end{proof}

We denote by $\mathcal{IJ}_5$ the closure in $\cA_5$ of the moduli space of intermediate Jacobians of cubic threefolds.
We have the following result:

\begin{corollary}\label{pencils2}
We have the following equality of $11$-dimensional irreducible cycles in $\cA_6$:
$$\overline{{PT}(\cT)}=\mathcal{IJ}_5\times \cA_1\subset \cA_5\times \cA_1\subset \cA_6,$$
where the closure on the left hand side is taken inside $\cA_6$.
\end{corollary}

\end{num}

\section{The $E_6$ lattice}
\label{sec:e6-lattice}

In this section we recall basic facts about the $E_6$ lattice.
Our reference for these is \cite[Chapters 8,9]{dolgachev2012classical-algebraic}.
\begin{num} Let $I^{1,6}$ be the standard Lorenzian lattice with the quadratic
form $x_0^2-\sum_{i=1}^6 x_i^2$. The negative definite $E_6$ lattice is
identified with $k^\perp$, where $k=(-3,1,\dotsc,1)$. Its dual $E_6^{\vee}$
is identified with $I^{1,6}/\bZ k$. Let us denote the standard basis
of $I^{1,6}$ by $f_0,f_1,\dotsc,f_6$, to avoid confusion with the
edges $e_i$ in a graph.

The roots of $E_6$ are the vectors with square $-2$. There are ${6
  \choose 2}+ {6\choose 3}+1 = 36$ pairs of roots corresponding to
$\alpha_{ij}=f_i-f_j$, $\alpha_{ijk}=f_0-f_i-f_j-f_k$ and
$\alpha_{\max}=2f_0 - f_1-\dotsc -f_6$. Obviously, if $r\in E_6$ is a root then $-r$ is a root as well.  The simple roots,
corresponding to the $E_6$ Dynkin diagram can be chosen to be
$r_1=\alpha_{123}$, $r_2=\alpha_{12}$, $r_3=\alpha_{23}$, $r_4=\alpha_{34}$,
$r_5=\alpha_{45}$ and $r_6=\alpha_{56}$.
\end{num}

\begin{num} The Weyl group $W(E_6)$ is the group generated by the reflections in the
roots. It has 51,840 elements.
The fundamental weights
$\omega_1,\dotsc,\omega_6$ are the vectors in $E_6^{\vee}$ with
$(r_i,\omega_j)=\delta_{ij}$.

The exceptional vectors are the vectors in the $W(E_6)$-orbit of
$\omega_6$. They can be identified with vectors $\ell$ in $I^{1,6}$ satisfying
$\ell^2=k\ell=-1$.
There are $6+6+15=27$ of them, namely:
\begin{align*}
  &a_i = f_i, \ \mbox{ for } \quad i=1,\dotsc,6; \\
  &b_i = 2f_0-f_1-\dots-f_6+f_i, \ \mbox{ for }\quad i=1,\dotsc,6; \\
  &c_{ij}= f_0-f_i-f_j, \  \mbox{ for } \quad 1\le i<j \le 6.
\end{align*}
\end{num}

\begin{num} For each root $r\in E_6$, there are $15$ exceptional vectors that are
orthogonal to it, $6$ exceptional vectors with $r\cdot \ell=1$ and $6$ vectors with
$r\cdot \ell=-1$. The collections of the $6$ pairs of exceptional vectors
non-orthogonal to a root vector are called \emph{double-sixes}. The
elements in each pair are exchanged by the reflection $w_r\in W(E_6)$ in the
root $r$.

There are $36$ double-sixes, one for each pair $\pm r$ of roots. For
example, the double-six for the root $r=\alpha_{\max}$ is
$\{a_1,a_2,\dotsc,a_6\}$, $\{b_1,b_2,\dotsc,b_6\}$. The reflection group acts transitively
on the set of the exceptional
vectors. This gives rise to an embedding $W(E_6)\subset S_{27}$. Under this
embedding, each reflection corresponds to a product of 6
transpositions. For example, the reflection in the root
$r=\alpha_{\max}$ is the permutation $(a_1, b_1)\cdots (a_6, b_6)\in S_{27}$.

Note that the choice of a root is equivalent to an ordering of a pair: when we write the same element of
$W(E_6)$ as a product $(b_1, a_1)\cdots (b_6, a_6)$, it corresponds to the root
$-\alpha_{\max}$.
The $W(E_6)$-action by conjugation is transitive on the set of reflections, i.e., double
sixes, so to study their properties it is usually sufficient to make
computations for one representative.
\end{num}

\begin{num} For a smooth cubic surface $S$, the above objects have the following incarnation:
\begin{itemize}
\item $I^{1,6} = \Pic(S)$ together with the intersection form,
\item $k= K_S$ and $E_6= K_S^\perp\subset \mbox{Pic}(S)$,
\item the exceptional vectors are identified with the lines
  $\ell_1,\dotsc, \ell_{27}$ on $S$,
\item a sixer is a set of 6 mutually disjoint lines, a double-six is
  the set of two sixers corresponding to the opposite roots.
\end{itemize}
\end{num}

The relationship between the $W(E_6)$-action and
the correspondence given by the line incidence is as follows.

\begin{definition}
The correspondence on the set of exceptional vectors is defined by setting
$$D(\ell) := \sum_{\{\ell':\ \ell'\cdot \ell=1\}} \ell'.$$
\end{definition}

\begin{remark}\label{num:corr-computation}
For further use, we retain the following computation:
\begin{align*}
  &D(a_1) = b_2+\dotsb + b_6 + c_{12}+\dotsb +c_{16} \\
  &D(b_1) = a_2+\dotsb + a_6 + c_{12}+\dotsb +c_{16} \\
  &D(a_1-b_1) = (b_2-a_2) + \dotsc (b_6-a_6).
\end{align*}
\end{remark}

\begin{num} The group $W(E_6)$ has $25$ irreducible representations corresponding to its $25$ conjugacy classes, which will  appear several times
in this paper. For conjugacy classes we use the ATLAS or GAP notation 1a, 2a, 2b, 2c, \ldots, 12a, (command `CharacterTable("U4(2).2")'). The
number refers to the order of the elements in the conjugacy class.
For instance, the reflections in $W(E_6)$ (products of
six transpositions) belong to the conjugacy class 2c, the product of two syzygetic reflections belongs
to the class 2b, whereas the product of two azygetic reflections belongs to the class 3b (see Section 5 for precise definitions).
\end{num}

\section{Degenerations of Jacobians and Prym varieties}
\label{sec:degs-jacobians}

\begin{num} By a theorem of Namikawa and Mumford, the classical Torelli map
$\cM_g\to \cA_g$ sending a smooth curve to its Jacobian extends to a
regular morphism $\ocM_g \to \cA_g\vor$ from the
Deligne-Mumford compactification of $\cM_g$ to the toroidal
compactification of $\cA_g$ for the second Voronoi fan. See
\cite{alexeev2011extending-torelli} for a transparent modern treatment
of this result, and extension results for other toroidal
compactifications of $\cA_g$. The result applies equally to the stacks
and to their coarse moduli spaces. Here, we will work with stacks, so that we have
universal families over them.
\end{num}

\begin{num}\label{subsec:homology-gps}
At the heart of the result of Namikawa and Mumford lies the Picard-Lefschetz formula for the
  monodromy of Jacobians in a family of curves, see
  e.g. \cite[Proposition 5]{namikawa1973on-the-canonical-holomorphic}.
  The map of fans for the toroidal morphism $\ocM_g \to \cA_g\vor$ is
  described as follows.  Fix a stable curve $[C]\in\ocM_g$, and let
  $\Gamma$ be its dual graph, with a chosen orientation. Degenerations
  of Jacobians are described in terms of the groups
\begin{displaymath}
  C_0(\Gamma,\bZ) = \bigoplus_{\text{vertices } v} \:\: \bZ v, \quad
  C_1(\Gamma,\bZ) = \bigoplus_{\text{edges } e} \:\: \bZ e  ,\quad
  H_1(\Gamma,\bZ) = \mbox{Ker} \bigl\{\partial\colon C_1(\Gamma, \bZ) \to C_0(\Gamma, \bZ)\bigr\}.
\end{displaymath}

The Jacobian $JC=\Pic^0(C)$ is a semiabelian group variety, that is, an extension
\begin{equation}\label{eqext}
  1\to H^1(\Gamma,\bC^*)
  \to \mbox{Pic}^0(C) \to \mbox{Pic}^0(\wt C) \to 0,
\end{equation}
where  $\wt C$  is the normalization of $C$.
In particular, $\Pic^0(C)$ is a multiplicative torus if and only if  $\wt C$ is a union
of $\bP^1$'s, or equivalently, if $b_1=h^1(\Gamma)=g$.


The monodromy of a degenerating family of
Jacobians is described as follows.  Fix a lattice
$\Lambda\simeq\bZ^g$ and a surjection $\Lambda \twoheadrightarrow
 H_1(\Gamma,\bZ)$. The rational polyhedral cone for a neighborhood of $[C]\in\ocM_g$
lives in the space $\Lambda^{\vee}\otimes\bR$ with the lattice
$\Lambda^{\vee}$. It is a simplicial cone of dimension $b_1=h^1(\Gamma)$
with the rays $e_i^*$ corresponding to the edges of $\Gamma$. Here,
$e_i^*$ is the linear function on $H_1(\Gamma,\bZ)\subset
C_1(\Gamma,\bZ)$ taking the value $\delta_{ij}$ on the edge $e_j\in
C_1(\Gamma,\bZ)$.

\vskip 4pt

The rational polyhedral cone corresponding to a neighborhood of $[JC]\in\cA_g\vor$
lives in the space $\Gamma^2(\Lambda^{\vee})\otimes\bR =
(\Sym^2(\Lambda)\otimes\bR)^{\vee}$, where the lattice $\Gamma^2(\Lambda^{\vee})$
is the second divided power of $\Lambda^{\vee}$. It is a simplicial cone with
the rays $(e_i^*)^2$ for all $e_i^*\ne0$, which means that $e_i$ is
not a bridge of the graph $\Gamma$.  We explain what this means in down to earth terms.
In an open analytic neighborhood
$U$ of $[C]$, one can choose local analytic coordinates $z_1,\dotsc,
z_{3g-3}$ so that the first $N$ coordinates correspond to smoothing
the nodes of $C$, labeled by the edges $e_i$ of the graph
$\Gamma$. Thus, we have a family of smooth curves over the open subset
$V=U-\bigcup_{i=1}^N\{z_i=0\}$.

\vskip 4pt

Then a complex-analytic map $V\to \mathcal H_g$ to the Siegel upper
half-plane is given by a formula (see
\cite[Thm.2]{namikawa1973on-the-canonical-holomorphic} or
\cite[18.7]{namikawa1976a-new-compactification-of-the-siegel2})
\begin{displaymath}\label{for-maptoSiegel}
  (z_i) \mapsto \sum_{i=1}^N M_i \cdot
  \frac{1}{2\pi\sqrt{-1}} \log z_i +
  \text{(a bounded holomorphic function)},
\end{displaymath}
where $M_i$ are the $g\times g$ matrices with integral coefficients corresponding to the
quadratic functions $(e_i^*)^2$ on $\Lambda\twoheadrightarrow
 H_1(\Gamma,\bZ)$.
After applying the coordinatewise exponential map
\begin{displaymath}
  \bC^{\frac{g(g+1)}2} \to (\bC^*)^{\frac{g(g+1)}2}, \quad
  u_{ij} \mapsto \exp(2\pi \sqrt{-1}\ u_{ij}),
\end{displaymath}
the matrices $M_i\cdot (\log z_i/2\pi\sqrt{-1})$ become Laurent
monomials in $z_i$. This monomial map describes a complex-analytic map
from a small complex-analytic neighborhood $U$ of $[C]\subset\ocM_g$
to an appropriate \'etale neighborhood of $\cA_g$. For the arguments
below the above two formulas suffice.  In particular, we do not need
to know the indeterminacy locus of the extended maps. Thus, we will
not need explicit coordinates near a boundary of $\agvor$.
\end{num}

\begin{num} The following weak form of Torelli's theorem is a sample of
our degeneration technique. This is far from being the easiest
way to prove the Torelli theorem, but it gives a good illustration of our
method which we later apply to PTK varieties.

\begin{lemma}\label{lem:torelli-full-dim}
  The image of the Torelli map $\cM_g\to \cA_g$ has full dimension $3g-3$.
\end{lemma}
\begin{proof}
  For every $g$, there exists a $3$-edge connected trivalent graph $\Gamma$
  of genus $g$ (exercise in graph theory). By Euler's formula, it has
  $3g-3$ edges. Recall that a connected graph is \emph{2-edge connected}
  if it has no bridges, i.e., the linear functions $e_i^*$ on
  $H_1(\Gamma,\bZ)$ are all nonzero, and it is \emph{3-edge connected} if
  for $i\ne j$ one has $e_i^*\ne \pm e_j^*$,
  i.e., $(e_i^*)^2\ne(e_j^*)^2$.

\vskip 3pt

  Let $C$ be a stable curve whose dual graph is $\Gamma$ and
  whose normalization is a disjoint union of $\bP^1$'s. Then the $3g-3$
  matrices $M_i$ in Formula \eqref{for-maptoSiegel}, i.e. the functions
  $(e_i^*)^2$, are linearly independent in $\Sym^2(\bZ^g)$,
  cf. \cite[Remark 3.6]{alexeev2011extending-torelli}.  By looking at the
  leading terms as $z_i\to 0$, this easily implies that the image
  has full dimension $3g-3$.

After applying the exponential function, the map becomes
  \begin{displaymath}
    (z_1,\dotsc,z_{3g-3}) \mapsto
    \text{(monomial map)} \times \text{(invertible function)}.
  \end{displaymath}
  Since the monomial part is given by monomials generating an algebra
  of transcendence degree $3g-3$, the image is full-dimensional.
\end{proof}

\begin{remark}\label{rem:regularity-not-needed}
  Note that the regularity of the extended Torelli map
  $\ocM_g\to\cA_g\vor$ played no role in the proof of
  Lemma \ref{lem:torelli-full-dim}. All we need for the conclusion is the fact that the monodromy
  matrices $M_i$ are linearly independent.
\end{remark}
\end{num}

\begin{num} The theory for Jacobians was extended to the case of Prym varieties in
\cite{alexeev2002degenerations-of-prym}. We briefly recall it. Let
$\overline{\cR}_g$ be the stack of Prym curves of genus $g$, classifying admissible pairs $[C, \iota]$ consisting of a stable curve with
involution $\iota:C\rightarrow C$, so that $C/\iota$ is a stable curve of genus $g$ and the map $C\rightarrow C/\iota$ is an admissible map of
stable curves. We refer to \cite{BeauvilleSchottkyPrym} and \cite{FL} for background on $\overline{\cR}_g$.   Consider one pair $[C,\iota]\in\overline{\cR}_g$ and
a small analytic neighborhood $U$ of it. As before, $\Gamma$ is
the dual graph of $C$.

Then the space $H_1(C,\bZ)$ of the Jabobian case is replaced by the
lattice $H_1 / H_1^{+}$. Here,
$H_1^+$ and $H_1^-$ are the $(+1)$- and the $(-1)$-eigenspaces of the
involution action $\iota_*$ on $H_1(C,\bZ)$ respectively. Via the natural projection $H_1 \twoheadrightarrow
 H_1 / H_1^{+}$, we identify $H_1^-$ with a finite index sublattice of $H_1 / H_1^{+}$.

The degeneration
of Prym varieties as groups is
\begin{displaymath}
  P(C,\iota) = \mbox{Ker}(1+\iota^*)^0 = \mbox{Im}(1-\iota^*),
  \qquad \iota^*\colon \Pic^0(C)\to\Pic^0(C).
\end{displaymath}

The monodromy of a degenerating family of
Prym varieties is obtained by restricting the
monodromy map for $JC$ to the $(-1)$-eigenspace. Combinatorially, it
works as follows: For every edge $e_i$ of $\Gamma$ we have a linear function $e_i^*$ on
the group $H_1^{-}$, the restriction of the linear function on $H_1(C,\bZ)$.
For the divisor $\{z_i=0\}$ on $U$ corresponding to smoothing the
node $P_i$ of $C$, the monodromy is given by the quadratic form
$(e_i^*)^2$ restricted to $H_1(\Gamma,\bZ)^-$.
Similarly to Lemma~\ref{lem:torelli-full-dim}, this can be used to
prove various facts about the Prym-Torelli map, but we will not pursue
it here.
\end{num}

\section{Degenerations of Prym-Tyurin-Kanev varieties}\label{sec:degPTK}

We choose a concrete boundary point in a compactification of the Hurwitz
scheme $\Hur$.  We start with a single cubic surface $S$ and the set
$\{\ell_1,\dotsc,\ell_{27}\}$ of 27 lines on it. Sometimes we shall use
the Schl\"afli notation $\{a_i,b_i,c_{ij}\}$ for them, as in Section \ref{sec:e6-lattice}. We fix an embedding of $W(E_6)$ into the symmetric group
$S_{27}$ permuting the $27$ lines on $S$.

\begin{num}
We choose $12$ roots $r_i$ which generate the root system
$E_6$. Let $w_i\in W(E_6)$ be the reflections in $r_i$; they
generate $W(E_6)$.  As we saw in Section \ref{sec:e6-lattice}, each $w_i$ is a
double-six. Fixing the root $r_i$ gives it an orientation.
\end{num}

\begin{num}
Consider a nodal genus $0$ curve $E$ whose normalization is a union of
$\bP^1$'s and whose dual graph is the tree $T$ shown in the left half
of Figure~\ref{fig:cover}.  The $24$ ends of this tree correspond to
$24$ points $p_1,\dotsc, p_{24}$ on $E$. We label the points by roots
$r_1,\dotsc,r_{12}$.  Each of the outside vertices has two ends, we
use the same label $r_i$ for both of them.
\end{num}

\begin{figure}[h]
 \centering
  \includegraphics[width=5in]{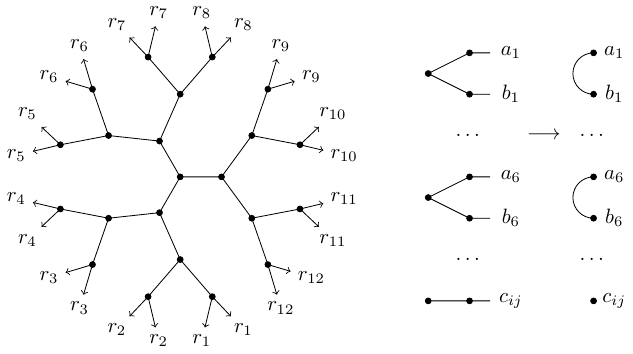}
  \caption{The tree $T$ for the target curve $E$ of genus 0}
  \label{fig:cover}
\end{figure}

\begin{definition}
Let $\pi\colon C\to E$ be an admissible $27:1$ cover ramified at
  the point $p_i$ with monodromy $w_i$ for $i=1, \ldots, 24$.
\end{definition}
For every irreducible component of $E$, the product of the monodromy elements
equals~$1$; this count includes the nodes. Since we required that for
every component on the boundary the two $w_i$'s are the same, the map
is unramified at the nodes. Thus, $\pi$ is \'etale over
$E\setminus\{p_1,\dotsc,p_{24}\}$.

\begin{num}
  Here is a concrete description of the dual graph $\Gamma$ of $C$. It
  has
\begin{displaymath}
    10 \times 27 + 12 \times (6+15) \text{ vertices and }
    21\times 27 \text{ edges}
\end{displaymath}
Each vertex $v$ of $T$ in the \emph{\'etale part} has $27$ vertices
over it. Over each of the outside $12$
vertices, there are $6$ vertices, where the map $\bP^1\to\bP^1$ is
$2:1$ and ramified at a pair of the points $p_i$ and $p_{i+12}$, and $15$ other vertices
where the map $\bP^1\to\bP^1$ is $1:1$.

  All the nodes of $E$ lie in the \'etale part, so for each internal
  edge $e$ of the tree $T$ there are $27$ edges of $\Gamma$.
  \end{num}

\begin{num}\label{num:G-prime}
The graph $\Gamma$ is homotopically  equivalent to the following much
simpler graph $\Gamma'$. It has:
\begin{enumerate}
\item $27$ vertices $\{v_s\}_{s=1}^{27}$, labeled by the lines on $S$. (Here, $s$ stands for ``sheets''.)
\item $12\times 6$ edges $e_{ij}$.  For each of the twelve roots
$r_i$, there are $6$ edges. For example, for $r=r_{\max}$, the edges
are $(a_1,b_1), \dotsc, (a_6,b_6)$. The first edge is directed
from $a_1$ to $b_1$, etc.
\end{enumerate} The graph $\Gamma'$ is obtained from $\Gamma$ by contracting the
tree in each sheet to a point, and removing the middle vertex of
degree $2$ for each of the $12\times 6$ paths corresponding to the
double-sixes. The process is illustrated in the right half of
Figure~\ref{fig:cover}.

By Euler's formula, the genus of $\Gamma$ is $12\times
6-27+1=46$. Thus, the curve $C$ has arithmetic genus $46$.
\end{num}

\begin{num}
Next we define a symmetric correspondence $\widetilde D\subset C\times C$ of
degree $10$, as follows. To each point $Q\in C$ over the
\'etale part in the sheet labeled $\ell_i$, associate 10 points in the
same fiber of $\pi$ that are labeled $\ell_{ij}$ by the lines that
intersect $\ell_i$.

This defines the curve $\wt D^0\subset C^0\times C^0$, where
$C^0=C\setminus \pi^{-1}\{p_1,\dotsc, p_{24}\}$.
The correspondence $\wt D\subset C\times C$ is the closure of $\wt
D^0$.
Let $p_i$ be a ramification point with
monodromy $w_i$. Without loss of generality, we may assume
$w=w_{\max}$. The points in the fiber $\pi\inv(p_i)$ are labeled
$a_1b_1$, \dots, $a_6b_6$ and  $c_{ij}$ for $i\neq j$. Then the correspondence is described by:
\begin{displaymath}
  a_1b_1 \mapsto \sum_{i=2}^6 (a_ib_i + c_{1i} ), \quad
  c_{12} \mapsto a_1b_1 + a_2b_2 + \sum_{i,j\ne 1,2} c_{ij},
  \quad\text{etc.}
\end{displaymath}

\end{num}

\begin{lemma}
There exists an analytic neighborhood $U\subset \overline{\cM}_{0,24}$ of
the point $[E, p_1,\dots, p_{24}]$ and a family of covers
$\pi_t\colon C_t\to E_t$ together with correspondences $\wt
D_t\subset C_t\times C_t$ over $U$, which extends $\pi\colon C\to E$
and $\wt D$.
\end{lemma}
\begin{proof}
Since the map $\pi$ is \'etale over each node of $E$, the families
$C_t$ and $\wt D_t$ extend naturally. The monodromy data determine
the $C_t$'s as topological spaces. Then the finite map $C_t\to E_t$
determines a unique structure of an algebraic curve on $C_t$.
\end{proof}

\begin{lemma}
  The correspondence $\wt D\subset C\times C$ induces an endomorphism
  of the homology group $D\colon H_1(\Gamma,\bZ)\to H_1(\Gamma,\bZ)$
  satisfying the relation $(D-1)(D+5)=0$. The $(-5)$-eigenspace
  $H_1^{(-5)}$ can be naturally identified with $\mathrm{Ker}(\phi)$, where
  \begin{displaymath}
    \phi\colon \bigoplus_{i=1}^{12} \bZ R_i \to E_6,
    \qquad R_i \mapsto r_i.
  \end{displaymath}
  Here, $R_i$ is a basis vector for the $(-5)$-eigenspace for the action of $D$ on the rank $6$ lattice generated by the edges of $\Gamma'$ above the root $r_i$. Since the vectors $r_i$ generate $E_6$, one
  has $\rk H_1^{(-5)}=6$.
\end{lemma}
\begin{proof}
  We will work with the graph $\Gamma'$ defined in~\ref{num:G-prime},
  since the homology groups of $\Gamma$ and $\Gamma'$ are canonically
  identified.  The group $C_0(\Gamma',\bZ)$ of vertices is
  $\bigoplus_{i=1}^{27}\bZ v_i$. The endomorphism $D^0$ on it is defined
  in the same way as the correspondence on the $27$ lines.  The induced
  endomorphism $D^1$ on $C_1(\Gamma',\bZ)$ is the following. Pick
  one of the roots $r_i$.  Without loss of generality, let us assume
  $r=\alpha_{\max}$. Then
\begin{displaymath}
    D^1 (a_1,b_1) = -(a_2,b_2) - \ldots - (a_6,b_6).
\end{displaymath}
By Remark \ref{num:corr-computation}, $D$ commutes with $\partial$, so
defines an endomorphism on $H_1(\Gamma',\bZ)$.

The endomorphism $D^1$ on $C_1(\Gamma',\bZ)$ splits into $12$
blocks each given by the $(6\times 6)$-matrix $N$ such that
$N_{ii}=0$ and $N_{ij}=-1$ for $i\ne j$. It is easy to see that
$(N-1)(N+5)=0$ and that the $(-5)$-eigenspace of $N$ is
$1$-dimensional and is generated by the vector
  $(a_1,b_1)+\ldots+(a_6,b_6)$.

This gives an identification $C_1(\Gamma',\bZ)^{(-5)}=
\bigoplus_{i=1}^{12}\bZ R_i$. The homomorphism $\partial\colon C_1\to
C_0$ is defined by $R_i\mapsto \sum_{s=1}^{27} (r_i, e^s) v^s$,
where $e^s$ are the 27 exceptional vectors. Since the bilinear form
on $E_6$ is nondegenerate and $e^s$ span $E_6^{\vee}$, one has
\begin{displaymath}
    \partial \Bigl(\sum_{i=1}^{12} n_i R_i\Bigr) = 0 \iff
    \left( \phi\Bigl(\sum_{i=1}^{12} n_i R_i\Bigr), e^s \right)=0 \ \mbox{ for  } s=1, \ldots, 27
    \iff \phi\Bigl(\sum_{i=1}^{12} n_i R_i\Bigr)=0.
  \end{displaymath}
  Therefore, $H_1^{(-5)} = C_1^{(-5)}\cap \mbox{Ker}(\partial) = \mbox{Ker}(\phi)$.
 \end{proof}

 It is an elementary linear algebra exercise to pick an
 appropriate basis in $\mbox{Ker}(\phi)$, which becomes especially easy if
 $r_1,\dotsc,r_6$ form a basis in $E_6$.

\begin{theorem}\label{thm:pt-monodromy}
The limit of PTK varieties $P(C_t,D_t)$ as a group is the
torus $(\bC^*)^6$ with the character group $H_1^{(-5)}$.
For each of the $21$ internal edges $e_i$ of the tree $T$, the
monodromy around the divisor $\{z_i=0\}$ in the neighborhood $U\subset \overline{\cM}_{0,24}$ is
given by the quadratic form
\begin{math}
M_i = \sum_{s=1}^{27} ((e_i^s)^*)^2
\end{math}
on $H_1^{(-5)}$.
\end{theorem}
\begin{proof}
The first statement is immediate: the limit of the Jacobians as a group
is a torus with the character group $H_1(\Gamma,\bZ)$, and the
PTK varieties are obtained by taking the $(-5)$-eigenspace.

Every internal edge $e_i$ of $T$ corresponds to a node of the curve
$E$. Over it, there are $27$ nodes of the curve $C$. The map is
\'etale, so the local coordinates $z_i^s$ for the smoothings of
these nodes can be identified with the local coordinate $z_i$. By
Section~\ref{sec:degs-jacobians}, the matrix for the monodromy
around $z_i^s=0$ is $((e_i^s)^*)^2$.  The monodromy matrix for
PTK varieties is obtained by adding these $27$ matrices
together and restricting to the $(-5)$-eigenspace.
\end{proof}

To compute the linear forms $(e^s_i)^*$ on $H_1(\Gamma,\bZ)$, we unwind the identification $H_1(\Gamma,\bZ) = H_1(\Gamma',\bZ)$.

\begin{lemma}\label{lem:explicit-fla}
  Let $p\colon \bigoplus_{i=1}^{12}\bZ R_i\to \bigoplus_{j=1}^{21} \bZ e_j$
  be the map which associates to $R_i$ the oriented path in the tree
  $T$ of Figure~\ref{fig:cover} from the central point $O$ to an end
  labeled $r_i$. Via
  the identification $H_1(\Gamma,\bZ)^{(-5)} =\mathrm{Ker}(\phi)\subset
  \bigoplus_{k=1}^{12}\bZ R_k$, the linear
  functions $(e_i^s)^*$ are defined by the formula
  \begin{displaymath}
    (e_i^s)^* (R_k) = \langle r_k, \ell_s \rangle
    \cdot \langle p(R_k), e_i^* \rangle,
  \end{displaymath}
  where the first pairing is $E_6\times E^*_6\to\bZ$, and for the second
  one $\langle e_j,e_i^*\rangle = \delta_{ij}$.
\end{lemma}
\begin{proof}
  Let $(v_{s_1},v_{s_2})$ be an edge in $\Gamma'$. To it, we associate
  the path in the graph $\Gamma$ going from the center of
  level $s_1$ to the center of level $s_2$:
\begin{displaymath}
  \vapath(O_{s_1}, r_1^{s_1})
  + \vapath(v_{s_1},v_{s_2})-
  \vapath(O_{s_2}, r_1^{s_2}).
\end{displaymath}
This rule gives an identification $H_1(\Gamma',\bZ)=H_1(\Gamma,\bZ)$.

For each of the $12$ roots $r_k$, we have $6$ edges in the graph
$\Gamma'$ going from the vertices $s$ with $\langle
r_k,\ell_s\rangle=1$ to the vertices $s$ with $\langle
r_k,\ell_s\rangle=-1$. The contribution of $R_k$ to the adjusted cycle
therefore is
\begin{displaymath}
  \sum_{s=1}^{27} \langle r_k, \ell_s \rangle \cdot \vapath(O_s, r_k) =
  \sum_{s=1}^{27} \langle r_k, \ell_s \rangle \cdot p(R_k) \bigr\rvert_{e_i= e_i^s}
\end{displaymath}
The value of the linear function $e_i^s$ on it is therefore given by the
formula in the statement.
\end{proof}

To complete the computation, we have to do the following:
\begin{enumerate}
\item Choose a basis of the $6$-dimensional space
$H_1(\Gamma,\bZ)^{(-5)} =\mathrm{Ker}(\phi)\subset \bigoplus_{k=1}^{12}\bZ R_k$.
\item Compute the $21\times 27$ linear functions $(e_i^s)^*$ on this $6$-dimensional space.
\item Compute the $21\times 27$ quadratic functions $((e_i^s)^*)^2$, each of which is a symmetric $6\times 6$-matrix.
\item And finally compute the $21$ monodromy matrices $M_i= \sum_{s=1}^{27} ((e_i^s)^{*})^2$ of Theorem~\ref{thm:pt-monodromy}.
\end{enumerate}

\begin{theorem}
  There exist collections of $E_6$ roots $r_1,\dotsc,r_{12}$
  generating the lattice $E_6$ for which the $21$ symmetric $(6\times
  6)$-matrices $M_i$ of Theorem~\ref{thm:pt-monodromy} are linearly
  independent.
\end{theorem}
\begin{proof}
  A concrete example is $r_1=\alpha_{135}$, $r_2=\alpha_{12}$,
  $r_3=\alpha_{23}$, $r_4=\alpha_{34}$, $r_5=\alpha_{45}$,
  $r_6=\alpha_{56}$, $r_7=\alpha_{456}$, $r_8=\alpha_{26}$,
  $r_9=\alpha_{123}$, $r_{10}=\alpha_{125}$, $r_{11}=\alpha_{256}$,
  $r_{12}=\alpha_{15}$.  An explicit computation using the the formula
  in Lemma~\ref{lem:explicit-fla}, aided by a computer algebra system,
  shows that
  \begin{enumerate}
  \item The monodromy matrices $M_i$ are all divisible by 6. This
    corresponds to the fact that the restriction of the principal
    polarization from the Jacobian to the PTK variety is
    6 times a principal polarization.
  \item For the normalized forms $M'_i=M_i/6$, the determinant of the
    corresponding $(21\times 21)$-matrix is $2^{12}\ne0$.
  \end{enumerate}
  A Mathematica notebook with an explicit computation is available at
  \cite{supporting-computer-computations}.
\end{proof}

\begin{corollary}
  Theorem \ref{thmmain1} holds.
\end{corollary}
\begin{proof}
  By the same argument as in the proof of
  Lemma~\ref{lem:torelli-full-dim}, the image of the complex-analytic
  map $U\to \cA_6$ has full dimension $21$. Thus, the map $PT:\Hur\to \cA_6$
  is dominant.
\end{proof}

\begin{remark}
Computer experimentation shows that for a very small portion of random choices
of the roots $r_1,\dotsc,r_{12}$, the matrices $M_i$ are linearly
independent. In most of these cases the determinant is $2^{12}$ but
in some cases it is $2^{13}$.

A necessary condition is for the roots $r_1,r_2$ to be
non-orthogonal, and similarly for the pairs $r_3,r_4$, etc.
Experimentation also shows that there is nothing special about
  the graph in Figure~\ref{fig:cover}. Any other trivalent graph with
$12$ vertices of degree one works no worse and no better.
\end{remark}


\section{Admissible covers and semiabelian Prym-Tyurin-Kanev varieties}\label{sec-Hurglobal}

In this section, we introduce the space $\hh$ of admissible
$E_6$-covers and define semiabelian Prym-Tyurin-Kanev varieties of $E_6$-admissible
pairs.
Then we study extensions of the Prym-Tyurin map to the Satake compactification
$\asat$ and the perfect cone toroidal compactification $\overline{\cA_6}:=\aperf$.


\subsection{The Hurwitz space}\label{subsecHurwitz}

\begin{num}\label{div5} We denote by  $\mathcal{H}$ the Hurwitz space of $E_6$-covers $\pi:C\rightarrow \bP^1$ together with a labeling $(p_1, \ldots, p_{24})$ of its branch points.
Let $\hh$ be the compactification of $\mathcal{H}$ by admissible $W(E_6)$-covers. By \cite{ACV}, the stack $\hh$ is isomorphic to the stack of \emph{balanced
 twisted stable} maps into the classifying stack $\mathcal{B} W(E_6)$ of $W(E_6)$, that is,
$$\hh:=\mm_{0,24}\Bigl(\mathcal{B} W(E_6)\Bigr).$$
By a slight abuse of notation, we shall use the same symbol $\hh$ both for the stack and for the associated coarse moduli space.
For details concerning the local structure of spaces of admissible coverings, we refer to \cite{ACV}. Note
that $\hh$ is a smooth stack isomorphic to the normalization of the \emph{Harris-Mumford moduli space} $\mathcal{HM}_{E_6}$ defined (in the case of covers with $S_n$-monodromy) in
\cite{HM}. The boundary $\hh\setminus \cM_{0,24}\bigl(\mathcal{B}W(E_6)\bigr)$ is a divisor with normal crossings. Points of $\mathcal{HM}_{E_6}$ are $E_6$-admissible coverings
$[\pi:C\rightarrow R, p_1, \ldots, p_{24}]$, where $C$ and $R$ are nodal curves of genus $46$ and $0$ respectively, and $p_1, \ldots, p_{24}\in R_{\mathrm{reg}}$
are the branch points of $\pi$. The local monodromy of $\pi$ around $p_i\in \mathbb P^1$ is given by a reflection $w_i\in W(E_6)$, for $i=1, \ldots, 24$. Let
$\mathfrak{b}:\hh\rightarrow \mm_{0, 24}$ be the \emph{branch} morphism and $\varphi:\hh\rightarrow \mm_{46}$ be the \emph{source} morphism. Obviously,
$S_{24}$ acts on $\hh$ and the projection $q:\hh\rightarrow \overline{\Hur}$ is a principal $S_{24}$-bundle. Passing to the $S_{24}$-quotient and denoting  $\widetilde{\cM}_{0,n}:=\mm_{0,n}/S_n$, we consider the induced branch and source maps
$$
\mathfrak{br}:\hhu \rightarrow \widetilde{\cM}_{0, 24} \ \ \mbox{ and } \ \ \ophi:\hhu\rightarrow \mm_{46},
$$
respectively. For $2\leq i\leq 12$,
let $B_i:=\sum_{|T|=i} \delta_{0:T}\in \mbox{Pic}(\mm_{0,24})$ be the boundary class, where the sum runs over all subsets $T\subset \{1, \ldots, 24\}$ of cardinality $i$.
Recall that $\delta_{0:T}$ is the class of the closure of the locus of pointed curves consisting of two rational components, such that the marked points lying on one component are precisely those labeled by $T$.
Let $\widetilde{B}_i$ be the \emph{reduced} boundary divisor on $\widetilde{\cM}_{0,24}$ which pulls-back to $B_i$ under the quotient map $\mm_{0,24}\rightarrow \widetilde{\cM}_{0,24}$.
\vskip 4pt

For each $E_6$-cover $[\pi:C\rightarrow \bP^1]\in \hhu$, there is an induced Kanev endomorphism at the level of Jacobians $D:JC\rightarrow JC$ and at the level of differentials
$D: H^0(C,\omega_C)\rightarrow H^0(C,\omega_C)$, which we denote by the same symbol. This induces a splitting
$$H^0(C, \omega_C)=H^0(C,\omega_C)^{(+1)}\oplus H^0(C, \omega_C)^{(-5)}$$ into $(+1)$ and $(-5)$-eigenspaces respectively. From (\ref{dimpt1}), it follows that
$$\mbox{dim } H^0(C,\omega_C)^{(-5)}=\mbox{dim } PT(C,D)=5,$$ hence $\mbox{dim } H^0(C,\omega_C)^{(+1)}=40$. We have a decomposition of the
rank $46$ Hodge bundle $\mathbb E:=\ophi^*(\mathbb E)$ pulled-back from $\mm_{46}$ into eigenbundles
$$\mathbb E=\mathbb E^{(+1)}\oplus \mathbb E^{(-5)},$$ where, as we pointed out,  $\mbox{rk}(\mathbb E^{(+1)})=40$ and $\mbox{rk}(\mathbb E^{(-5)})=6$. We set
$\lambda^{(+1)}:=c_1(\mathbb E^{(+1)})$ and $\lambda^{(-5)}:=c_1(\mathbb E^{(-5)})$, therefore $\lambda:=\ophi^*(\lambda)=\lambda^{(+1)}+
\lambda^{(-5)}$. We summarize the discussion in the following diagram:

\begin{equation}\label{Snaction}
     \xymatrix{
         \hh \ar[r]^{q} \ar[d]_{\mathfrak{b}} & \hhu \ar[d]_{\mathfrak{br}} \ar[r]^{\ophi} & \overline{\cM}_{46} \\
          \overline{\cM}_{0,24} \ar[r]^{}       & \widetilde{\cM}_{0,24} . }
\end{equation}
\end{num}

\subsection{Semiabelian Prym-Tyurin-Kanev varieties}

\begin{lemma}\label{lem:singular-corr}
  Let $[\pi\colon C\to R, p_1, \ldots, p_{24}]\in\hh$ be an
  $E_6$-admissible cover. Then it comes with a correspondence
  $\wD\subset C\times C$ inducing an endomorphism $D\colon JC\to JC$ of
  the semiabelian group variety $JC=\Pic^0(C)$, which satisfies the
  same identity $(D-1)(D+5)=0$ as for covers of smooth curves. The
  group variety $PT(C,D):=\mathrm{Im}(D-1)$ is a semiabelian group
  subvariety of $JC$.
\end{lemma}
\begin{proof}
  Consider any one-parameter family $\pi_s\colon\cC_s\to\R_s$ of
  $E_6$-admissible covers over a smooth base $(S,0)$ such that $\pi_s$ is
  smooth for $s\ne0$ and $\pi_0=\pi$. It gives an identification of
  the smooth fibers of $\pi$ with the nearby fibers of $\pi_s$ up to
  the monodromy $W(E_6)$. Thus, we have a symmetric degree $10$
  correspondence $\wD^0$ on the smooth locus of $\pi$, and it
  clearly does not depend on a chosen one-parameter family. We take
  $\wD\subset C\times C$ to be the closure of $\wD^0$.

  By the above identification, this correspondence satisfies the
  identity
  $$(D'-1)(D'+5) = 5\cdot\operatorname{Trace(\pi)}.$$ The
  target curve $R$ is a tree of $\bP^1$s, so $\Pic^0(R)=0$. Any
  element of $\Pic^0(C)$ can be represented by a divisor on $C$ of
  multidegree $(0,\dotsc,0)$ supported on the smooth locus of
  $C$. Thus, $\wD$ induces a correspondence
  $D\colon\Pic^0 C\to \Pic^0(C)$ satisfying $(D-1)(D+5)=0$.
  The image of a homomorphism $(D-1)\colon JC\to JC$ is a semiabelian
  variety, which finishes the proof.
\end{proof}

\begin{definition}
  Any semiabelian variety $G$ has a unique extension
  $1\to T\to G\to A\to 0$. We call $\tor(G)=T$ and $\ab(G)=A$ the
  toric and abelian parts of $G$ respectively, and we call their dimensions the
  toric and abelian ranks of $G$.  In particular, we shall talk about
  the toric rank $k_{PT}$ of a Prym-Tyurin-Kanev (PTK) variety $PT$.
\end{definition}

\begin{lemma}\label{lem:semiabelian-Prym}
  Let $1\to T\to G\to A\to 0$ be a semiabelian variety with an
  endomorphism satisfying $(D-1)(D+5)=0$. Then:
  \begin{enumerate}
  \item $P:=\im(D-1)$ coincides with $\mathrm{Ker}(D+5)_0$, the connected
    component of identity.
  \item $D$ induces homomorphisms $D_T$ of the toric part $T$ and
    $D_A$ of the abelian part $A$.
  \item The toric part of $P$ coincides with $P_T:=(D_T-1)T$, and the
    abelian part of $P$ is isogenous to $P_A:=(D_A-1)A$.
  \end{enumerate}
\end{lemma}
\begin{proof}
  (1) One has the inclusions $\mbox{Im}(D-1)\subset \mbox{Ker}(D+5)^0$,
  $\im(D+5)\subset \mbox{Ker}(D-1)^0$, and
  $$\mbox{Ker}(D-1)\cap\mbox{Ker}(D+5)\subset G[6],$$ which is a finite group.
  Since $G$ is also 6-divisible, $\im(D-1)$ are $\im(D+5)$ span $G$, and
  so they are semiabelian subvarieties of complementary dimensions. The
  quotient $\mbox{Ker}(D+5)/\im(D-1)$ is the kernel of a surjective
  homomorphism $(D+5)\colon G/\im(D-1)\to \mbox{Im}(D+5)$ of varieties of the same
  dimension, so it is finite. Thus, $\im(D-1) = \mbox{Ker}(D+5)_0$.

  (2) The homomorphism from the affine variety $T$ to the projective
  variety $A$ is constant, so $D(T)\subset T$ and we get $D_T = D|_T$,
  which in turn induces an endomorphism $D_A$ of $A=G/T$.

  (3) Clearly, $\im(D_T-1) \subset \im(D-1)$ and
  $\im(D-1) \onto \im(D_A-1)$. The kernel of the homomorphism
  $\im(D-1) \to \im(D_A-1)\subset A$ is $T\cap \im(D-1)$. One has
  $$\im(D_T-1) \subset T\cap \im(D-1) \subset \mbox{Ker}(D_T+5).$$ By (1)
  applied to $T$ the difference between the last and the first groups
  is finite. Therefore, the difference between the middle and the first
  groups is finite. In other words, the homomophism $\mbox{Ker}(P/P_T \to
  P_A)$ has finite kernel. Thus, $P/P_T$ is an abelian variety that is
  the abelian part of $P$, it is isogenous to $P_A$, and $P_T$ is the toric
  part of $P$.
\end{proof}

The structure of an $E_6$-admissible map $\pi\colon C\to R$ makes the
computation of $PT(C,D)$ especially easy. The target curve $R$ is a
tree $\{R_i\simeq\bP^1\}_{i=1}^n$ of smooth rational curves.
Then we have induced maps $\pi_i\colon C_i\to R_i$ and
correspondences $D_i$ for $C_i$ satisfying the same quadratic
identity. The curves $C_i$ are smooth; however, they may be
disconnected.

\begin{lemma}\label{lem:parts-of-pt}
  One has the following:
  \begin{enumerate}
  \item The abelian part of $PT(C,D)$ is isogenous to $\prod_{i=1}^n
    P(C_i,D_i)$.
  \item The correspondence $\widetilde{D}$ induces correspondences $D_{C_0}$,
    $D_{C_1}$ and $D_{H_1}$ on the cycle groups
    $C_0(\Gamma, \mathbb Z)$, $C_1(\mathbb Z)$ and on the homology group $H_1(\Gamma,\mathbb Z)=\mathrm{Ker}\bigl\{C_1(\Gamma,\mathbb Z)\to C_0(\Gamma,\mathbb Z)\bigr\}$ of the dual graph
    $\Gamma$ of $C$. The toric part of the semiabelian variety
    $PT(C,D)$ has the character group $\im (D_{H_1}-1)$.
  \end{enumerate}
\end{lemma}
\begin{proof}
  (1) follows from Lemma \ref{lem:semiabelian-Prym}. The definition and
  properties of the correspondences on $C_0(\Gamma,\mathbb Z)$, $C_1(\Gamma, \mathbb Z)$ and $H_1(\Gamma, \mathbb Z)$ are
  immediate. For a homomorphism $\phi\colon T_1\to T_2$ of tori with
  dual homomorphism $\phi^*\colon X_2\to X_1$ of character lattices,
  $\im\phi$ is a torus with character lattice $\im\phi^*$. We apply
  this to the toric part of $JC$ whose character lattice is
  $H_1(\Gamma,\bZ)$ and use Lemma \ref{lem:semiabelian-Prym}.
\end{proof}

\subsection{Extensions of the Prym-Tyurin map}

\begin{num}\label{targets-of-extended-pt}
There are several natural targets to consider for the extended
Prym-Tyurin map. The easiest one is the Satake-Baily-Borel
compactification
$\agsat = \cA_g\sqcup\cA_{g-1}\sqcup\dots\sqcup\cA_0$ ($g=6$ in our
case). Other natural targets are the toroidal compactifications
$\overline{\cA_g}=\agperf$ and $\agvor$ for the perfect cone, respectively 2nd Voronoi fans. The
space $\agperf$ has the advantage of having only one boundary
divisor. The space $\agvor$ is modular: by
\cite{alexeev2002complete-moduli} it is the normalization of the main
component of the moduli space of principally polarized stable
semiabelic varieties.

All toroidal compactifications of $\cA_g$ contain the same open subset
$\widetilde{\cA}_g$ introduced by Mumford
\cite{mumford1983on-the-kodaira-dimension}, the moduli space of
principally polarized abelian varieties of dimension $g$ together with
their degenerations of toric rank $1$, This is a partial
compactification of $\cA_g$ isomorphic to the blow-up of the open
subset
$\cA^{\rm Sat}_{g, \mathrm{tor.rk}\leq 1}=\cA_g\sqcup \cA_{g-1}$ in
$\agsat$. Moreover, $\widetilde{\cA}_g=\cA_g \sqcup \widetilde{D}_g$,
where $\widetilde{D}_g$ is the universal Kummer variety over
$\mathcal{A}_{g-1}$.  The closure of $\wD_g$ in $\agperf$ is the
unique boundary divisor $D_g$.

\vskip 5pt

On the other hand, by \cite{mumford1983on-the-kodaira-dimension}, any
semiabelian variety $G$ with a principally polarized abelian part of dimension $g-1$ has
a canonical compactification $X$, a rank-1 degeneration of ppav. In
the language of \cite{alexeev2002complete-moduli}, in this case there
exists a unique principally polarized stable semiabelic pair
$(G\acts X\supset\Theta)$. So the moduli of toric rank $\le1$
semiabelian varieties and the moduli of toric rank $\le1$ stable
semiabelic pairs are the same.
\end{num}



\begin{lemma}
  Let $\cC\to\cR\to S$ be a family of admissible covers parameterized
  by $\hh$ over a reduced scheme $S$ such that the restrictions of
  $\cC\to S$, $\cR\to S$ are smooth over an open dense subset. Then
  there is a semiabelian group scheme $\cP\cT\to S$ whose fiber over
  $s\in S$ is $PT(C_s,D_s)$ as defined above.
\end{lemma}
\begin{proof}
  By taking the closure of the correspondence over $U$, we obtain a
  correspondence on $\cC$ whose fibers were described in
  Lemma \ref{lem:singular-corr}.  Thus, we have a semiabelian group scheme
  $\Pic^0_{\cC/S}$ together with an endomorphism $\cD$ over $S$ giving
  an endomorphism as in Lemma \ref{lem:singular-corr} fiberwise. Then
  $\cP\cT:=\im(\cD-\operatorname{id})$ satisfies the conditions of the
  statement.
\end{proof}

\begin{lemma}\label{lem:map-to-satake}
  The map $PT:\mathrm{Hur}\to \cA_6$ extends to a regular map
  $PT^{\mathrm{sat}}:\hhu\to \asat$. A point $[C\to R]$ maps to the abelian part
  of $PT(C,D)$.
\end{lemma}
\begin{proof}
As discussed in the previous section, $\hhu$ is a smooth stack and the boundary
$\hhu\setminus\mathrm{Hur}$ is a divisor with normal crossings.

  The extension exists by Borel's
  extension theorem \cite{borel}. To find the image of $[C\to R]$ it
  is sufficient to consider a one-parameter family of covers
  parameterized by $(S,0)$ with covers of smooth curves for $s\ne 0$.
  It is known that for any family of dimension $g$ semiabelian
  varieties $G\to S$ whose restriction to $S\setminus 0$ is a family
  of principally polarized abelian varieties, the image of $0\in S$ of
  the extended map to $\cA_g^{\rm Sat}$ is the abelian part of $G_0$. We apply
  this to the family $\cP\cT\to S$ of the previous lemma.
\end{proof}

\begin{theorem}\label{thm:map-to-perf}
  The rational map $PT: \hhu \dashrightarrow \aperf$ has an indeterminacy locus of
  codimension at least $2$. On the open subset
  $\Hur_{\mathrm{tor.rk}\leq 1}\subset\hhu$ where PTK varieties
  have toric rank $\le1$, the map is regular and \emph{proper}, and
  it factors through $\widetilde{\cA}_6$.
\end{theorem}
\begin{proof}
  The indeterminacy locus has codimension $\ge2$ simply because $\hhu$
  (viewed as a variety) is normal and $\aperf$ is proper.  By the
  discussion in Paragraph \ref{targets-of-extended-pt}, we have a
  regular map $\Hur_{\mathrm{tor.rk}\leq 1}\to \avor$ and
  Lemma~\ref{lem:map-to-satake} implies that the image is contained in
  $\widetilde{\cA}_6$.  Thus, we have a morphism
  $PT\colon\Hur_{\mathrm{tor.rk}\leq 1}\to\wcA_6\subset\aperf$.  Since
  both maps $g\colon\wcA_6\to\cA^{\rm Sat}_{6,\mathrm{tor.rk}\le1}$
  and
  $g\circ PT\colon\Hur_{\mathrm{tor.rk}\leq 1}\to\cA^{\rm
    Sat}_{6,\mathrm{tor.rk}\le1}$
  are proper, it follows that $PT$ is proper.
\end{proof}

\section{Positivity properties of the Hurwitz space of $E_6$-covers}\label{sec:pos}

In this section we study in detail the divisor theory on the Hurwitz space $\hh$. In particular, we express explicitly the Hodge class $\lambda$ in terms of boundary classes
and we prove that the canonical class of $\hh$ is big. We recall that the pull-back of the Hodge class under the morphism
\[
\ophi:\hhu\rightarrow \mm_{46},
\]
is the sum
\[
\lambda :=\ophi^*(\lambda)=\lambda^{(+1)}+
\lambda^{(-5)} := c_1(\mathbb E^{(+1)}) + c_1(\mathbb E^{(-5)})
\]
of the first Chern classes of the Hodge eigenbundles for the $+1$ and the $-5$ eigenvalues, see Paragraph \ref{subsecHurwitz}.

\begin{lemma}
Both Hodge eigenclasses $\lambda^{(+1)}$ and $\lambda^{(-5)}\in CH^1(\hhu)$ are nef.
\end{lemma}
\begin{proof} Koll\'ar \cite{Kol} showed that the Hodge bundle $\mathbb E$ is semipositive, therefore the eigenbundles $\mathbb E^{(+1)}$ and $\mathbb E^{(-5)}$
as quotients of $\mathbb E$ are semipositive as well. Therefore $\mbox{det}(\mathbb E^{(+1)})$ and $\mbox{det}(\mathbb E^{(-5)})$ are nef line bundles.
\end{proof}

\begin{num} The conclusion of Theorem \ref{thmmain1} can be restated in terms of the positivity of $\lambda^{(-5)}$. The fact that the map $PT:\Hur \rightarrow \cA_6$ is dominant implies that the class $\lambda^{(-5)}\in CH^1(\hhu)$ is big. One can formulate a necessary and sufficient condition for $\cA_6$ to be of general type in similar terms.

\begin{corollary}\label{koda6}
Let $D_i$ be the irreducible divisors supported on
$\ohur\setminus\Hur_{\mathrm{tor.rk}\leq 1}$. Then to prove that $\cA_6$ is of general
type, it suffices to show that there exist some integers $a_i$ such
that the divisor $\overline{PT^*(K_{\overline{\cA}_6})}+\sum a_iD_i$ on $\hhu$
is big.
\end{corollary}
\begin{proof}
If this divisor is big on $\hhu$, then its corresponding linear system has maximal Iitaka dimension. Then the linear
system $|PT_*PT^*(K_{\overline{\cA}_6})|$ has maximal
Iitaka dimension as well. Since all boundary divisors $D_i$ are contracted under the Prym-Tyurin map, we write $PT_*\Bigl(PT^*(K_{\overline{\cA}_6})\Bigr)=PT_*\Bigl(PT^*(K_{\overline{\cA}_6})+\sum_i a_iD_i\Bigr)= \deg(PT)K_{\overline{\cA}_6}$. So $K_{\overline{\cA}_6}$ is big.
\end{proof}

\end{num}

\vskip 4pt

We now turn to describing the geometry of the Hurwitz space $\hh$. We make the following:
\begin{definition}\label{def:lcm-mu}
For a partition $\mu=(\mu_1, \ldots, \mu_{\ell})\vdash n$, we define
$\mbox{lcm}(\mu):=\mbox{lcm}(\mu_1, \ldots, \mu_{\ell})$ and $\frac{1}{\mu}:=\frac{1}{\mu_1}+\cdots+\frac{1}{\mu_{\ell}}$.
For $i=2, \ldots, 12$, we denote by $\mathcal{P}_{i}$ the set of partitions $\mu\vdash 27$ associated to the cycle decompositions of the conjugacy classes of products of $i$ reflections in $W(E_6) \subset S_{27}$.
\end{definition}

The possible partitions of $27$ corresponding to products of reflections can be read off Table \ref{tab:conj-classes}. For instance, we find that $\mathcal{P}_2=\bigl\{(2^{10}, 1^7), (3^6,1^9)\bigr\}$.

\begin{num}\label{subsec-boundcomb} We now describe a way of indexing the boundary divisors of $\hhu$. We fix the following combinatorial data:
\begin{enumerate}
\item A partition $I\sqcup J=\{1,\dotsc,24\}$, such that $|I|\ge2$,
$|J|\ge2$.
\item Reflections $\{w_i\}_{i\in I}$ and $\{w_j\}_{j\in J}$ in $W(E_6)$, such
that $\prod_{i\in I} w_i = u$, $\prod_{j\in J}w_j=u\inv$, for some
$u\in W(E_6)$. The sequence $w_1,\dotsc,w_{24}$ is defined up to
conjugation by the same element $g\in W(E_6)$.
\end{enumerate}

To this data, we associate the locus of $E_6$-admissible covers with labeled branch points
$$t:=\bigl[\pi: C\rightarrow R, \ p_1, \ldots, p_{24}\bigr]\in \mathcal{HM}_{E_6},$$
where $[R=R_1\cup_q R_2, p_1, \ldots, p_{24}]\in B_{|I|}\subset \mm_{0,24}$ is a pointed union of two smooth rational curves meeting at the point $q$. The marked points lying on $R_1$ are precisely those labeled by the set $I$. Over $q$, the map $\pi$ is ramified according to $u$,
that is, the points in $\pi^{-1}(q)$ correspond to cycles in the
permutation $u$ considered as an element of $S_{27}$. Let $\mu:=(\mu_1, \ldots, \mu_{\ell})\vdash 27$ be the partition induced by $u\in S_{27}$ and denote by $E_{i:\mu}$ the boundary divisor on $\hh$ classifying $E_6$-twisted stable maps with underlying admissible cover as above, with $\pi^{-1}(q)$ having partition type $\mu$, and precisely $i$ of the points $p_1, \ldots, p_{24}$ lying on $R_1$. Only partitions from the set $\mathcal{P}_i$ introduced in Definition \ref{def:lcm-mu} are considered.

In Table \ref{tab:conj-classes} we give the list of partitions of $27$ appearing as products of reflections in $W(E_6)$ (using the GAP notation for the conjugacy classes). For future use, we also record the invariants $\frac{1}{\mu}$, for each partition $\mu$. If one partition from this list appears in $\mathcal{P}_i$, it will appear in $\mathcal{P}_{i+2j}$, for all $i+2j\leq 12$.

\vskip15pt

{\Small
  \begin{table}[htp!]
\begin{center}
\begin{tabular}{|c|c|c|c|}
 \hline
Minimal number of reflections & Partition $\mu$ of $27$ & Conjugacy class  & $\frac{1}{\mu}$ \\
\hline
0 &  $1^{27}$ & 1a &  27 \\
1 &   $(2^6, 1^{15} )$ & 2c & 18 \\
2  &   $(2^{10}, 1^{7} )$ & 2b &  12 \\
2  &   $(3^6, 1^{9} )$ & 3b &  11\\
3  &   $(2^{12}, 1^{3} )$ & 2d &  9\\
3  &   $(4^{5}, 2^{1}, 1^5 )$ & 4d & $\frac{27}{4}$\\
3  &   $(6^{1}, 3^{4}, 2^3, 1^3 )$ & 6e & 6 \\
4 & $(2^{12}, 1^3 )$  & 2a & 9\\
4 & $(3^{9} )$  & 3c & 3\\
4 & $(4^{6}, 1^3 )$  & 4a & $\frac{9}{2}$ \\
4  & $(4^{5}, 2^3, 1 )$  & 4b & $\frac{15}{4}$ \\
4 & $(5^{5}, 1^2 )$  & 5a & 3 \\
4 & $(6^{3}, 2^3, 1^3)$  & 6b & 5\\
4 & $(6^{2}, 3^2, 2^4, 1)$  & 6d &  4\\
5 & $(4^{5}, 2^3, 1)$  & 4c & $\frac{15}{4}$\\
5 & $(6^{2}, 3^5)$  & 6f & 2\\
5 & $(6^{4}, 3^1)$  & 6g & 1\\
5 & $(8^{3}, 2^1, 1)$  & 8a & $\frac{15}{8}$ \\
5 & $(10^{1}, 5^3, 2^1)$  & 10a & $\frac{6}{5}$ \\
5 & $(12^{1}, 6^1, 4^2, 1)$  & 12b & $\frac{7}{4}$ \\
6 & $(3^9)$ & 3a & $3$ \\
6 & $(6^4, 3^1 )$ & 6a & $1$ \\
6 & $(6^3, 2^3, 1^3)$ & 6c & 5\\
6 & $(9^3 )$ & 9a & $\frac{1}{3}$ \\
6 & $(12^2, 3^1)$ & 12a & $\frac{1}{2}$ \\
\hline
\end{tabular}
\end{center}
    \caption{Products of reflections in $W(E_6)$}
    \label{tab:conj-classes}
  \end{table}
}

\vskip10pt
\end{num}

\begin{num} We recall the local structure of the morphism $\mathfrak{b}:\hh \rightarrow \mm_{0, 24}$, over the point $t$, see also \cite{HM} p.62.  The (non-normalized) space
$\mathcal{HM}_{E_6}$ is locally described by its local ring
\begin{equation}\label{localring}
\widehat{\cO}_{t, \mathcal{HM}_{E_6}}=\mathbb C
[[t_1, \ldots, t_{21}, s_1, \ldots, s_{\ell}]]/s_1^{\mu_1}=\cdots=s_{\ell}^{\mu_{\ell}}=t_1,
\end{equation}
where $t_1$ is the local parameter on $\mm_{0, 24}$ corresponding to smoothing the node $q\in R$. By passing to the normalization $\nu:\hh\rightarrow \mathcal{HM}_{E_6}$, we deduce that over each point  $t'\in \nu^{-1}(t)$ the map $\mathfrak{b}:\hh\rightarrow \mm_{0, 24}$ is ramified with index $\mathrm{lcm}(\mu)$. The fibre $\nu^{-1}(t)$ consists of $\mu_1 \cdots \mu_{\ell}/\mbox{lcm}(\mu)$ points. The local ring at $t'$ is then given by
$\widehat{\cO}_{t',\hh}=\mathbb C[[\tau, t_2, \ldots, t_{b-3}]],$
and the normalization map $\nu$ is given in local coordinates by
$$t_1=\tau^{\mathrm{lcm}(\mu)}, \ s_1=\zeta_{\mu_1}^{\frac{\mathrm{lcm}(\mu)}{\mu_1}}, \ \ldots, \ s_{\ell}=\zeta_{\mu_{\ell}}^{\frac{\mathrm{lcm}(\mu)}{\mu_{\ell}}},$$
where $\zeta_{\mu_i}$ is a $\mu_i$-th root of unity for $i=1, \ldots, \ell$. This description implies that for each $i=2, \ldots, 12$, we have a decomposition $$\mathfrak{b}^*(B_i)=\sum_{\mu\in \mathcal{P}_i} \mathrm{lcm}(\mu) E_{i:\mu}.$$
\end{num}

In view of applications to the Kodaira dimension of $\hh$, we discuss in detail the pull-back $\mathfrak{b}^*(B_2)$. We pick a point $t=[\pi:C=C_1\cup C_2\rightarrow R=R_1\cup_q R_2,\  p_1, \ldots, p_{24}]\in b^*(B_2)$ as in in \ref{subsec-boundcomb}, where $C_i=\pi^{-1}(R_i)$. Without loss of generality, we assume that $I=\{1, \ldots, 22\}$, thus $p_1, \ldots, p_{22}\in R_1$ and $p_{23}, p_{24}\in R_2$. The group $G=\langle w_1, \ldots,
w_{22}\rangle$ generated by the reflections in the remaining roots $r_1, \ldots, r_{22}\in E_6$  is the Weyl group for a  lattice $L=L_G\subset E_6$. Since $\prod_{i=1}^{24} w_i=1$, it follows that $w_{23}\cdot w_{24}\in G$, hence $\mbox{rk}(L)\geq \mbox{rk}(E_6)-1=5$.

\begin{num} \label{subsec-E_0} Assume that the reflections $w_{23}$ and $w_{24}$ corresponding to the coalescing points $p_{23}$ and $p_{24}$ are equal, hence $w_{23}=w_{24}$. In this case, the corresponding partition is $\mu=(1^{27})$ and we set $E_0:=E_{2:1^{27}}$.
We denote by $E_{L}$ the boundary divisor of admissible covers in $E_{2:(1^{27})}$ corresponding to the
lattice $L$. The map $\mathfrak{b}$ is unramified along each divisor $E_L$ and we have $$E_0=\sum_{L\subset E_6} E_L\subset \hh.$$
The general cover $t$ corresponding to each divisor $E_L$ carries no automorphism preserving all branch points $p_1, \ldots, p_{24}$, that is, $\mbox{Aut}(t)=\{\mbox{Id}\}$.

\vskip 3pt

Suppose now that the reflections $w_{23}$ and $w_{24}$ are distinct. Following \cite[Section 9.1]{dolgachev2012classical-algebraic}, we distinguish two
possibilities depending on the relative position of the two double-sixes, described in terms of a general admissible cover $t=\bigl[\pi:C=C_1\cup C_2 \rightarrow R_1\cup_q R_2, \ p_1, \ldots, p_{24}\bigr]$.
\end{num}

\begin{num}\label{subsec-azy} The reflections $w_{23}$ and $w_{24}$ form an \emph{azygetic} pair, that is, the corresponding roots $r_{23}$ and $r_{24}$ satisfy $r_{23}\cdot r_{24}\neq 0$. In this case, $\langle w_{23}, w_{24}\rangle =W(A_2)$ and $r_{23}+r_{24}$ or $r_{23}-r_{24}$ is again a root that is azygetic to both $r_{23}$ and $r_{24}$. The double-sixes associated to $w_{23}$ and $w_{24}$ share $6$ points and the permutation $w_{23}\cdot w_{24}$ decomposes into $6$ disjoint three
cycles, therefore $\mu=(3^6, 1^9)\vdash 27$. Accordingly, $C_2=\pi^{-1}(R_2)$ decomposes into six rational components mapping $3:1$, respectively $9$
components mapping isomorphically onto $R_2$. If $$E_{\mathrm{azy}}:=E_{2:(3^6, 1^9)}\subset \hh$$ is the boundary divisor parametrizing such
points, then $\mathfrak{b}$ is triply ramified along $E_{\mathrm{azy}}$. The general point of $E_{\mathrm{azy}}$ has no non-trivial automorphisms
preserving all the branch points.
\end{num}

\begin{num}\label{subsec-syz} The reflections $w_{23}$ and $w_{24}$ form a \emph{syzygetic} pair, that is,  $r_{23}\cdot r_{24}=0$. We have $\langle w_{23}, w_{24}\rangle =W(A_1^2)$. The two associated double-sixes share $4$ points and $w_{23}\cdot w_{24}\in S_{27}$ decomposes into a product of $10$ disjoint transpositions, therefore $\mu=(2^{10}, 1^7)$. Eight of these transpositions are parts of the double-sixes corresponding to $w_{23}$ and $w_{24}$ that remain disjoint respectively. Note that $C_2$ consists of $8$ rational components mapping $2:1$ onto $R_2$, as well as a smooth rational component, say $Z$, mapping $4:1$ onto $R_2$. The fibers $\pi_Z^{-1}(q), \pi_Z^{-1}(p_{23})$ and $\pi_Z^{-1}(p_{24})$ each consist of two ramification points. We denote by $$E_{\mathrm{syz}}:=E_{2: (2^{10}, 1^7)}\subset \hh$$ the boundary divisor of admissible syzygetic covers. For a general cover $t\in E_{\mathrm{syz}}$, note that $\mbox{Aut}(t)=\mathbb Z_2$, see Remark \ref{aut1}.
\end{num}

\begin{num} To summarize the discussion above, we have the following relation:
\begin{equation}\label{pb}
\mathfrak{b}^*(B_2)=E_0+3E_{\mathrm{azy}}+2E_{\mathrm{syz}}.
\end{equation}
In opposition to $E_0$, we show in Theorem \ref{thm:syzazy} that the  boundary divisors  $E_{\mathrm{azy}}$ or $E_{\mathrm{syz}}$ have fewer components. Precisely, for a general element $t\in E_{\mathrm{azy}}$ or $t\in E_{\mathrm{azy}}$, we have $G=W(L)=W(E_6)$, hence the subcurve $C_1=\pi^{-1}(R_1)$ is irreducible.
\end{num}

\begin{num} The Hurwitz formula applied to the ramified cover $\mathfrak{b}:\hh\rightarrow \mm_{0, 24}$, coupled with the expression
$K_{\mm_{0, 24}}\equiv \sum_{i=2}^{12} \bigl(\frac{i(24-i)}{23}-2\bigr)B_i$ to be found, e.g., in \cite{KM}, yields
\begin{equation}\label{canhur}
K_{\hh}=\mathfrak{b}^*K_{\mm_{0, 24}}+\mbox{Ram}(\mathfrak{b})= -\frac{2}{23}[E_0]+\frac{19}{23} [E_{\mathrm{syz}}]+\frac{40}{23}[E_{\mathrm{azy}}]+N,
\end{equation}
where $N$ is the \emph{effective} combination of the boundary divisors of $\hh$ disjoint from  $E_0, E_{\mathrm{syz}}$ and $E_{\mathrm{azy}}$, with the coefficient of $[E_{i:\mu}]$ for $i=3, \ldots, 12$ being equal to $\mathrm{lcm}(\mu)\bigl(\frac{i(24-i)}{23}-1\bigr)-1>0$.

\vskip 3pt

The ramification divisor of the projection $q:\hh\rightarrow \overline{\Hur}$ is contained in the pull-back $\mathfrak{b}^*(B_2)$ (recall the commutative diagram \ref{Snaction}). Note that $B_2$ is the ramification divisor of the quotient  map $\mm_{0, 24}\rightarrow \widetilde{\cM}_{0, 24}$. The general point of each of the components of $E_0$ and $E_{\mathrm{azy}}$ admits an involution compatible with the involution of the rational curve $R_2$ preserving $q$ and interchanging the branch points $p_{23}$ and
$p_{24}$ respectively. No such automorphism exists for a general point of the divisor $E_{\mathrm{syz}}$ (see Remark \ref{aut1}), thus $$\mbox{Ram}(q)=E_0+E_{\mathrm{azy}}.$$

\begin{remark}\label{aut1}
We illustrate the above statement in the case of the divisor $E_{\mathrm{syz}}$. We choose a general point $t:=\bigl[\pi:C\rightarrow R=R_1\cup_q R_2, p_1, \ \ldots, p_{24}\bigr]\in E_{\mathrm{syz}}$, and denote by $\pi_Z:Z\rightarrow R_2$, the degree $4$ cover having as source a smooth rational curve $Z$ and such that $\pi_Z^*(q)=2u+2v,$ and $\pi_Z^*(p_i)=2x_i+2y_i$, for $i=23, 24$. Then $\mbox{Aut}(t)=\mathbb Z_2$. Indeed, there exists a \emph{unique} automorphism  $\sigma\in \mbox{Aut}(Z)$ with $\sigma(u)=u, \sigma(v)=v$, $\sigma(x_{23})=y_{23}$ and $\sigma(x_{24})=y_{24}$ and such that $\pi_Z\circ \sigma=\pi_Z$. Note that $\sigma$ induces the unique non-trivial automorphism  of $t$ fixing all the branch points.
In contrast, the general point of $E_{\mathrm{azy}}$ corresponds to an admissible cover which has no automorphisms fixing all the branch points.
\end{remark}

\begin{definition}\label{divunord}
On the space $\hhu$ of unlabeled $E_6$-covers, we introduce the \emph{reduced} boundary divisors $D_0, D_{\mathrm{syz}}, D_{\mathrm{azy}}$, as well as  the boundary divisors $\bigl\{D_{i:\mu}:3\leq i\leq 12, \mu\in \mathcal{P}_i\bigr\}$ which pull-back under the map $q:\hh\rightarrow \hhu$ to the corresponding divisors indexed by $E$'s, that is,
$q^*(D_0)=2E_0$, $q^*(D_{\mathrm{azy}})=2E_{\mathrm{azy}}$, \ $q^*(D_{\mathrm{syz}})=E_{\mathrm{syz}}$ and $q^*(D_{i:\mu})=E_{i:\mu}$, for $3\leq i\leq 12$ and $\mu\in \mathcal{P}_i$. More generally, for each sublattice $L\subset E_6$, we denote by $D_L\subset \hhu$ the reduced divisor which pulls back to $E_L$ under the map $q$.

If $D$ is an irreducible divisor on $\hhu$, we denote as usual by $[D]:=[D]_{\mathbb Q}\in CH^1(\hhu)_{\mathbb Q}$ its $\mathbb Q$-class, that is, the quotient of its usual class by the order of the automorphism group of a general point from $D$.
\end{definition}

\begin{theorem}\label{kanonisch}
The canonical class of the Hurwitz space $\hhu$ is given by the formula:
$$K_{\hhu}= -\frac{25}{46}[D_0]+\frac{19}{23}[D_{\mathrm{syz}}]+\frac{17}{46}[D_{\mathrm{azy}}]+\sum_{i=3}^{12}\sum_{\mu \in \mathcal{P}_{i}} \ \Bigl(\mathrm{lcm}(\mu)\Bigl(\frac{i(24-i)}{23}-1\Bigr)-1\Bigr) [D_{i:\mu}].$$
\end{theorem}
\begin{proof}
We apply the Riemann-Hurwitz formula to the map $q:\hh\rightarrow \hhu$ and we find
$$q^*(K_{\overline{\Hur}})=K_{\hh}-[E_0]-[E_{\mathrm{azy}}]=-\frac{25}{23} [E_0]+\frac{19}{23} [E_{\mathrm{syz}}] +\frac{17}{23}[E_{\mathrm{azy}}]+\cdots\in CH^1(\hh).$$
\end{proof}
\end{num}

\subsection{The Hodge class on the space of admissible $E_6$-covers}

\begin{num} We describe the Hodge class on the Hurwitz space $\hh$ in terms of boundary divisors.  Let $\psi_1, \ldots, \psi_{24}\in \mbox{Pic}(\mm_{0, 24})$ be the cotangent tautological classes corresponding to the marked points. The universal curve over $\mm_{0, 24}$ is the morphism $\pi:=\pi_{25}:\mm_{0, 25}\rightarrow \mm_{0, 24}$, forgetting the marked point labeled by $25$. The following formulas are well-known, see e.g., \cite{FG}:
\begin{proposition}\label{fg}
$(1) \ c_1(\omega_{\pi})=\psi_{25}-\sum_{i=1}^{24} \delta_{0:i, 25}\in CH^1(\mm_{0, 25}).$
$$ (2) \ \sum_{i=1}^{24} \psi_i= \sum_{i=2}^{12}\frac{i(24-i)}{23} [B_i] \in CH^1(\mm_{0, 24}); \ \  \mbox{ }(3) \ \kappa_1= \sum_{i=2}^{12}\frac{(i-1)(23-i)}{23} [B_i].$$
\end{proposition}

We now find a boundary expression for the Hodge class at the level of $\hh$.
\begin{theorem}\label{lam}
The Hodge class at the level of $\hh$ is given by the following formula:
$$\lambda = \sum_{i=2}^{12}\sum_{\mu\in \mathcal{P}_{i}}\frac{1}{12} \mathrm{lcm}(\mu) \Bigl(\frac{9i(24-i)}{23}-27+\frac{1}{\mu}\Bigr) [E_{i:\mu}].$$
\end{theorem}
Note that a  boundary formula for $\lambda$ in the case of $S_n$-covers appeared first in \cite{KKZ} and was confirmed later with algebraic methods in \cite{GK}.

\begin{proof}
Over the Hurwitz space $\hh$ we consider the universal $E_6$-admissible cover $f:\cC\rightarrow \cP$, where $\cP:=\hh\times_{\mm_{0, 24}} \mm_{0, 25}$ is the universal
\emph{orbicurve} of genus zero over $\hh$. Over a general point $t=[C\rightarrow R, p_1, \ldots, p_{24}]$ of a boundary divisor $E_{i:\mu}$, where $\mu=(\mu_1, \ldots, \mu_{\ell})\in \mathcal{P}_i$ corresponds to the local description (\ref{localring}), even though $P$ has a singularity of type $A_{\mathrm{lcm}(\mu)-1}$, the space $\cC$ has singularities of type $A_{\mathrm{lcm}(\mu)/\mu_i-1}$ at the $\ell$ points corresponding to the inverse image of $R_{\mathrm{sing}}$.
\vskip 3pt

Let $\phi:\cP\rightarrow \hh$ and $\oq:\cP\rightarrow \mm_{0, 25}$ be the two projections and put $v:=\phi\circ f:\cC\rightarrow \hh$ and $\of:=\oq\circ f:\cC\rightarrow \mm_{0, 25}$.
The ramification divisor of $f$ decomposes as $R_1+\cdots +R_{24} = R \subset \cC$, where a general point of $R_i$ is of the form
$[C\rightarrow R, \ p_1, \ldots, p_{24}, x]$, with $x\in C$ one of the six ramification points lying over the branch point $p_i$. In particular $f_*([R_i])=6[\mathfrak B_i]$, where $\mathfrak B_i\subset \cP$ is the corresponding branch divisor.

\vskip 3pt
We apply the Riemann-Hurwitz formula for $f$ and write: $c_1(\omega_v)=f^*\oq^* c_1(\omega_{\pi})+[R]$. We are going to push-forward via $v$ the square of this identity and describe all the intervening terms in the process. Over $\hh$ we have the identity:
$$v_* c_1^2(\omega_v)=v_*\Bigl(\of^*c_1^2(\omega_{\pi})+2\of^* c_1(\omega_{\pi})\cdot [R]+[R]^2\Bigr).$$
We evaluate each term: $v_*\Bigl(\of^*c_1(\omega_{\pi})\cdot [R]\Bigr)=$$$\sum_{i=1}^{24} \phi_*\Bigl(\oq^* c_1\left(\omega_{\pi}\right)\cdot 6[\mathfrak B_i]\Bigr)=6\sum_{i=1}^{24}\phi_*\oq^*\Bigl(c_1(\omega_{\pi})\cdot [\Delta_{0:i, 25}]\Bigr)=6\mathfrak{b}^*\Bigl(\sum_{i=1}^{24} \psi_i\Bigr).$$
Furthermore, we write $f^*(\mathfrak{B}_i)=2R_i+A_i$, where the residual divisor $A_i$ defined by the previous equality  maps $15:1$ onto $\mathfrak{B}_i$. Note that $A_i$ and $R_i$ are disjoint,
hence $f^*([\mathfrak{B_i}])\cdot R_i=2R_i^2$, therefore
$$v_*([R_i]^2)=3\phi_*([\mathfrak{B}_i^2])=3\phi_*(\oq^*\bigl(\delta_{0:i, 25}^2)\bigr)=-3\mathfrak{b}^*(\psi_i).$$
Using Proposition \ref{fg}, we find that
$$v_*([R]^2)=v_*\Bigl(\sum_{i=1}^{24} [R_i]^2\Bigr)\equiv -3\sum_{i=2}^{12}\frac{i(24-i)}{23} \mathfrak{b}^*(B_i).$$
\vskip 3pt
We use Proposition \ref{fg}, and the relation $\pi_*(\delta_{0:i, 25}^2)=-\psi_i$, to write:
$$v_*\of^*c_1^2(\omega_{\pi})=\phi_*\Bigl(27 \oq^*c_1^2(\omega_{\pi})\Bigr)=27\mathfrak{b}^*\pi_*\Bigl(\psi_{25}-\sum_{i=1}^{24} \delta_{0:i, 25}\Bigr)^2=$$
$$27\mathfrak{b}^*\Bigl(\kappa_1-\sum_{i=1}^{24} \psi_i\Bigr)\equiv -27\mathfrak{b}^*\Bigl(\sum_{i=2}^{12} B_i\Bigr).$$
We find the following expression for the pull-back of the Mumford $\kappa$ class to $\hh$:
\begin{equation}\label{mumf1}
v_*c_1^2(\omega_v)\equiv \sum_{i=2}^{12} \Bigl(\frac{9i(24-i)}{23}-27\Bigr)\mathfrak{b}^*(B_i)\equiv \sum_{i=2}^{12}\sum_{\mu\in \mathcal{P}_i} \mathrm{lcm}(\mu)\Bigl(\frac{9i(24-i)}{23}-27\Bigr)E_{i:\mu}.
\end{equation}
Using Mumford's  GRR calculation in the case of the universal genus $46$ curve $v:\cC\rightarrow \hh$, coupled with the local analysis of the fibres of the map $\mathfrak{b}$,
we have that
$$12\varphi^*(\lambda)\equiv v_*c_1^2(\omega_v)+\sum_{i=2}^{12}\sum_{\mu\in \mathcal{P}_i} \mathrm{lcm}(\mu_1, \ldots, \mu_{\ell})\Bigl(\frac{1}{\mu_1}+\cdots+\frac{1}{\mu_{\ell}}\Bigr) E_{i:\mu}.$$
Substituting in (\ref{mumf1}), we finish the proof.
\end{proof}

\begin{remark} Using Definition \ref{divunord}, we spell out Theorem \ref{lam} at the level of $\hhu$:
\begin{equation}\label{lamhhu}
\lambda=\frac{33}{46}[D_0]+\frac{7}{46} [D_{\mathrm{azy}}]+\frac{17}{46} [D_{\mathrm{syz}}]+\cdots \in CH^1(\hhu).
\end{equation}
\end{remark}

\begin{proposition}\label{dzero}
The morphism $\varphi:\hh\rightarrow \mm_{46}$ has ramification of order $12$ along the divisor $E_0$. In particular, the class
$\varphi^*(\delta_0)-12[E_0]-2[E_{\mathrm{syz}}]\in CH^1(\hh)$ is effective.
\end{proposition}
\begin{proof}
The morphism $\varphi$ factors via $\overline{\Hur}$, that is, $\varphi=\ophi\circ q$, where we recall that $q:\hh\rightarrow \overline{\Hur}$ is the projection map and $\ophi:\overline{\Hur}\rightarrow \mm_{46}$. We have observed that $q$ is ramified along $E_0$. Furthermore, since the general element of $\varphi(E_0)$ is an irreducible $6$-nodal curve, the local intersection number $\bigl(\ophi(\Gamma)\cdot \delta_0\bigr)_{\ophi(t)}$, for any curve $\Gamma \subset \overline{\Hur}$ passing through a point $t\in q(E_0)$, is at least equal to $6$. Finally, $[E_{\mathrm{syz}}]$ appears with multiplicity $2$ because, as pointed out in Remark \ref{aut1}, each point of $E_{\mathrm{syz}}$ has an automorphism of order $2$.
\end{proof}
\end{num}

\subsection{The positivity of the canonical class of $\hh$}

\begin{num}\label{moriwaki3} To establish the bigness of the class $K_{\hh}$, we use Moriwaki's class \cite{Mo98}
$$\mathfrak{mo}:=(8g+4)\lambda-g\delta_0-\sum_{i=1}^{\lfloor \frac{g}{2}\rfloor} 4i(g-i)\delta_i \in CH^1(\mm_g).$$
It is shown in \cite{Mo98} that $\mathfrak{mo}$ non-negatively intersects all complete curves in $\mm_g$ whose members are stable genus $g$ curves with at most
one node. Furthermore, the rational map $\phi_{n\cdot \mathfrak{mo}}:\mm_g\dashrightarrow \bP^{\nu}$ defined by a linear system $|n\cdot \mathfrak{mo}|$ with $n\gg 0$, induces a \emph{regular} morphism on $\cM_g$. In our situation when $g=46$, this implies that the pull-back $\varphi^*(\mathfrak{mo})$ is an \emph{effective}
$\mathbb Q$-divisor class on $\hh$, which we shall determine.
In what follows, if $D_1$ and $D_2$ are divisors on a normal variety $X$, we write $D_1\geq D_2$ if $D_1-D_2$ is effective.
\begin{proposition}\label{scaling}
The following divisor class on the Hurwitz space $\hh$ is effective:
$$-\frac{2}{23}E_0+\frac{523}{2415}E_{\mathrm{syz}}+\frac{62}{115}E_{\mathrm{azy}}+\sum_{i=3}^{12}\sum_{\mu\in \mathcal{P}_i }\frac{93}{1610} i(24-i)\mathrm{lcm}(\mu) E_{i:\mu}$$
\end{proposition}
\begin{proof} We give a lower bound for the coefficient of $E_{i:\mu}$ in the expression $\varphi^*(\lambda)$ of Theorem \ref{lam}, by observing that for a partition $(\mu_1, \ldots, \mu_{\ell})\vdash 27$, the inequality $\frac{1}{\mu_1}+\cdots+\frac{1}{\mu_{\ell}}\leq 27$ holds. Using this estimate together with  Theorem \ref{lam} \  $\varphi^*(\lambda)=\frac{33}{23}[E_0]+\frac{17}{46}[E_{\mathrm{syz}}]+\frac{7}{23}[E_{\mathrm{azy}}]+\cdots$, as well as Proposition \ref{dzero}, we write
$$
0\leq \frac{1}{210}\varphi^*(\mathfrak{mo})\leq \frac{372}{210}\varphi^*(\lambda)-\frac{46\cdot 12}{210} [E_0]-\frac{46\cdot 2}{210}[E_{\mathrm{syz}}]=$$
$$-\frac{2}{23}[E_0]+\frac{523}{2415} [E_{\mathrm{syz}}]+\frac{62}{115} [E_{\mathrm{azy}}] +\sum_{i=3}^{12}\sum_{\mu\in \mathcal{P}_i}\frac {93}{1610}i(24-i)\mathrm{lcm}(\mu) [E_{i:\mu}].
$$
The scaling has been chosen to match the negative $E_0$ coefficient in the class $K_{\hh}$ of (\ref{canhur}).
\end{proof}
\vskip 3pt

As a step towards determining the Kodaira dimension of $\overline{\Hur}$
we establish the following:
\begin{theorem}\label{kodhh}
The canonical class of $\hh$ is big.
\end{theorem}

\begin{proof} Recalling that $\mathfrak{b}:\hh\rightarrow \overline{\mathcal{M}}_{0,24}$, for each $0<\alpha <1$, using (\ref{canhur}) we write the equality
$$K_{\hh}=(1-\alpha)\mathfrak{b}^*(\kappa_1)+\alpha \mathfrak{b}^*(\kappa_1)-\sum_{i=2}^{12}\sum_{\mu\in \mathcal{P}_i} E_{i:\mu}.$$
Since the class $\kappa_1\in CH^1(\mm_{0, 24})$ is well-known to be
ample, in order to establish that $K_{\hh}$ is big, it suffices to
show that for $\alpha$ sufficiently close to $1$, the class
$\alpha \mathfrak{b}^*(\kappa_1)-\sum_{i, \mu\in \mathcal{P}_i}
[E_{i:\mu}]$ is effective. After brief inspection, this turns out to
be a consequence of Proposition \ref{scaling}.
\end{proof}
\end{num}

\section{The Prym-Tyurin map along the boundary components of $\hhu$}
\label{secboundary}

In this section we study the extended Prym-Tyurin map and refine the
analysis of the boundary divisors of $\hhu$. In particular we identify
the divisors that are not contracted by the extended Prym-Tyurin map
 $PT:\hhu\dashrightarrow \overline{\cA}_6$.

\begin{num} \label{subsec-not} Following ~\ref{subsec-boundcomb}, we denote by $E_{I: L_1, L_2, \mu}$ the divisor of $\hh$ of $E_6$-admissible covers
$$t:=[\pi:C := C_1 \cup C_2 \ra R_1 \cup_q R_2, p_1, \ldots,
p_{24}],$$
where $I\cup J=\{1, \ldots, 24\}$,
$R_1$ contains the branch points $\{p_i\}_{i\in I}$,
with roots $\{r_i\}_{i\in I}$ generating the lattice $L_1\subset E_6$
and the corresponding reflections generating the group $G:= W(L_1)
\subset W(E_6)$, whereas $R_2$ contains the branch points
$\{p_j\}_{j\in J}$, with roots
$\{r_j\}_{j\in J}$ generating the lattice $L_2\subset E_6$ and
reflections generating the group $H:= W(L_2) \subset W(E_6)$
respectively. We set
$u:=\prod_{i\in I} w_i$, therefore $u^{-1}=\prod_{j\in J} w_j$. As
before, $\mu\vdash 27$ is the partition corresponding to the cycle
type of $u\in S_{27}$ which describes the fibre $\pi^{-1}(q)$. Let
$O_G$ (respectively $O_H$) denote the set of orbits of $G$
(respectively $H$) on the set $\overline{27}:=\{1, \ldots, 27\}$. In
particular, there is a bijection between $O_G$ (respectively $O_H$)
and the set of irreducible components of $C_1$ (respectively
$C_2$). Returning to the notation in \ref{subsec-boundcomb}, for
$\mu\in \mathcal{P}_{i}$ we write
$E_{i:\mu}=\sum_{|I|=i, L_1, L_2} E_{I: L_1, L_2, \mu}$, the sum being
taken over sublattices $L_1$ and $L_2$ of $E_6$ as above.
\end{num}


\begin{definition} \label{def6.4}
Let $u\in W(E_6)$ and let $A=A_1\sqcup \dotsc \sqcup A_a$ and
$B=B_1\sqcup \dotsc \sqcup B_b$ be two $u$-invariant partitions of
the set $\overline{27}$. We define the graph $\Gamma(u,A,B)$ to be the following
bipartite graph:
\begin{enumerate}
\item The vertices are $A_1, \ldots, A_a$ and $B_1, \ldots, B_b$ respectively.
\item The edges correspond to cycles $C_k$ in the cyclic
representation of $u\in S_{27}$, including
cycles of length $1$.
\item For each cycle $C_k$, there exist unique vertices $A_i$ and $B_j$ containing
the set $c_k$. Then the edge $C_k$ joins $A_i$ and $B_j$.
\end{enumerate}

When both partitions $A$ and $B$ are trivial, that is each consists of the single set $\overline{27}$, we set $\Gamma_u:=\Gamma(u,\overline{27},\overline{27})$ and
  $\Gamma_1:=\Gamma(1,\overline{27},\overline{27})$ respectively.
\end{definition}

\begin{example}
  The graph $\Gamma_1$ has $2$ vertices and $27$ edges. One has
  $C_1(\Gamma_1, \mathbb Z)=\mathbb Z^{27}$, and
  $H_1(\Gamma_1, \mathbb Z)\simeq\mathbb Z^{26}$ consists of elements
  $\sum_{s=1}^{27} n_s e_s$ with $\sum_{s=1}^{27} n_s=0$. There is a
  natural degree $10$ homomorphism
  $D_{C_1}\colon C_1(\Gamma_1, \mathbb Q)\to C_1(\Gamma_1, \mathbb Q)$
  with eigenvalues $10,1,-5$, which induces a homomorphism
  $D_{H_1}\colon H_1(\Gamma_1, \mathbb Q)\to H_1(\Gamma_1, \mathbb Q)$ with
  $(+1)$-eigenspace of dimension 20 and $(-5)$-eigenspace of dimension
  $6$ respectively.
\end{example}

In particular, for the dual graph $\Gamma$ of $C$, the group $H_1(\Gamma,\bZ)$ comes with an endomorphism $D_{H_1}$
by Lemma \ref{lem:parts-of-pt}.

\begin{theorem}\label{thm:toric-rank}
  Let $t \in E_{I,J:L_1,L_2}$ be a general point in a boundary divisor
  of $\hh$ corresponding to the above data. Then the toric rank of the
  PTK variety $PT(C,D)$ equals the dimension of the
  $(-5)$-eigenspace
  $H_1\bigl(\Gamma(u,O_G,O_H), \mathbb Q\bigr)^{(-5)}$ of
  the endomorphism $D_{H_1}$ on   $H_1\bigl(\Gamma(u,O_G,O_H), \mathbb
  Q\bigr)$.
\end{theorem}
\begin{proof}
  By Lemma~\ref{lem:parts-of-pt}, the toric rank of $PT(C,D)$ equals
  $$\rank\im (D_{H_1}-1) = \rank\ker (D_{H_1}+5) = \dim (H_1\otimes\bQ)^{(-5)}.$$
\end{proof}

In case both curves $C_1$ and $C_2$ are irreducible, the above result simplifies considerably.

\begin{corollary} \label{cor:toric-rank}
Assume that $|O_G|=|O_H|=1$, that is, both groups $G$ and $H$ act
transitively on the set $\overline{27}$. Then the toric rank of $PT(C,D)$ equals the
dimension of invariant subspace of $u$ in the $6$-dimensional
representation $E_6\otimes \mathbb Q$ of $W(E_6)$.
\end{corollary}

Corollary~\ref{cor:toric-rank}
agrees with the result of \cite[p.236]{lange2008a-galois-theoretic-approach} concerning the abelian part of $PT(C,D)$.

\begin{lemma}
For $u\in W(E_6)$, the following statements hold:
\begin{enumerate}
\item $H_1(\Gamma_u,\mathbb Q)^{(-5)} = \left( H_1(\Gamma_1,\mathbb Q)^{(-5)} \right)^u$
(that is, the $u$-invariant subspace), and
\item $H_1(\Gamma(u,A,B),\mathbb Q)^{(-5)} \subset H_1(\Gamma_u,\mathbb Z)^{(-5)}$.
  \end{enumerate}
\end{lemma}
\begin{proof}
Suppose we have the following cycle decomposition $u=C_1\cdot C_2 \cdots \cdot C_k\in S_{27}$ and let $n:=\mbox{ord}(u)$ and $\ell(C_i)$ denote the length of $C_i$. We write $C_1(\Gamma_u, \mathbb Z)=\bigoplus_{i=1}^k \mathbb Z e_{C_i}$. Then one has an orthogonal projection
$C_1(\Gamma_1,\mathbb Z)\onto C_1(\Gamma_u,\mathbb Z)$ given by
$e \mapsto \frac1{n}\sum_{i=0}^{n-1} u^i\cdot e$ for an edge $e$,
which identifies $C_1(\Gamma_u,\mathbb Z)$ with a sublattice in
$C_1(\Gamma_1,\mathbb Q)$ via the injection $e_{C_i}\mapsto \frac{1}{\ell(C_i)} \sum_{j\in C_i} e_j$. This induces a surjection from
  $H_1(\Gamma_1,\mathbb Z)$ to $H_1(\Gamma_u,\mathbb Z)$, which clearly commutes
  with $D$, that is, $D(C_1(\Gamma_u, \mathbb Z))\subset C_1(\Gamma_u, \mathbb Z)$.  It follows that $H_1(\Gamma_u,\mathbb Z)^{(-5)}$ is the projection
  of $H_1(\Gamma_1,\mathbb Z)^{(-5)}$ to the $(-5)$-eigenspace in
  $C_1(\Gamma_u,\mathbb Q)$ and that
   $ H_1(\Gamma_u,\mathbb Q)^{(-5)} = \left( H_1(\Gamma_1,\mathbb Q)^{(-5)} \right)^u$.

The graph $\Gamma(u,A,B)$ is obtained from
$\Gamma_u$ by splitting the two
vertices into $a+b$ new vertices. This can be obtained by inserting
in place of the two vertices two trees with $a$ and $b$ vertices
  -- without changing $H_1$ -- and then removing the edges of these
  trees. Thus, one has an inclusion
\begin{math}
H_1(\Gamma(u,A,B),\mathbb Z) \subset
H_1(\Gamma_u,\mathbb Z )     ,
\end{math}
commuting with $D$, which gives an inclusion
$H_1(\Gamma(u,A,B), \mathbb Z)^{(-5)} \subset H_1(\Gamma_u,\mathbb Z )^{(-5)}$.
\end{proof}

\begin{lemma}
The $(-5)$-eigenspace $H_1(\Gamma(u,A,B),\mathbb Q)^{(-5)}$ is a subspace of
the $u$-invariant subspace $(E_{6}\otimes\mathbb Q)^u$ in the standard
$6$-dimensional $W(E_6)$-representation.
\end{lemma}
\begin{proof}
Indeed, $H_1(\Gamma_1,\mathbb Q)^{(-5)} = E_6\otimes\mathbb Q$, therefore
$\left( H_1(\Gamma_1, \mathbb Q)^u \right)^{(-5)} = (E_6\otimes\mathbb Q)^u$.
\end{proof}

\vskip 3pt
\begin{num} In order to illustrate Theorem~\ref{thm:toric-rank} in concrete situations, we classify all root sublattices of $E_6$.  Recall that with the notation of Section \ref{sec:e6-lattice}, the standard roots in the $E_6$ Dynkin diagram are
$r_2=\alpha_{12}, \ldots, r_6=\alpha_{56}$, and
$r_1=\alpha_{123}$. In the extended Dynkin diagram $\tE_6$ there is an additional root
$r_0=-\alpha_{\max}$, so that
$3r_4+2r_1+2r_3+2r_5+r_2+r_6+r_0=0$.
\end{num}

\begin{lemma}\label{lem-sublattices}
The following is the complete list of root sublattices $L \subset E_6$:
\begin{enumerate}
\item If $\dim(L) =6$, then $L$ is either $E_6$, or isomorphic to $A_5A_1$, or $A_2^3$.
\item If $\dim(L)=5$, then $L$ is isomorphic to $A_5$, $D_5$, $A_4A_1$, $A_3A_1^2$, or $A_2^2A_1$.
\item If $\dim(L)=4$, then $L$ is isomorphic to $A_4$, $D_4$, $A_2^2$, $A_3 A_1$, $A_2 A_1^2$, or $A_1^4$.
\item If $\dim(L)=3$, then $L$ is isomorphic to $A_3$, $A_2A_1$, or $A_1^3$.
\item If $\dim(L)=2$, then $L$ is isomorphic to $A_2$, or $A_1^2$.
\item If $\dim(L)=1$, then $L$ is isomorphic to $A_1$.
\end{enumerate}
Furthermore, all the above sublattices (and the associated subgroups) can be obtained by removing vertices from the extended $E_6$ diagram $\tE_6$:
\[
 \xymatrix{
      \stackrel{r_2}{\bullet} \ar@{-}[r] &  \stackrel{r_3}{\bullet}   \ar@{-}[r] &  \stackrel{r_4}{\bullet}  \ar@{-}[r] \ar@{-}[d]
      &  \stackrel{r_5}{\bullet}  \ar@{-}[r]  &  \stackrel{r_6}{\bullet}  \\
       & & \quad  {\bullet}_{r_1}  \ar@{-}[d] &   &  \\
       & & \quad  {\star}_{r_0}. &   &
    }
\]
If the root lattices $L,L'$ corresponding to reflections subgroups $G$
and $G'$ of $W(E_6)$ are isomorphic, then they differ by an
automorphism of the $E_6$ lattice, and the corresponding
subgroups $G$ and $G'$ are conjugate in $W(E_6)$.
\end{lemma}
\begin{proof}
We first note that there is a natural bijection between root sublattices $L$ of $E_6$ and subgroups $G$
generated by reflections of $W(E_6)$. One has $\Aut(E_6)=W(E_6)\oplus\bZ_2$, with $\bZ_2$ acting on $E_6$ by multiplication by $\pm1$.  Any
automorphism of $E_6$ induces an automorphism of $W(E_6)$, and
the kernel of $\phi\colon\Aut(E_6)\to \Aut W(E_6)$ is
$\bZ_2$. Finally, by \cite[Section 2.3]{Franzsen-thesis2001}), all
automorphisms of $W(E_6)$ are inner, so that $\Aut W(E_6)=W(E_6)$
and $\phi$ is surjective.

Thus, the proof reduces to showing that all root sublattices of
$E_6$ are of the above types, and that if $L,L'$ are isomorphic as
abstract root lattices, then they differ by an element of
$\Aut(E_6)$. The statement that all such root sublattices correspond
to proper subdiagrams of the extended Dynkin diagram
$\widetilde{E}_6$ is an \emph{a posteriori} observation.

The standard method for finding all root sublattices of a given root
lattice is described in \cite{BS,dy}.  A modern treatment can be found in
\cite[Theorem 1]{DyerLehrer2011}. The method is to repeatedly apply
the following two procedures to Dynkin diagrams $\Gamma$, starting
from $\Gamma=E_6$: (1) remove a node, and/or (2) replace one of the
connected components $\Gamma_s$ of $\Gamma$ by an extended Dynkin
diagram $\widetilde{\Gamma}_s$ and remove a node from it. Applying
the above method repeatedly, we obtain all the lattices listed
above. The fact that isomorphic root sublattices differ by an
automorphism of $E_6$ is a case by case computation. This  can also be found in \cite[Table
  10.2]{oshima2006a-classification-of-subsystems}.
\end{proof}

\begin{num} \label{subsec-roots-diag} Table~\ref{tab:orbits} lists the orbits for one choice of roots (the other choices being similar) for each type of lattice. The last column describes the degrees of the maps from the irreducible components of $C_1$ to $R_1$. We keep the Schl\"afli
notation $a_i, b_i, c_{ij}$ for the elements of the set $\overline{27}$, which is being identified with the set of lines of a cubic surface.
The smooth (possibly disconnected) curve $C_1$ is a $27$-sheeted cover of $R_1=\bP^1$, with branch points $\{p_i\}_{i\in I}$ with local
monodromy given by the reflection $w_i$, and an additional  branch point $q$, with local monodromy $u^{-1}$, where $u=\Pi_{i\in I} w_i$.

{\small
\begin{table}[htbp]
  \centering
\begin{tabular}{|c|c|c|c|}
 \hline
Sublattice & Roots & Orbits  & Degrees \\
\hline
$E_6$ &  $r_1, \ldots, r_6$  &    {\small $\{a_i, b_i, c_{ij}\}$ }   & 27\\
$A_5A_1$ & $r_0, r_2, \ldots, r_6$ &   {\small $\{a_i,b_i\}$, $\{c_{ij}\}$ }   &15, 12 \\
$A_2^3$ &  $r_i, i\neq 4$  &  {\Small $\{a_i, b_i, c_{ij} \mid 1\leq i,j \leq 3\}$,  $\{a_i, b_i, c_{ij}\mid 4\leq i,j \leq 6\}$,}  &  $9^3$  \\
& & {\small $\{c_{ij} \mid 1\leq i \leq 3, 4\leq j \leq 6 \}$ } & \\
$D_5$ &  $r_1, \ldots, r_5$  & {\small $\{a_6 \}$, $\{ a_i, b_6, c_{ij} \mid 1\leq i, j\leq 5 \}, \{b_i, c_{i6} \mid 1\leq i \leq 5 \}$ } & 1, 10, 16\\
$A_5$ &  $r_2, \ldots, r_6$  & {\small $\{a_i \}, \{b_i\}, \{c_{ij}\}$ } & $6^2,15$ \\
$A_4A_1$ & $r_0, r_2, \ldots, r_5$ & {\small $\{ a_i, c_{ij} \mid 1\leq i,j \leq 4\}, \{b_i , c_{56} \mid 1\leq i\leq 4\}$ } & $2, 5, 10^2$ \\
& & {\small $\{ a_5, a_6\}, \{b_j, c_{ij} \mid 1\leq i \leq 4, 5 \leq j \leq 6\} $ } & \\
$A_3A_1^2$ & $r_0, r_2, r_3, r_4, r_6$ &  {\Small $\{a_1, b_i, c_{ij}, c_{56} \mid 2\leq i,j \leq 4 \}, \{ b_1 \}, \{a_i, c_{ij}  \mid 2\leq i,j \leq 4 \}$ } &  $8^2, 6,4,1$\\
& & {\Small $\{a_5, a_6, c_{15}, c_{16} \}, \{ b_j,c_{ij} \mid 2 \leq i \leq 4, 5 \leq j \leq 6  \}$ } & \\
$A_2^2A_1$& $r_i, i\neq 0,4$ &  {\Small $\{ a_i, c_{ij} \mid 1\leq i,j \leq 3 \}, \{ b_i, c_{ij} \mid 4\leq i,j \leq 6 \}, \{a_4, a_5, a_6 \}$ } & $9, 6^2, 3^2 $  \\
& & {\Small $\{b_1, b_2, b_3 \}, \{ c_{ij} \mid 1\leq i \leq 3, 4\leq j \leq 6 \}$ } & \\
$A_4$ & $r_2, \ldots, r_5$ &  {\Small $\{ a_i, c_{ij} \mid 1\leq i,j \leq 4 \}, \{ b_i, c_{56} \mid 1\leq i \leq 4 \}, \{a_5 \}, \{a_6\} $ }   & $10, 5^3, 1^2$ \\
 & & {\Small $\{ b_5, c_{i6} \mid 1\leq i \leq 4 \}, \{b_6, c_{i5}\mid 1\leq i \leq 4 \}$ } & \\
$D_4$ & $r_1, r_3, r_4, r_5$ &  {\Small $\{a_1, c_{ij}, b_6 \mid 2\leq i,j \leq 5\}, \{ a_i, c_{1i} \mid 2\leq i \leq 5 \}, \{ a_6\}$ }   & $8^3, 1^3$\\
& &  {\Small $\{ b_1 \}, \{ b_i, c_{i6} \mid 2\leq i \leq 5 \}, \{ c_{16}\}$ } & \\
$A_2^2$ & $r_2, r_3, r_5, r_6$ &  {\Small $\{a_1, a_2, a_3 \} , \{b_1, b_2, b_3 \}, \{a_4, a_5, a_6 \}, \{b_4, b_5, b_6 \}$ }   & $9, 3^6$\\
& & {\Small $\{ c_{12}, c_{13}, c_{23} \}, \{ c_{45}, c_{46}, c_{56} \}, \{ c_{ij} \mid 1\leq i \leq 3, 4\leq j \leq 6 \}$} & \\
$A_3A_1$ & $r_2, r_3, r_4, r_6$ &  {\Small $\{c_{56}\}, \{a_5, a_6\}, \{b_5, b_6\}, \{a_1, \ldots, a_4\}, \{b_1, \ldots, b_4\}$ }   & $8, 6, 4^2, 2^2, 1$\\
& & {\Small $\{ c_{ij} \mid 1\leq i,j \leq 4\}, \{c_{ij} \mid 1\leq i \leq 4, 5\leq j \leq 6\} $} & \\
$A_2 A_1^2$ & $r_1, r_2, r_3, r_5$ &  {\Small $\{a_6 \}, \{b_6, c_{45} \}, \{b_1, b_2, b_3 \}, \{ c_{16},  c_{26}, c_{36} \}, \{ b_5, b_6, c_{64}, c_{65}\}$ }   & $6^2, 4,  3^2, 2^2, 1 $\\
& & {\Small $\{a_4, a_5 \}, \{ a_i, c_{ij} \mid 1\leq i,j \leq 3\}, \{c_{ij} \mid 1\leq i\leq 3, 4\leq j\leq 5\}$ }  & \\
$A_1^4$ & $r_0, r_2, r_4, r_6$ &  {\Small $\{a_6 \}, \{b_1\}, \{a_1, b_6, c_{23}, c_{45} \}, \{a_2, a_3, c_{12}, c_{13} \}, \{ a_4, a_5, c_{14}, c_{15}  \} $  }  & $4^6, 1^3$ \\
& & {\Small $\{ b_2, b_3, c_{26}, c_{36}\}, \{ b_4, b_5, c_{46}, c_{56}\}, \{ c_{24}, c_{34}, c_{25}, c_{35} \}, \{c_{16}\}$} & \\
$A_3$ & $r_2, r_3, r_4$ &  {\Small $\{  c_{12},  c_{13}, c_{14}, c_{23},  c_{24}, c_{34} \}, \{ c_{15},  c_{25}, c_{35}, c_{45} \}, \{ c_{16},  c_{26}, c_{36}, c_{46} \}$ }   & $6, 4^4, 1^5 $\\
& & {\Small $\{a_1, a_2, a_3,   a_4 \}, \{b_1, b_2, b_3,b_4 \}, \{a_i\}, \{b_i\}, 5\leq i \leq 6, \{c_{56} \}$} & \\
$A_2A_1$ & $r_1, r_2, r_3$ &  {\Small $\{b_1, b_2, b_3 \}, \{c_{14}, c_{24}, c_{34} \}, \{ c_{15},  c_{25}, c_{35}\}, \{a_1, a_2, a_3, c_{12},  c_{13}, c_{23}\}$ }   & $6, 3^4, 2^3, 1^3$\\
& & {\Small $\{ c_{16},  c_{26}, c_{36}  \}, \{b_j , c_{kl} \}, \{j,k,l\} = \{4,5,6\}, \{ a_i \}, 4\leq i \leq 6$} & \\
$A_1^3$ & $r_2, r_4, r_5$ &  {\Small $\{c_{13}, c_{14}, c_{23}, c_{24} \}, \{c_{15}, c_{16}, c_{25}, c_{26} \}, \{c_{35}, c_{36}, c_{45}, c_{46} \}$ }   & $4^3, 2^6, 1^3$\\
& & {\Small $\{a_i, a_{i+1}\}, \{b_i, b_{i+1}\}, \{c_{i i+1}\}, i=1,3,5$} & \\
$A_2$ & $r_2, r_3$ &  {\Small $\{a_1, a_2, a_3\}, \{b_1, b_2, b_3\}, \{c_{12}, c_{13}, c_{23} \}, \{c_{14}, c_{24}, c_{34} \}$ }   & $ 3^6,  1^9$\\
& & {\Small $\{c_{15}, c_{25}, c_{35} \}, \{c_{16}, c_{26}, c_{36} \}, \{a_i\}, \{b_i\}, \{c_{ij}\}, 4\leq i,j \leq 6$} & \\
$A_1^2$ & $r_2, r_4$ &  {\Small $\{c_{13}, c_{23}, c_{14}, c_{24} \}, \{c_{56}\}, \{a_i, a_{i+1}\}, \{b_i, b_{i+1}\}, $ }   & $ 4, 2^8,  1^7$\\
& & {\Small $\{c_{ij}, c_{i+1j}\}, \{a_j\}, \{b_j\}, \{c_{ii+1}\}, i=1,3, j=5,6$ } & \\
$A_1$ & $r_0$ &  {\Small $\{a_i, b_i\}, \{c_{ij}\}, 1\leq i, j \leq 6$} & $2^6, 1^{15}$ \\
\hline
\end{tabular}
\medskip
  \caption{Sublattices and Orbits}
  \label{tab:orbits}
\end{table}

}

We apply Theorem~\ref{thm:toric-rank} to compute the toric ranks associated to the  divisors
\[
E_L:=\sum_{|I|=22} E_{I; L, A_1, (1^{27})}\subset \hh.
\]
Since $\mbox{dim}(L)\geq 5$, using Lemma \ref{lem-sublattices}, we have the following possibilities:
\[
L\in \{E_6, A_5A_1,A_2^3, A_5,D_5,A_4A_1, A_3A_1^2,A_2^2A_1\}.
\]

\begin{proposition}\label{prop-toricrank}
The toric rank of each boundary divisor $E_L$ with $L\neq E_6$ is equal to zero. The toric rank of $E_{E_6}$ is equal to $1$.
\end{proposition}

\begin{proof}
Note that there are $|O_H|=21$ vertices on the right, of which $15$ vertices are ends and thus can be removed without changing the homology of the graph. The remaining $6$ vertices each have degree $2$. Contracting unnecessary edges, we reduce the calculation to a graph with $|O_G|$ vertices and $6$ edges. The $6$ edges correspond to the $6$ transpositions appearing in the decomposition of $w_{23}=w_{24}\in S_{27}$.

Assume $|O_G|=1$, which, by Table~\ref{tab:orbits}, is the case if and only if  $L=E_6$. Then $H_1(\Gamma, \mathbb Z)=\bigoplus_{i=1}^6 \mathbb Z e_i$ and $D(e_i)=-\sum_{j\neq i} e_j$. Therefore
$H_1(\Gamma, \mathbb Q)^{(-5)}$ is $1$-dimensional and generated by the element $e_1+\cdots+e_6$. The other cases follow similarly by direct calculation.
\end{proof}
\end{num}

\begin{remark}
For more details concerning the calculation of the toric rank in the
case $L=D_5$, see Paragraph \ref{num:example-toric-rank}.
\end{remark}

\begin{num} Although the divisor theory of $\hh$ is quite complicated, we now show that most of these divisors are contracted under the Prym-Tyurin map.
We first establish the following:

\begin{theorem}\label{thmboundPT}
  Assume that the image of a component $B$ of $E_{I:L_1,L_2}$ under
  the rational Prym-Tyurin map $\hh\ratmap \aperf$ has codimension $1$ in
  $\aperf$. Then $\{|I|, |I^c|\}=\{2,22\}$.
\end{theorem}

\noindent {\em Proof.}
With notation as in \ref{subsec-not},
denote $P_i:=PT(C_i,D_i)$ the PTK varieties for the two parts
of $R$. Recall that by Lemma \ref{lem:semiabelian-Prym} the abelian part
$\abpt$ is isogenous to $P_1\times P_2$.

Without loss of generality, we may assume that $i:=|I|\geq 12$.
Since $\codim(\aperf\setminus \widetilde{\cA}_6)\ge2$,
if the irreducible components $R_i$, $i=1,2$, of $R$
image of $B$ has codimension $1$, then for
a general point of $B$, the toric rank $k_P$ of the corresponding
PTK variety $P$ is either $0$ or $1$.

\vskip 4pt

Suppose first that $k_P=0$. In this case in fact
$P\cong P_1\times P_2$.  If both $P_1$ and $P_2$ have positive
dimension, then $P$ belongs to a subvariety of $\cA_6$ parametrizing
products and each such subvariety has codimension greater than $1$. So
one of the $P_i$ is zero. The parameter space of $P_1$ has dimension
at most $i-2$, that of $P_2$ has dimension at most $22-i$. Since the
parameter space of $P$ is $20$-dimensional and $i\geq 12$, we have
$P_2 = 0$, and $\mbox{dim}(P_1)=6$. We deduce $i=22$.

\vskip 4pt

Now assume $k_P =1$. In this case the image of $B$ is the boundary
divisor $D_6$ of $\aperf$. Then $P_1\times P_2$ must be a
general abelian variety of dimension $5$.
Once again, one of the $P_i$ is zero. The assumption
$i\geq 12$ implies $P_2=0$, hence $\mbox{dim}(P_1)=5$. The parameter
space of $P_1$ is $15$-dimensional, which implies $i\geq 17$. Let
$\{p_1, \ldots , p_{\ell}\}=C_1\cap C_2$ be the set of the nodes of
$C$, which also label the edges of $\Gamma$. For each $i$, let $p_i'$
and $p_i''$ be the points of $C_1$ and $C_2$ respectively that we
identify to obtain $p_i$ on $C$. Choose the orientation of $\Gamma$ in
such a way that each edge $p_i$ is oriented from $C_1$ to $C_2$. The
extension
\[
0 \lra H^1(\Gamma,\bC^*) \lra JC \lra JC_1 \times JC_2 \lra 0
\]
is given by the map $\phi : H_1 (\Gamma , \bZ) \ra JC_1 \times JC_2$
sending the edge $p_i$ to $p_i''-p_i'$. We observe that the extension
\[
0 \lra \bC^* = (D_T-1)H^1(\Gamma,\bC^*) \lra P \lra \mbox{ab}(P) \lra 0
\]
is given by the map
$\phi_P : (D_{H_1}-1) H_1 (\Gamma , \bZ) \ra \mbox{ab}(P)$. Clearly
the composition of $\phi_P$ with the isogeny $\mbox{ab}(P) \to
P_1\times P_2$ is the restriction of $\phi$ to
$(D_{H_1}-1)H_1(\Gamma,\bZ)$ composed with the projection $JC_1\times JC_2\to
(D_A-1)(JC_1\times JC_2) = P_1\times P_2$.




Since $P_2=0$, this means that the extension class of $P$ does not depend on $C_2$
(up to a finite set). Therefore the moduli of $C_2$ does not produce
positive moduli for the extension class of $P$.  It follows that the
cover $C_1 \ra R_1$ depends on $20$ moduli, hence $R_1$ contains at
least $22$ branch points, therefore $i=22$.  \qed

The following result shows that the boundary divisors have many fewer irreducible components than one would a priori expect. Recall that in \ref{subsec-azy} and \ref{subsec-syz} we introduced the  divisors $E_{\mathrm{azy}}:=\sum_{|I|=22} E_{I:L_1,A_2, (3^6,1^9)}$ and $E_{\mathrm{syz}}:=\sum_{|I|=22} E_{I:L_1, A_1^2, (2^{10}, 1^7)}$ respectively.

\begin{theorem}\label{thm:syzazy}
Assume that $|I|=22$ and $L_2 = A_2$ or $A_1^2$. Then $E_{I:L_1,L_2}$ is empty unless $L_1 = E_6$. In other words, for a general $E_6$-admissible cover
$$\bigl[\pi:C = C_1 \cup C_2 \rightarrow R_1\cup_q R_2, \ p_1, \ldots, p_{24}\bigr]\: \in \: E_{\mathrm{azy}} \: \hbox{ or } \: E_{\mathrm{syz}},$$ the curve
$C_1$ is irreducible with monodromy $W(E_6)$ over $R_1$.
\end{theorem}

\begin{proof}
Consider first the azygetic case $L_2 = A_2$. Then, as we saw in \ref{subsec-azy}, the curve $C_2$ has $15$ components, each of which intersects $C_1$ in exactly one point. Therefore, no component of $C_2$ can connect two components of $C_1$ and $C_1$ is irreducible.

\vskip 3pt

In the syzygetic case, as we saw in \ref{subsec-syz}, the curve $C_2$ has $16$ components. Of the components of $C_2$, only the $4$-sheeted cover (denoted by $Z$ in \ref{subsec-syz}) intersects $C_1$ in two points. All other components of $C_2$ intersect $C_1$ in exactly one point. It follows that $C_1$ has at most two irreducible components. Looking now at Table \ref{tab:orbits}, we see that there are only two possibilities for the lattice $L_1$, namely $L_1 = E_6$ or $L_1 = A_5A_1$.

\vskip 4pt

We now eliminate the possibility $L_1 = A_5A_1$ in the syzygetic case. It is a consequence of Lemma \ref{lem-sublattices} that the $A_5$ summand of $L_1$ is the orthogonal complement of the $A_1$ summand. Hence the lattice $L_1$ and the group $G_1$ generated by the reflections $w_1, \ldots, w_{22}$ are determined by the $A_1$ sublattice. Since all reflections are conjugate, we can assume that the $A_1$ summand is generated by the reflection $w_0$ in the root $r_0$ (see \ref{subsec-roots-diag}). Since $\langle G_1, w_{23} \rangle = \langle G_1, w_{24} \rangle = W(E_6)$, the reflections $w_{23}, w_{24}$ do not belong to $G_1$. Therefore the pairs $(w_0, w_{23})$, $(w_0, w_{24})$ are azygetic. \\
After a permutation of the indices $\{1, \ldots, 6\}$, we can assume that $w_{23}$ is the reflection in the root $\alpha_{123}$ and $w_{24}$ is the reflection in the root $\alpha_{145}$ (see, e.g., \cite[Section 9.1]{dolgachev2012classical-algebraic}). The composition $w_{23}\cdot w_{24}\in S_{27}$, contains the double transposition $(a_1, b_6)(c_{23}, c_{45})$ which acts on the $4$-sheeted cover $Z$. However, $w_{23}\cdot w_{24}$ also contains the transposition $(a_2, c_{13})$ which acts on a degree $2$ component of $C_2$, i.e., the points corresponding to $a_2$ and $c_{13}$ come together over the node. Looking at the orbits of $C_1$ in Table \ref{tab:orbits} in the $A_5A_1$ case, we see that the points $a_2$ and $c_{13}$ belong to two different components of $C_1$ and cannot come together over the node, which is a contradiction.
\end{proof}
\end{num}

We now consider the components of the divisor $E_0$ introduced in \ref{subsec-E_0}. Recall that
\[
E_L:=\sum_{|I|=22} E_{I; L, A_1, (1^{27})}\subset \hh.
\]

\begin{theorem}
For $L \subsetneq E_6$ the divisor $E_L$ is contracted by $PT$.
\end{theorem}
\begin{proof}
Let $[\pi:C := C_1 \cup C_2 \ra R_1 \cup_q R_2]$ be a general element of a component $B$ of $E_L$ with $L \subsetneq E_6$. By Proposition \ref{prop-toricrank}, the toric rank of $P := PT(C,D)$ is $0$. As in the proof of Theorem \ref{thmboundPT} and with the notation there, we have $P = P_1 \times P_2 = P_1 = PT(C_1, D_1)$ because all the components of $C_2$ are rational. Furthemore, since $C_1 \ra R_1$ is not ramified at $q$, the isomorphism class of $C_1$ and hence also of $P_1$ is independent of the choice of the point $q$. It follows that $P =P_1$ depends on at most $19=\mbox{dim}(\mathcal{M}_{0,22})$ parameters, hence $B$ is contracted by $PT$.
\end{proof}

We summarize the results of this section in terms of the Hurwitz space $\hhu:=\hh/S_{24}$:

\begin{theorem}\label{contractions} The only boundary divisors of
  $\hhu$ that are not contracted under the Prym-Tyurin map
  $PT: \hhu\ratmap \overline{\cA}_6$
  are $D_{E_6}$, $D_{\mathrm{syz}}$ and $D_{\mathrm{azy}}$. The
  divisor $D_{E_6}$ maps onto the boundary divisor $D_6$ of
  $\overline{\cA}_6$, whereas $D_{\mathrm{syz}}$ and
  $D_{\mathrm{azy}}$ map onto divisors not supported on the boundary
  of $\overline{\cA}_6$.
\end{theorem}

\section{Ordinary Prym varieties regarded as Prym-Tyurin-Kanev varieties}\label{sec:PvPTK}

The aim of this section is to illustrate how $6$-dimensional Prym varieties appear as PTK varieties of type $E_6$ and thus prove Theorem \ref{prym6}.
The Prym moduli space $\cR_7$ has codimension $3$ inside $\cA_6$, where we identify $\cR_7$ with the image of the generically injective Prym map $P:\cR_7\rightarrow \cA_6$. We shall show that the boundary divisor $D_{D_5}$ of $\hhu$ is an irreducible component of $PT^{-1}(\cR_7)$ and we shall explicitly describe the $2$-dimensional fibres of the restriction
$PT_{D_{D_5}}:D_{D_5}\dashrightarrow \overline{\mathcal{R}}_7$.

\begin{num}
Consider an admissible cover $[\pi:C= C_1 \cup C_2 \ra R_1 \cup R_2]$ in the divisor $D_{D_5}$ of $\hhu$. We choose such a cover as follows.
The cover $C_1 \ra R_1$ has $D_5$-monodromy generated by the roots $r_1, \ldots, r_5$, and it is ramified at $22$ distinct points. The local monodromy at each branch point is given by one of the reflections $w_i\in W(D_5)$ associated to $r_i$, choosing the ordering such that $\prod_{i=1}^{22} w_i = 1$. The cover $C_2 \ra R_2$ has $A_1$-monodromy generated by the root $r_0$, and is branched at $2$ points. Both covers are unramified at the point $q \in R_1 \cap R_2$. As listed in Table \ref{tab:orbits}, we have the following irreducible components and orbits for $C_1$:
\begin{eqnarray*}
 F_1: &  &\{b_1, b_2, b_3, b_4, b_5, c_{16}, c_{26}, c_{36}, c_{46},
  c_{56} \}\\
 F_2: &  &\{a_1, a_2, a_3, a_4, a_5, c_{12}, c_{13}, c_{14}, c_{15},
  c_{23}, c_{24}, c_{25}, c_{34}, c_{35}, c_{45}, b_6 \}\\
 F_0: &  &\{a_6 \},
\end{eqnarray*}
and the following irreducible components and orbits for $C_2$:
\begin{eqnarray*}
1\leq i \leq 6 & H_i : & \{a_i, b_i\} \\
7\leq i \leq 21 & H_i : & \{c_{k(i)\ell(i)}\}, \hbox{ for some choice of integers }k(i) < \ell(i), \hbox{ between 1 and 6}.
\end{eqnarray*}
One computes that the three components $F_1, F_2$ and $F_0$ of $C_1$ have genera $13, 29$ and $0$, and
map onto $R_1$ with degree $10$, $16$ and $1$ respectively. The components of $C_2$ are all rational with $H_1, \ldots, H_6$ mapping $2:1$ to $R_2$ and $H_7, \ldots, H_{21}$ mapping isomorphically. The description of the orbits given above also specifies the points of intersection $F_i$ and $H_j$. For instance, $H_6$ intersects $F_2$ at a point corresponding to $b_6$ and it intersects $F_0$ at a point corresponding to $a_6$.
\end{num}

\begin{num}\label{num:example-toric-rank}
In order to compute the toric rank of the Prym-Tyurin variety $P:= (D-1)(JC_1)$, we apply the correspondence $D$ to the homology group
$H_1(\Gamma', \mathbb{Z})$, where $\Gamma'$ denotes the simplified dual graph of the stable curve $C_1 \cup C_2$ (see Section \ref{num:G-prime} for the notation).  The
graph $\Gamma'$ consists of $2$ vertices joined by $5$ edges: $e_1:=(b_1, a_1),  e_2:=(b_2, a_2), e_3:=(b_3, a_3),
e_4:=(b_4, a_4), e_5:=(b_5, a_5)$ and  $H_1(\Gamma', \mathbb{Z})= \bigoplus_{i=1}^4 \mathbb Z\ (e_i- e_{i+1})$
(see  \ref{subsec:homology-gps}).
One computes
$$
D(\partial (e_1 - e_2)) = D(a_1-b_1) - D(a_2-b_2) = b_2 - a_2 -(b_1 - a_1) = \partial (e_1 - e_2).
$$
By Remark \ref{num:corr-computation}, $D$ commutes with  $\partial$, hence $D(e_1- e_2) = (e_1 -e_2)$.
Similarly, one checks that $D(e_i- e_{i+1}) = (e_i -e_{i+1})$, for $i=1,\dots, 4$, hence $(D-1)H_1(\Gamma', \mathbb{Z})=0$.
Therefore, the Prym-Tyurin variety $P:= (D-1)(JC_1)$ has toric rank $0$ and  it is contained in $JC_1$, since $JC_2 = \{0\}$.
\end{num}

\begin{num} As is apparent from the description of the orbits, the correspondence $D$ restricts to a fixed-point-free involution $\iota:F_1\rightarrow F_1$, a correspondence
$D_{2}$ of valence $5$ on $F_2$,
a correspondence $D_{12}:F_1 \to F_2$ and its transpose $D_{21}:F_2 \to F_1$ of degree $8$ over $F_1$ and of degree $5$ over $F_2$.

The variety $P$ is the image of the following endomorphism of $JC_1 = JF_1 \times JF_2$:
\[
\left(
\begin{array}{cc}
\iota -1 & D_{21} \\
D_{12} & D_2 -1
\end{array}
\right).
\]
\end{num}
\begin{num} Let $f: F_1 \ra Y$ be the induced
unramified double cover on the curve $Y:= F_1 / \langle \iota \rangle $ of genus $7$.  Note that the degree $10$ map $\pi_1:F_1 \ra  R_1$ factors
through  a degree $5$ map $h: Y \ra R_1$.
The image $Q_1 := (\iota -1)JF_1 \subset JF_1$ is the ordinary Prym variety $P(F_1, \iota)$ associated to the double cover $[f:F_1 \ra Y]\in \cR_7$.
\end{num}

\begin{num} The relationship between the curves $F_1$ and $F_2$
(or  between the tower $F_1 \stackrel{f}{\to} Y \stackrel{h}{\to} R_1$
and the map $\pi_2:F_2 \to R_1$)
is an instance of the {\it pentagonal construction}  (\cite[Section 5.17]{donagi1992fibres-prym-map}).
This is the $n=5$ case of the $n$-gonal construction,
see \cite[Section 2]{donagi1992fibres-prym-map} or \cite[Section 1]{IzLaSt}, which applies to covers
$F_1 \stackrel{f}{\to} Y \stackrel{h}{\to} R_1$
whose Galois group is the Weyl group $W(D_n)$.
The idea is to consider the following curve inside the symmetric product $F_1^{(n)}$:
$$
h_* F_1:= \bigl\{ G \in F_1^{(n)} \ : \ \mbox{Nm}_{f} (G) =  h^{-1}(t), \textnormal { for some } t \in R_1\bigr\}.
$$
The induced map  $h_*F_1 \ra R_1$ is of degree $2^n=32$ and one checks
that above a branch point $t \in R_1$ there are exactly $2^{n-2}=8$
simple ramification points in $h_*F_1$.
\end{num}

\begin{proposition}
 $h_*F_1$ is the union of two isomorphic components
 $h_*F_1 = X_0 \sqcup X_1$, with $X_0 \simeq X_1$ being smooth curves of genus
$1+2^{n-3}(n+g(Y)-5) =29.$
\end{proposition}
\begin{proof}
The splitting is explained in \cite[Section 2.2]{donagi1992fibres-prym-map} and \cite[Section 1]{IzLaSt}. The smoothness is proved in \cite[Lemma 1.1]{IzLaSt}.
The genus calculation follows from the Hurwitz formula.
\end{proof}
Two divisors $G_1, G_2 \in h_*F_1 $ with
$\mbox{Nm}_{f}(G_1) = \mbox{Nm}_{f}(G_2)$ belong to the same component if they share an \emph{even} number of points of $F_1$.

We specialize to the case $n=5$.
Let $X=X_0$ be the component of
$h_*F_1$ whose fiber over a point $t \in R_1$ can be identified with the class of the divisor $c_{16} + \cdots +c_{56}$. The proof of the following result is immediate.
\begin{proposition}
The map
$\psi: F_2\rightarrow X$ given by   $x \mapsto  D_{21}(x) \in h_*F_1$
is an isomorphism.
\end{proposition}

\begin{remark}\label{remngonal}
Under the above identification, the restriction $D_2$ of the Kanev (incidence) correspondence coincides with the correspondence $D$ defined in \cite[Section 2]{IzLaSt}. Also, the restriction $D_{21}$ of the Kanev correspondence coincides with the correspondence $S$ defined in \cite[Section 2]{IL}. It follows from \cite[Corollary 6.2]{IzLaSt} that the image $Q_2$ of the ordinary Prym variety $Q_1$ in $JF_2$ by $D_{21}$ is the eigen-abelian variety of $D_2$ for the eigenvalue $-n+2 = -3$. It also follows from \cite[Section 6.6]{IzLaSt} and \cite[Theorem 3.1]{Kan87} that in this case $Q_2$ is a Prym-Tyurin variety of dimension $6$ and exponent $-(-3)+1=4$ for the correspondence $D_2$. The restriction $\rho$ of the correspondence $D-1$ to $Q_1 \subset JF_1$ gives the sequence of isogenies of principally polarized abelian varieties
\[
\begin{array}{ccccc}
Q_1 & \stackrel{\rho}{\lra} & P & \lra & Q_2 \\
x_1 & \longmapsto & ((\iota -1)x_1, D_{12} x_1) & \longmapsto & D_{12} x_1.
\end{array}
\]
\end{remark}

\begin{proposition} \label{isoPryms}
The map $\rho$ factors through multiplication by $2$ to induce an isomorphism $Q_1 := P(F_1, \iota) \simeq P$ and a surjection $Q_1 \ra Q_2 := P(F_2, D_2)$ whose kernel is a maximal isotropic subgroup $\bH$ (with respect to the Weil pairing) of the group of points of order $2$ in $Q_1$.
\end{proposition}
\begin{proof}
For an abelian variety $A$, we denote by $n_A:A\rightarrow A$ the morphism given by multiplication by $n\in \mathbb Z$. It follows from \cite[Corollary 2.3]{IL} that $D_{21}\circ D_{12} = 8_{Q_1}$. A straightforward generalization of the proof of \cite[Proposition 3.3]{IL} implies that $D_{12} = \varphi \circ 2_{Q_1}$, for an isogeny $\varphi : Q_1 \ra Q_2$ such that $\varphi^* \Theta_{Q_2} = \Theta_{Q_1}$, where $\Theta_{Q_i}$ is the polarization of $Q_i$. Therefore we have $\varphi\circ  \varphi^t = 2_{Q_2}$. It follows that the kernel of $\varphi$ is a maximal isotropic subgroup $\bH$ of the group $Q_1[2]$ of points of order $2$ in $Q_1$.
Since the restriction of $\iota -1$ to $Q_1$ is $-2_{Q_1}$, its kernel is the subgroup of points of order $2$. Therefore $\rho = \psi \circ 2_{Q_1}$, where now $\psi: Q_1 \ra P$ is injective, hence an isomorphism.
\end{proof}

\section{The Weyl-Petri realization of the Hodge eigenbundles}\label{sec:WP}

Since the pull-back of the Hodge class from $\overline{\cA}_6$ is precisely the class $\lambda^{(-5)}$ on $\hhu$, describing it in terms intrinsic to the Hurwitz space is of obvious importance. Here we show that, at least on an open dense subset of $\hhu$, both Hodge eigenbundles $\mathbb E^{(+1)}$ and $\mathbb E^{(-5)}$ admit a Petri-like incarnation, which makes them amenable to intersection-theoretic calculations.

\begin{num} For a smooth $E_6$-cover $[\pi: C\rightarrow \mathbb P^1]\in \Hur$, set $L:=\pi^*(\cO_{\mathbb P^1}(1))\in W^1_{27}(C)$ and let
$$\mu(L):H^0(C,L)\otimes H^0(C, \omega_C\otimes L^{\vee})\rightarrow H^0(C,\omega_C)$$ be the Petri map given by multiplication of global sections.
Assume $h^0(C,L)=2$, so that $h^0(C, \omega_C\otimes L^{\vee})=20$. If, as expected, for a general choice of $[C,L]\in \Hur$, the map $\mu(L)$ is injective, then $\mbox{Im } \mu(L)$ is a codimension $6$ subspace of $H^0(C, \omega_C)$. Remarkably, this is the $(+1)$-eigenspace of $H^0(C,\omega_C)$. At the end of this paper, we shall establish that a general covering from $\Hur$ is Petri general:

\end{num}

\begin{theorem}\label{petri}
For a general point $[C,L]\in \Hur$, the multiplication map $\mu(L)$ is injective.
\end{theorem}

Postponing the proof, we have the following description of the Hodge eigenbundles.

\begin{theorem}\label{incarnations}
Let $[C,L]\in \hhu$ be an element corresponding to a nodal curve of genus $46$ and a base point free line bundle $L\in W^1_{27}(C)$, such that $h^0(C,L)=2$ and the Petri map $\mu(L)$ is injective. One has the following canonical identifications:
$$\mathrm{(i)}\ \ \mbox{ } \  H^0(C,\omega_C)^{(+1)}=H^0(C,L)\otimes H^0(C,\omega_C\otimes L^{\vee}).$$
$$\mathrm{(ii)}\ \ \mbox{ }\  H^0(C,\omega_C)^{(-5)}=\left(\frac{H^0(C,L^{\otimes 2})}{S^2 H^0(C,L)}\right)^{\vee}\otimes \bigwedge^2 H^0(C,L).$$
\end{theorem}
\begin{proof}
Let $L$ be an $E_6$-pencil on $C$. Consider a general divisor $\Gamma\in |L|$ and the exact sequence
\begin{equation}\label{standard}
0 \lra \cO_C \stackrel{\cdot s}\lra L \lra \cO_{\Gamma}(\Gamma) \lra 0
\end{equation}
induced by a section $s \in H^0(C,L)$ with $\mbox{div}(s)=\Gamma$, and its cohomology sequence
\begin{equation}\label{cohstandard}
0 \lra H^0 (C, \cO_C) \lra H^0(C,L) \lra H^0(\cO_{\Gamma}(\Gamma)) \stackrel{\alpha}{\lra} H^1(C, \cO_C) \lra H^1(C,L) \lra 0.
\end{equation}
There is an action of $W(E_6)$ on $H^0(\cO_{\Gamma}(\Gamma))$ compatible with the trivial action on $H^0 (C,L)$, because $L$ and $H^0(L)$ are pull-backs from $\bP^1$. We identify the space $H^0(\cO_{\Gamma}(\Gamma))$ with the vector space generated by the $27$ lines on a smooth cubic surface; each line is represented by a point of $\Gamma$ and the incidence correspondence of lines is the Kanev correspondence $D$. Therefore the representation of $W(E_6)$ on $H^0(\cO_{\Gamma}(\Gamma))$ splits into the sum of three irreducible representations: the trivial $1$-dimensional one, the $6$-dimensional one which coincides with the representation on the primitive cohomology of a cubic surface and the $20$-dimensional one, which coincides with the one on the space of linear equivalences on a cubic surface, see for instance \cite{AV}.

\vskip 3pt

The Kanev correspondence $D$ induces an endomorphism on $H^0(\cO_{\Gamma}(\Gamma))$ compatible with the endomorphism $D^{\vee}\in  \mbox{End}\bigl(H^1(C,\cO_C)\bigr)$ via the cohomology sequence \eqref{cohstandard}. On a cubic surface, the action of the incidence correspondence on the primitive cohomology is equal to multiplication by $-5$ and its action on the space of rational equivalences is the identity. Therefore this is also how the action on $H^0(\cO_{\Gamma}(\Gamma))$ can be described. It follows that the image $\alpha\bigl(H^0(\cO_{\Gamma}(\Gamma))\bigr)$ contains the $(-5)$-eigenspace, that is, we have an inclusion
$$\bigl(H^0(C,\omega_C)^{(-5)}\bigr)^{\vee}\subseteq \alpha\bigl(H^0(\cO_{\Gamma}(\Gamma))\bigr) = \left(\frac{H^0(C,\omega_C)}{H^0(C,L)\otimes H^0(C, \omega_C\otimes L^{\vee})}\right)^{\vee}.$$
When the Petri map $\mu(L)$ is injective, the two spaces appearing in this inclusion have the same dimension and the inclusion becomes an equality, which establishes the first claim.

\vskip 4pt
To prove the second claim, we start by observing that the
Base Point Free Pencil Trick yields the sequence  $0\rightarrow \bigwedge ^2 H^0(C,L)\otimes L^{\vee}\rightarrow H^0(C,L)\otimes \mathcal{O}_C\rightarrow L\rightarrow 0$. After tensoring with $L$ and taking cohomology, we arrive at the following exact sequence
$$0\longrightarrow \frac{H^0(C,L^{\otimes 2})}{\mbox{Sym}^2 H^0(C,L)} \longrightarrow \bigwedge ^2 H^0(C,L)\otimes H^1(C, \mathcal{O}_C)\stackrel{u}\longrightarrow H^0(C,L)\otimes H^1(C,L)\longrightarrow 0.$$
To describe the map $u$ in this sequence, let us choose a basis $s_1, s_2\in H^0(C,L)$. Then,
$$u(s_1\wedge s_2\otimes f)=s_1\otimes (s_2\cdot f)-s_2\otimes (s_1\cdot f)\in H^0(C,L)\otimes H^1(C,L).$$ It follows via Serre duality, that $\mbox{Ker}(u)$ consists of all linear maps $v \colon H^0(C,\omega_C)\rightarrow \mathbb C$ vanishing on $H^0(C,L)\otimes H^0(C,\omega_C\otimes L^{\vee})\subset H^0(C,\omega_C)$, which proves the claim.
\end{proof}

\vskip 5pt

The identifications provided by Theorem \ref{incarnations} extend to isomorphisms of vector bundles over a partial compactification of $\Hur$ which we shall introduce now. This allows us to express to Hodge classes
$\lambda^{(+1)}$ and $\lambda^{(-5)}$ in terms of certain tautological classes and define the Petri map globally at the level of the moduli stack.


\vskip 3pt

\begin{num} Let $\widetilde{\cM}_{46}$ be the open subvariety of $\overline{\cM}_{46}$ parametrizing irreducible curves  and denote by $\mathcal{G}^1_{27}\rightarrow \widetilde{\cM}_{46}$ the stack parametrizing pairs $[C, L]$, where $[C]\in \widetilde{\cM}_{46}$ (in particular $C$ is a \emph{stable} curve) and $L$ is a torsion free sheaf of degree $27$ on $C$ with $h^0(C,L)\geq 2$. Note that $\mathcal{G}^1_{27}$ is a locally closed substack of the universal Picard stack of degree $27$ over $\widetilde{\cM}_{46}$. Let $\mathcal G_{E_6}$ be the locus of pairs $[C,L]\in \mathcal{G}^1_{27}$, where $L$ is \emph{locally free} and \emph{base point free} with $h^0(C,L)=2$  and the monodromy of the pencil $|L|$ is equal to $W(E_6)$. We denote by $\sigma:\mathcal{G}_{E_6}\rightarrow \widetilde{\cM}_{46}$ the projection map given by $\sigma([C,L]):=[C]$.
\end{num}

\begin{num} One has a birational map $\beta:\hhu\dashrightarrow \mathcal{G}_{E_6}\subset \mathcal{G}^1_{27}$ which can be extended over each boundary divisor of $\hhu$
 not contracted under the Prym-Tyurin map (see Theorem \ref{contractions} for a description of these divisors). Let $t:=[\pi:C=C_1\cup C_2\rightarrow R_1\cup_q R_2]$ be a general point of one of the divisors $D_{E_6}$, $D_{\mathrm{azy}}$ or $D_{\mathrm{syz}}$, where we recall that $C_i:=\pi^{-1}(R_i)$ for $i=1,2$ and only two branch points of $\pi$ specialize to $R_2$. We assign to $t$ the point $\bigl[\mathrm{st}(C), \mathrm{st}(f^*\mathcal{O}_{R_1\cup R_2}(1,0))\bigr]\in \mathcal{G}_{E_6}$, where $\mathrm{st}$ is the map assigning to a nodal curve $X$ its stable model $\mathrm{st}(X)$  and to a line bundle $L$ on $X$ the line bundle $\mathrm{st}(L)$ on $\mathrm{st}(C)$ obtained by adding base points to each destabilizing component of $X$ which is contracted. Geometrically, for the general point of each of the divisors $D_{E_6}$, $D_{\mathrm{azy}}$ or $D_{\mathrm{syz}}$, the map $\beta:\hhu\dashrightarrow \mathcal{G}_{E_6}$ contracts the curve $C_2$. We still denote by $D_{\mathrm{azy}}, D_{\mathrm{syz}}$ and $D_{E_6}$ the images under $\beta$ of the boundary divisors denoted by the same symbols on $\hhu$.

\vskip 4pt

We now describe the effect of $\beta$ along each of the boundary divisors in question. If $t\in D_{\mathrm{azy}}$ is a general point, then $C_1$ is smooth and $L:=\pi_{C_1}^*(\mathcal{O}_{R_1}(1))\in W^1_{27}(C_1)$ has $6$ triple ramification points over the branch point $q\in R_1$. Then $\beta(t)=[C_1,L]\in \mathcal{G}_{E_6}$. If, on the other hand, $t$ is a general point of $D_{\mathrm{syz}}$, then retaining the notation of Remark \ref{aut1}, $C_1$ is a smooth curve of genus $45$, meeting the smooth rational component $Z$ in two points $u, v\in \pi^{-1}(q)$. Then $\beta(t)=[C',L]$, where $C':=C_1/u\sim v$ is an irreducible $1$-nodal curve of genus $46$ and $L\in W^1_{27}(C')$ is the pencil inducing the map $\pi$. Finally, if $t\in D_{E_6}$, then $\beta(t)=[C',L]$, where $C'$  is a $6$-nodal curve obtained from $C_1$ by identifying the points of $\pi^{-1}(q)$ which belong to the same component of $C_2$.

\vskip 4pt

We record the formula $\lambda=\frac{33}{46}[D_0]+\frac{7}{46} [D_{\mathrm{azy}}]+\frac{17}{92} [D_{\mathrm{syz}}]+\cdots \in CH^1(\mathcal{G}_{E_6})$ for the Hodge class at the level of $\mathcal{G}_{E_6}$. The factor of $\frac{1}{2}$ in front of $[D_{\mathrm{syz}}]$ compared to (\ref{lamhhu}) is explained by the fact that the general point of $D_{\mathrm{syz}}\subset \hhu$ has an automorphism of order $2$, whereas its image under $\beta$ has only trivial automorphsms.
\end{num}

\begin{num} At the level of $\mathcal{G}_{E_6}$  one can introduce several tautological classes along the lines of \cite{F09}. We denote by $f:\cC_{E_6}\rightarrow \mathcal{G}_{E_6}$ the universal genus $46$ curve and choose a universal line bundle $\mathcal{L}\in \mbox{Pic}(\cC_{E_6})$  satisfying the property $\cL_{| f^{-1}([C,L])}=L\in W^1_{27}(C)$, for each $[C,L]\in \mathcal{G}_{E_6}$. We then define the following tautological classes:
$$\mathfrak{A}:=f_*\bigl(c_1^2(\mathcal{L})\bigr), \mbox{ } \mbox{  }  \mathfrak{B}:=f_*\bigl(c_1(\mathcal{L})\cdot c_1(\omega_f)\bigr), \ \kappa:=f_*\bigl(c_1^2(\omega_f)\bigr)\in CH^1(\mathcal{G}_{E_6}).$$
Via Grauert's theorem, we observe that $\cV:=f_*\mathcal{L}$ is a locally free sheaf  of rank two on $\mathcal{G}_{E_6}$. Similarly, the sheaf
$$\cV_2:=f_*(\mathcal{L}^{\otimes 2})$$ is locally free of rank $9$ over $\mathcal{G}_{E_6}$. Globalizing at the level of moduli the multiplication map of global
sections $\mbox{Sym}^2 H^0(C,L)\rightarrow H^0(C, L^{\otimes 2})$, we define the rank $6$ vector bundle $\mathcal{E}_2$ over $\mathcal{G}_{E_6}$ via the following exact sequence:
$$0\longrightarrow \mbox{Sym}^2(\cV)\longrightarrow \cV_2\longrightarrow \mathcal{E}_2\longrightarrow 0.$$
\end{num}

\begin{num} The choice of  $\mathcal{L}$ is not unique; replacing $\mathcal{L}$ by $\mathcal{L}':=\mathcal{L}\otimes f^*(\alpha)$, where $\alpha\in \mbox{Pic}(\mathcal{G}_{E_6})$ and denoting the corresponding tautological classes by $\mathfrak{A}', \mathfrak{B}'\in CH^1(\mathcal{G}_{E_6})$ respectively, we find the relations
$$\mathfrak{A}'=\mathfrak{A}+2\cdot 27 \cdot \alpha \ \mbox{ and } \ \mathfrak{B}'=\mathfrak{B}+(2\cdot 46 -2)\cdot \alpha.$$
It follows that $\mathfrak{B}'-\frac{5}{3}\mathfrak{A}'=\mathfrak{B}-\frac{5}{3}\mathfrak{A}$, that is, the class
\begin{equation}\label{gamma}
\gamma:=\mathfrak B-\frac{5}{3}\mathfrak A\in CH^1(\mathcal G_{E_6})
\end{equation}
is well-defined and independent of the choice of a Poincar\'e bundle $\mathcal{L}$. We now describe in a series of calculations the Chern classes of the vector bundles we have just introduced.
\end{num}

\begin{proposition}\label{v2}
The following relations hold in $CH^1(\mathcal{G}_{E_6})$:
$$c_1(\cV_2)=\lambda-\mathfrak{B}+2\mathfrak{A} \  \ \mbox{ and }  \  \ c_1\Bigl(R^1f_*(\omega_f\otimes \mathcal{L}^{\vee})\Bigr)=\lambda+\frac{\mathfrak{A}}{2}-\frac{\mathfrak{B}}{2}-c_1(\cV).$$
\end{proposition}
\begin{proof} We apply Grothendieck-Riemann-Roch to $f:\mathcal{C}_{E_6}\rightarrow \mathcal{G}_{E_6}$ and write
$$c_1(\cV_2)=f_*\Bigl[\Bigl(1+2c_1(\mathcal{L})+2c_1^2(\mathcal{L})\Bigr)\cdot \Bigl(1-\frac{c_1(\omega_f)}{2}+\frac{c_1^2(\Omega^1_f)+c_2(\Omega^1_f)}{12}\Bigr)\Bigr]_2.$$
Now use Mumford's formula $f_*\bigl(c_1^2(\Omega_f^1)+c_2(\Omega_f^1)\bigr)=12\lambda$, see \cite{HM} p. 49,  and conclude. \end{proof}

\begin{num} Theorem \ref{petri} (to be proved in Section 10) shows that the Petri map $\mu(L)$ is injective for a general point of $[C,L]\in \mathcal{G}_{E_6}$. However, we cannot rule out the (unlikely) possibility that $\mu(L)$ is not injective along a divisor $\mathfrak{N}$ on $\mathcal{G}_{E_6}$. We denote by $\mathfrak{n}:=[\mathfrak{N}]\in CH^1(\mathcal{G}_{E_6})$. This (possibly zero) class  is effective. Globalizing Theorem \ref{incarnations}, we obtain isomorphisms of vector bundles over $\mathcal{G}_{E_6}-\mathfrak{N}$:
$$\mathbb E^{(+1)}=R^1f_*\bigl(\omega_f\otimes \mathcal{L}^{\vee}\bigr)\otimes \cV \  \ \mbox{ and } \ \ \mathbb E^{(-5)}= \mathcal{E}_2^{\vee} \otimes \mbox{det}(\cV).$$
Extending this to $\mathcal{G}_{E_6}$, there  exists an injection of vector bundles $R^1f_*\bigl(\omega_f\otimes \mathcal{L}^{\vee}\bigr) \otimes \cV \hookrightarrow \mathbb E^{(+1)}$, with quotient a sheaf supported on $\mathfrak{N}$ and on possibly other higher codimension cycles.
\end{num}
\begin{proposition}\label{che}
The following formulas hold at the level of $\mathcal{G}_{E_6}$:
$$\lambda^{(+1)}=2\lambda-\gamma+\mathfrak{n} \ \mbox{ and } \lambda^{(-5)}=-\lambda+\gamma-\mathfrak{n}.$$
\end{proposition}
\begin{proof} We have that $\lambda^{(+1)}=c_1\Bigl(R^1f_*\bigl(\omega_f\otimes \mathcal{L}^{\vee}\bigr)\otimes \cV \Bigr)-[\mathfrak{N}]\in CH^1(\mathcal{G}_{E_6})$ and the rest is  a consequence of Theorem \ref{incarnations} coupled with Proposition \ref{v2}.
\end{proof}

\begin{proposition}\label{basepoint}
We have that $\mathfrak{A}=27c_1(\cV)\in CH^1(\mathcal{G}_{E_6})$.
\end{proposition}
\begin{proof} Recall that $\mathcal{G}_{E_6}$ has been defined as a locus of pairs $[C,L]$ such that $L$ is a base point free pencil. In particular, the image under $f$ of the codimension $2$  locus in $\mathcal{C}_{E_6}$ where the morphism of vector bundles  $f^*(\cV)\rightarrow \mathcal{L}$ is not surjective is empty, hence by Porteous' formula
$$0=f_*\Bigl(c_2(f^*\cV)-c_1(f^*\cV)\cdot c_1(\mathcal{L})+c_1^2(\mathcal{L})\Bigr)=-27c_1(\cV)+\mathfrak{A}.$$
\end{proof}

Essential in all the ensuing calculations is the following result expressing the divisor $D_{\mathrm{azy}}$ in terms of Hodge eigenbundles and showing that its class is quite positive:

\begin{theorem}\label{azy3} The following relation holds:
$$[D_{\mathrm{azy}}]= 5\lambda+\lambda^{(-5)}-3[D_{E_6}]-\frac{5}{6}[D_{\mathrm{syz}}]+\mathfrak{n}\in CH^1(\mathcal{G}_{E_6}).$$
\end{theorem}
\begin{proof}
The idea is to represent $D_{\mathrm{azy}}$ as the push-forward of the codimension two locus in the universal curve $\mathcal{C}_{E_6}$ of the locus of pairs $[C,L,p]$ such that $h^0(C, L(-3p))\geq 1$. We form the fibre product of the universal curve $\mathcal{C}_{E_6}$ together with its projections:
$$
\begin{CD}
{\mathcal{C}_{E_6}} @<\pi_1<< {\mathcal C_{E_6}\times_{\mathcal{G}_{E_6}}\mathcal{C}_{E_6}} @>\pi_2>> {\mathcal{C}_{E_6}} \\
\end{CD}.
$$
For each $k\geq 1$, we consider the \emph{locally free} jet bundle $J_k(\mathcal{L})$ defined, e.g., in \cite{E96}, as a locally free replacement of the sheaf of principal parts $\mathcal P_{f}^k(\cL):=(\pi_2)_{*}\Bigl(\pi_1^*(\mathcal{L})\otimes \mathcal{I}_{(k+1)\Delta}\Bigr)$ on $\mathcal{C}_{E_6}$.
Note that $\mathcal P_f^k(\mathcal{L})$ is not locally free along the codimension two locus in $\mathcal{C}_{E_6}$ where $f$ is not smooth. To remedy this problem,  we consider the \emph{wronskian} locally free replacements $J_f^k(\cL)$, which are related by the following commutative diagram for each $k\geq 1$:
$$
     \xymatrix{
         0 \ar[r] & \Omega_f^{k}\otimes \cL \ar[r]^{} \ar[d]_{} & \mathcal{P}_f^k(\cL) \ar[d]_{} \ar[r]^{} & \mathcal{P}_f^{k-1}(\cL) \ar[d] \ar[r] & 0 \\
          0 \ar[r] & \omega_f^{\otimes k}\otimes \mathcal{L} \ar[r] &J_f^k(\mathcal{L}) \ar[r]      & J_f^{k-1}(\mathcal{L}) \ar[r] & 0 }
$$
Here $\Omega_f^k$ denotes the $\mathcal O_{\mathcal G_{E_6}}$-module $\mathcal{I}_{k\Delta}/\mathcal{I}_{(k+1)\Delta}$. The first vertical row here is induced by the canonical map $\Omega_f^k\rightarrow \omega_f^{\otimes k}$, relating the sheaf of relative K\"ahler differentials to the relative dualizing sheaf of the family $f$. The sheaves $\mathcal P^k_f(\cL)$ and $J^k_f(\cL)$ differ only along the codimension two singular locus of $f$.
Furthermore, for each integer $k\geq 0$ there is a vector bundle morphism $\nu_k:f^*(\cV)\rightarrow J_f^k(\cL)$, which for points $[C,L,p]\in \mathcal{G}_{E_6}$ such that $p\in C_{\mathrm{reg}}$, is just the evaluation morphism $H^0(C,L)\rightarrow H^0(L_{| (k+1)p})$. We specialize now to the case $k=2$ and consider the codimension two locus $Z\subset \mathcal{C}_{E_6}$ where $\nu_2:f^*(\cV)\rightarrow J_f^2(\mathcal{L})$ is not injective. Then, at least over the locus of smooth curves, $D_{\mathrm{azy}}$ is the set-theoretic image of $Z$. A simple local analysis shows that the morphism $\nu_2$ is simply degenerate for each point $[C,L,p]$, where $p\in C_{\mathrm{sing}}$. Taking into account that a general point of $D_{\mathrm{azy}}$ corresponds to a pencil with \emph{six} triple points aligned over one branch point, and that the stable model of a general element of the divisor $D_{\mathrm{syz}}$ corresponds to a curve with \emph{one} node, whereas that of a general point of $D_{E_6}$ to a curve with \emph{six} nodes and so on, we obtain the formula:
\[
6[D_{\mathrm{azy}}]= f_* c_2\left(\frac{J_f^2(\mathcal{L})}{f^*(\cV)}\right)-6[D_{E_6}]-3[D_{\mathrm{syz}}]\in CH^1(\mathcal{G}_{E_6}).
\]
The fact that $D_{\mathrm{syz}}$ appears with multiplicity $3$ is a result of the following local computation. We choose a family $F:X\rightarrow B$ of curves of genus $46$ over a smooth $1$-dimensional base $B$, such that $X$ is smooth, and there is a point $b_0\in B$ such that $X_b:=F^{-1}(b)$ is smooth for $b\in B\setminus \{b_0\}$, whereas $X_{b_0}$ has a unique node $N\in X$. Assume also that $L\in \mbox{Pic}(X)$ is a line bundle such that $L_b:=L_{|X_b}$ is a pencil with $E_6$-monodromy on $X_b$ for each $b\in B$, and furthermore $[X_{b_0}, L_{b_0}]\in D_{\mathrm{syz}}$. Choose a local parameter $t\in \mathcal{O}_{B, b_0}$ and $x, y\in \mathcal{O}_{X,N}$, such that $xy=t$ represents the local equation of $X$ around the point $N$. Then $\omega_F$ is locally generated by the meromorphic differential $\tau=\frac{dx}{x}=-\frac{dy}{y}$. We choose two sections $s_1, s_2\in H^0(X,L)$, where $s_1$ does not vanish at $N$ and $s_2$ vanishes with order $2$ at $N$ along both branches of $X_{b_0}$. Then we have the relation $s_{2,N}=(x^2+y^2)s_{1,N}$ between the germs of the two sections $s_1$ and $s_2$ at $N$. We compute
$$d(s_2)=2xdx+2ydy=2(x^2-y^2)\tau, \ \ \mbox{ and } \ \ d(x^2-y^2)=2(x^2+y^2)\tau.$$
In local coordinates, the map $H^0\bigl(X_{b_0}, L_{b_0}\bigr) \rightarrow H^0\bigl(X_{b_0}, L_{b_0|3N}\bigr)$ is then given by the $2\times 2$ minors of the following matrix:
$$\begin{pmatrix}
1 & 0& 0\\
x^2+y^2 & x^2-y^2 & x^2+y^2 \\
\end{pmatrix}.
$$
This completes the proof that $[D_{\mathrm{syz}}]$ appears with multiplicity $3$ in the degeneracy locus.

\vskip 4pt

We compute: $c_1(J_f^2(\cL))=3c_1(\cL)+3c_1(\omega_f)$ and $c_2(J_f^2(\cL))=3c_1^2(\cL)+6c_1(\cL)\cdot c_1(\omega_f)+2c_1^2(\omega_f)$,
hence
$$f_* c_2\left(\frac{J_f^2(\mathcal{L})}{f^*(\cV)}\right)=3\mathfrak{A}+6\mathfrak{B}-3(d+2g-2)c_1(\cV)+2\kappa_1=6\gamma+2\kappa_1.$$
Furthermore, $\kappa_1=12\lambda-6[D_{E_6}]-[D_{\mathrm{syz}}]-\cdots$, hence after applying Proposition \ref{che}, we obtain the claimed formula.
\end{proof}

We can also express the divisors $D_{\mathrm{syz}}$ and $D_{\mathrm{azy}}$ in terms of the Hodge eigenclasses.

\begin{proposition}\label{prop:syzazy4}
The following formulas hold in $CH^1(\mathcal{G}_{E_6})$:
$$[D_{\mathrm{azy}}]=\frac{25}{16}\lambda+\frac{51}{16}\lambda^{(-5)}+\frac{3}{4}[D_{E_6}]+\frac{51}{16}\mathfrak{n} \ \mbox { and } \ [D_{\mathrm{syz}}]=\frac{33}{8}\lambda-\frac{21}{8}\lambda^{(-5)}-\frac{9}{2}[D_{E_6}]-\frac{21}{8}\mathfrak{n}.$$
\end{proposition}

\begin{proof} Combine Theorem \ref{azy3} with the expression of the Hodge class $\lambda$ in terms of the boundary divisor classes on $\mathcal{G}_{E_6}$.
\end{proof}

\begin{corollary}\label{syzbound}
We have $[D_{\mathrm{syz}}]\leq \frac{33}{8}\lambda-\frac{21}{8}\lambda^{(-5)}-\frac{9}{2}[D_{E_6}]$.
\end{corollary}

\vskip 3pt

We are now in a position to determine the class of the ramification divisor of the Prym-Tyurin map in terms of the classes already introduced. Recall that $D_6:=\overline{\cA}_6 \setminus \cA_6$ is the irreducible boundary divisor of the perfect cone compactification of $\cA_6$ and $\lambda_1\in CH^1(\overline{\cA}_6)$ denotes the Hodge class. Note that $K_{\overline{\cA}_6}=7\lambda_1-[D_6]$, see \cite{mumford1983on-the-kodaira-dimension}.

\begin{theorem}\label{ram}
The ramification divisor of the map $PT:\mathcal{G}_{E_6} \dashrightarrow \overline{\cA}_6$ is given by
$$\bigl[\mathrm{Ram}(PT)\bigr]=\frac{73}{32}\lambda-\frac{221}{32}\lambda^{(-5)}-\frac{9}{8} [D_{E_6}]+\frac{3}{32}\mathfrak{n}.$$
\end{theorem}
\begin{proof} The general point of $D_{E_6}$ corresponds to a semi-abelian variety of torus rank $1$, whereas for all the other boundary divisors in $\mathfrak{br}^*(\widetilde{B}_2)$ the corresponding torus rank is zero. Moreover $PT^*(D_6)=D_{E_6}$ (recall that in this proof, the map $PT$ is defined on the partial compactification $\mathcal{G}_{E_6}$, the formula above does not hold on $\hhu$). Via the Hurwitz formula, we obtain that
$$\bigl[\mathrm{Ram}(PT)\bigr]=K_{\mathcal{G}_{E_6}}-PT^*\bigl(7\lambda_1-[D_6]\bigr)=K_{\mathcal{G}_{E_6}}-7\lambda^{(-5)}+[D_{E_6}].$$
Recall that the canonical class $K_{\hhu}$ has been expressed in terms of boundary divisors on $\hhu$. Using Theorem \ref{azy3}, we can pass to a new basis in $CH^1(\mathcal{G}_{E_6})$ involving the Hodge eigenbundles and one boundary divisor, namely $D_{E_6}$. After simple manipulations we obtain
\begin{equation}\label{cang}
K_{\mathcal{G}_{E_6}}=\frac{73}{32}\lambda+\frac{3}{32}\lambda^{(-5)}- \frac{17}{8}[D_{E_6}]+\frac{3}{32}\mathfrak{n},
\end{equation}
which then leads to the claimed formula.
\end{proof}

We now complete the proof of Theorem \ref{thmmain2}. In what follows, we revert to the notation of the introduction and $PT:\hhu\ratmap \overline{\cA}_6$ denotes the
extended Prym-Tyurin map.

\begin{theorem}
The canonical class of the partial compactification $\mathcal{G}_{E_6}$ of $\Hur$ is big. It follows that there exists a divisor $E$ on $\hhu$ with $PT_*(E)=0$, such that
$K_{\hhu}+E$ is big.
\end{theorem}
\begin{proof}
The varieties $\mathcal{G}_{E_6}$ and $\hhu$ differ in codimension one only along boundary divisors that are collapsed under the Prym-Tyurin map. Showing that $K_{\mathcal{G}_{E_6}}$ is big implies therefore the second half of the claim, and thus Theorem \ref{thmmain2}. Using Theorem \ref{kanonisch} (note the caveat about the already mentioned factor $\frac{1}{2}$ in front of the coefficient of $[D_{\mathrm{syz}}]$ when passing from $\hhu$ to $\mathcal{G}_{E_6}$), coupled with Proposition  \ref{prop:syzazy4}, we write:
$$K_{\mathcal{G}_{E_6}}=-\frac{25}{46}[D_{E_6}]+\frac{19}{46}[D_{\mathrm{syz}}]+\frac{17}{46}[D_{\mathrm{azy}}]\geq -\frac{25}{46}[D_{E_6}]+\frac{19}{46}[D_{\mathrm{syz}}]+\frac{17}{46}\Bigl(\frac{25}{16}\lambda+\frac{51}{16}\lambda^{(-5)}+\frac{3}{4}[D_{E_6}]\Bigr)$$
$$=\frac{867}{736}\lambda^{(-5)}+\frac{425}{736}\Bigl(\lambda-\frac{196}{425}[D_{E_6}]\Bigr).$$
Putting Proposition \ref{dzero} together with the fact that $\lambda^{(-5)}$ is big, the conclusion follows by comparing the ratio of the $\lambda$ and $[D_{E_6}]$-coefficients of the last expression. Indeed, it is shown in  (\ref{moriwaki3}) by pulling-back the Moriwaki class from $\mm_{46}$ that the $\mathbb Q$-class $\lambda-\frac{6}{8+\frac{4}{g}}[D_{E_6}]=\lambda-\frac{23}{31}[D_{E_6}]$ is effective on $\mathcal{G}_{E_6}$. It follows that $\lambda-\frac{196}{425}[D_{E_6}]$ is then also an effective class, hence $K_{\mathcal{G}_{E_6}}$ is big.
\end{proof}

\section{The ramification divisor of the Prym-Tyurin map}\label{sec:ramif}

The aim of this section is to describe the differential of the Prym-Tyurin map $PT$ and prove Theorem \ref{rampt}. In this section, the tangent spaces we consider are those of the corresponding moduli stacks. As in the previous section, we fix a smooth $E_6$-cover
$\pi:C\rightarrow \mathbb P^1$ with branch divisor $B:=p_1+\cdots+p_{24}$ and denote $L:=\pi^*(\cO_{\bP^1}(1))$.

Via the \'etale map $\mathfrak{br}:\Hur \rightarrow \cM_{0,24}/S_{24}$, we identify the  cotangent space $T^{\vee}_{[C,\pi]}(\Hur)$ with $H^0(\bP^1, \omega_{\bP^1}^{\otimes 2}(B))$.  The cotangent space $T^{\vee}_{[P(C,f)]}(\cA_6)$ is identified with
$\Sym^2 H^0(C,\omega_C^{\otimes 2})^{(-5)}$.

\begin{definition}
Let $R$ and $A$ be the ramification and antiramification
divisors of $\pi$, that is, the effective divisors of $C$ defined by the formulas
\begin{displaymath}
    \pi^*(B) = 2R + A, \quad
    K_C = \pi^*(K_{\bP^1})  + R, \quad
    2K_C + A = \pi^*(2K_{\bP^1} + B).
  \end{displaymath}
\end{definition}

\begin{definition}
Let $\mbox{tr}: \pi_* \cO_C(-A) \to \cO_{\bP^1}$ be the trace map on regular
functions. For an open affine subset $U\subset \bP^1$, a regular
function $\varphi\in \Gamma(U, \cO_C(-A))$, and  a point $y\in U$, one has
$$\mbox{tr}(\varphi)(y) = \sum_{x\in f\inv(y)} \varphi(x),$$ counted with
multiplicities. Note that $\mbox{tr}$ is surjective.
Let $\pi_* \cO_C(2K_C) \to \cO_{\bP^1}\bigl(2K_{\bP^1}+ B\bigr)$
be the induced trace map at the level of quadratic differentials. We denote
the corresponding map on global sections by $\mbox{Tr}: H^0(C, \omega_C^{\otimes 2}) \to
H^0(\bP^1, \omega_{\bP^1}^{\otimes 2}(B))$.
\end{definition}

\begin{theorem}
The codifferential $(dPT)^{\vee}_{[C,\pi] }: T_{[PT(C, \pi)]}^{\vee}\bigl(\cA_6\bigr)\rightarrow T^{\vee}_{[C,\pi]}\bigl(\Hur \bigr)$ is given by the following
composition of maps:
\begin{displaymath}
\Sym^2 H^0(\omega_C)^{(-5)}
\into \Sym^2 H^0(\omega_C)
\xrightarrow{\mathrm{mul}} H^0(\omega_C^{\otimes 2})
\xrightarrow{\mathrm{Tr}} H^0(\omega_{\bP^1}^{\otimes 2}(p_1+\cdots+p_{24})).
\end{displaymath}
\end{theorem}
\begin{proof}
The second map is the codifferential of the Torelli
map $\cM_{46}\to \cA_{46}$.
The first map is the codifferential of the map from the moduli
space of ppav of dimension $46$ together with an endomorphism $D$
having eigenvalues $(+1)$ and $(-5)$ (with eigenspaces of dimensions $20$ and $6$ respectively) to $\cA_6$.
The third map is the codifferential of the map $\Hur\to \cM_{46}$.
\end{proof}


\begin{num} We now analyze the differential $d PT$ at a point $[C, \pi]\in \Hur$ in detail.  For each of the $24$
branch points $p_i\in \bP^1$, let $\{r_{ij}\}_{j=1}^6\subset C$ be the
ramification points lying over $p_i$. The formal neighborhoods of the
points $r_{ij}$ are naturally identified, so that we can choose a
single local parameter $x$ and write any quadratic differential
$\gamma\in H^0(C, \omega_C^{\otimes 2})$ as
\begin{displaymath}
  \gamma = \varphi_{ij}(x) \cdot (dx)^{\otimes 2} \ \
  \text{near } r_{ij}\in C.
\end{displaymath}
Choose a local parameter $y$ at the point $p_i$, so that $\pi$ is given locally by the map $y=x^2$. We
can use the same local parameter at the remaining $15$ antiramification points $\{q_{ik}\}_{k=1}^{15}$
over $p_i$ at which $\pi$ is unramified, and write $\gamma =
\psi_{ik}(y)\cdot (dy)^{\otimes 2}$ near $q_{ik}\in C$, for $k=1, \ldots, 15$.
\end{num}

\begin{lemma}\label{lem:ker-of-trace}
  The kernel of the trace map $\mathrm{Tr}: H^0(C, \omega_C^{\otimes 2}) \to H^0\bigl(\bP^1, \omega_{\bP^1}^{\otimes 2}(B)\bigr)$ consists of quadratic differentials $\gamma$ such that
$$\sum_{j=1}^6 \varphi_{ij}(r_{ij}) =0, \ \ \mbox{ for } \ i=1, \ldots, 24.$$
\end{lemma}
\begin{proof}
From $y=x^2$, we get $dy= 2x\, dx$ and $(dx)^{\otimes 2} = (dy)^{\otimes 2} / 4y$.
We have that
\begin{displaymath}
\mbox{Tr}(\gamma) = \left( \frac1{4y}
    \sum_{j=1}^6 \big( \varphi_{ij} (x) + \varphi_{ij} (-x)\big)
    + \sum_{k=1}^{15} \psi_{ik}(y)
    \right) \cdot (dy)^{\otimes 2}
    \quad \text{near } p_i.
\end{displaymath}
Suppose $\mbox{Tr}(\gamma)=0$. Then the leading coefficient
$\frac12\sum_{j=1}^6 \varphi_{ij}(r_{ij})$ is zero. Conversely,
assume that the $24$ expressions are zero. Then $\mbox{Tr}(\gamma) \in
H^0(\bP^1, \omega_{\bP^1}^{\otimes 2})=0$.
\end{proof}

\vskip 3pt

In order to understand the condition in Lemma~\ref{lem:ker-of-trace},
we recall the action of the endomorphism  $D\colon H^0(C, \omega_C)\to
H^0(C, \omega_C)$ induced by the Kanev correspondence in local coordinates at the points $p\in C$ and
$q\in \pi\inv(p)$, see also  Theorem \ref{incarnations}.

\begin{num}{\bf{The unramified case.}} Suppose that $\pi$ is unramified at
$p$, thus $\Gamma:=\pi\inv(p)=\sum_{s=1}^{27} q_s$. Since $\pi$ is
\'etale, we can use the same local parameter $y$ at $p$, as well as at each
$q_s\in C$. Let $\alpha\in H^0(C, \omega_C)$. In a formal
neighborhood of each point $q_s$, we write locally $\alpha=
\alpha_s(y)dy$.

Assume $p=[0:1]\in\bP^1$. One has $\sum_{s=1}^{27}\mathrm{Res}_{q_s}(\alpha\cdot
\frac{x_0}{x_1})=0$, so $\sum_{s=1}^{27} \alpha_s(q_s)=0$. The
action of the correspondence on $(\alpha_s)$ is described by an
endomorphism of
$\hat\cO^{26} = \left\{ \sum_{s=1}^{27} \alpha_s=0 \right\} \subset \hat\cO^{27}$, where
$\hat\cO=\hat{\cO}_{\mathbb P^1, p}$.
This endomorphism is given by the same integral $(26\times 26)$-matrix
as the action of $D$ on $H^0(\cO_{\Gamma}(\Gamma))/H^0(C,L)$, as in the proof of Theorem \ref{incarnations}. Thus, $D$ has two eigenvalues $(+1)$ and $(-5)$ with the
eigenspaces of dimensions $20$ and $6$ respectively. Choose a basis $\{v_m\}_{m=1}^6$
of the $(-5)$-eigenspace in $\bZ^{26}$. Then an element
$\alpha^{(-5)}\in H^0(C, \omega_C)^{(-5)}$ can be locally written uniquely as
\begin{displaymath}
  \alpha^{(-5)} = \sum_{m=1}^6 \delta_m  v_m \in \hat\cO^{27},
  \quad \text{for some } \delta_m\in \hat\cO.
\end{displaymath}
\end{num}

\begin{num}{\bf{The ramified case.}} Suppose $\pi$ is branched at $p$ and  $\pi \inv(p)$ consists of ramification
points $r_1, \ldots, r_6$ and $15$ antiramification points $q_k$. The points $r_i$
correspond to the ordered pairs $(a_i, b_i)$ of sheets coming
together. On the sheets, the correspondence is defined by
\begin{displaymath}
a_i \mapsto \sum_{j\neq i} (b_j+ c_{ij})  \ \mbox{ and }
\quad
b_i \mapsto \sum_{j\neq i}(a_j+c_{ij}),  \qquad \mbox{ for } i=1, \ldots, 6.
\end{displaymath}
As above, we use a local coordinate $y$ for $p\in \bP^1$ and the $15$ points $q_k\in C$, and a local coordinate $x$ for the ramification points $r_i$, with
$y=x^2$. Thus, we write locally
\begin{displaymath}
  \alpha = \alpha_{r_i}(x) dx \quad\text{near } r_i, \  \mbox{ and }\
  \alpha = \alpha_{q_k}(y) dy \quad\text{near } q_k.
\end{displaymath}
The local involution $x\mapsto -x$ splits the
differential form into the odd and even parts:
\begin{eqnarray*}
  \alpha_{r_i}(x)dx =& \alpha_{r_i}^{\mathrm{odd}}(x^2)dx  +
  \alpha_{r_i}^{\mathrm{ev}}(x^2)xdx, \\
  \alpha_{r_i}(-x)d(-x) =& -\alpha_{r_i}^{\mathrm{odd}}(x^2)dx  + \alpha_{r_i}^{\mathrm{ev}}(x^2)xdx .
\end{eqnarray*}
The even part can be written in terms of $y$ as
$\frac12 \alpha_{r_i}^{\mathrm{ev}}(y) dy$. The odd parts have no
such interpretation and we claim that they do not
mix with the $15$ sheets on which $\pi$ is \'etale:
\end{num}

\begin{lemma}\label{lem:action-ram-case}
The correspondence $D$ induces an endomorphism on the $6$-dimensional
$\hat\cO_y$-module of odd parts $\langle\alpha_{r_i}^{\mathrm{odd}}dx\rangle$.
It is given by a matrix which has $0$ on the main diagonal and $(-1)$
elsewhere. The $(-5)$ eigenspace is $1$-dimensional with generator
$(1,\dotsc,1)$, and, for every element $\alpha\in H^0(C, \omega_C)^{(-5)}$, one has
\begin{displaymath}
(\alpha_{r_1}^{\mathrm{odd}}, \dotsc, \alpha_{r_6}^{\mathrm{odd}}) =
(\phi, \dotsc, \phi),
\quad \text{for some } \phi=\phi(y)
\text{ independent of } i=1, \ldots, 6.
\end{displaymath}
\end{lemma}
\begin{proof}
This case is obtained by taking a limit of the unramified case.  We work in the complex-analytic topology.  A ramification
point $r\in C$ is a limit of two points $q_{a_i},q_{b_i}\in C$ on the sheets
$a_i$ and $b_i$ respectively. The local parameters $x,x'$ at $q_{a_i},q_{b_i}$
are identified as $x'=-x$. One has $y=x^2 = (x')^2$ and we look at
  the limit as $x$ tends to  $0$. From $dx'=-dx$ it follows
  \begin{eqnarray*}
    \alpha_{q_{a_i}}(x)dx =&\alpha_{r_i}^{\mathrm{odd}}(x^2)dx + \alpha_{r_i}^{\mathrm{ev}}(x^2)x dx, \\
    \alpha_{q_{b_i}}(x)dx =&  -\alpha_{r_i}^{\mathrm{odd}}(x^2)dx + \alpha_{r_i}^{\mathrm{ev}}(x^2)x dx.
  \end{eqnarray*}
Under the correspondence between the $27$ sheets, the $\alpha^{\mathrm{odd}}$
contributions from $a_i$ and $b_i$ to the $5$ sheets $c_{ij}$  ($j\neq i$) cancel out.
Conversely, the terms $\alpha_{c_{ij}}$ contribute to $\alpha^{\mathrm{ev}}$ but not to $\alpha^{\mathrm{odd}}$ at the point
$r_i$.

It follows that the homomorphism $D$ sends the $\hat\cO^6$ block of
the odd parts $\alpha^{\mathrm{odd}}_{r_i}$ to itself. The matrix of this linear
map is the same as the matrix of an endomorphism of $\bZ^6$
with the basis of vectors $a_i-b_i$, that is,
\begin{math}
a_i-b_i \mapsto -\sum_{j\ne i} (a_j-b_j).
\end{math}
It is easy to see that this linear map has eigenvalues $(+1)$ and
$(-5)$ and that the $(-5)$-eigenspace is one-dimensional and is
generated by the vector $(1,\dotsc,1)$. The statement now follows.
\end{proof}

\begin{corollary}\label{cor:phis-equal}
Let $\beta\in\Sym^2 H^0(C, \omega_C)^{(-5)}$ and let $\gamma=\mathrm{mul}(\beta)$ be its
image in $H^0(C, \omega_C^{\otimes 2})$. Then in the notation of
Lemma~\ref{lem:ker-of-trace}, one has
$\varphi_{ij}(r_{ij})= \varphi_{ij'}(r_{ij'})$, for all $i=1, \ldots, 24$ and all $1\le j,j'\le 6$.
\end{corollary}
\begin{proof}
Let $\alpha,\alpha'\in H^0(C, \omega_C)^{(-5)}$. Then in the notation of
Lemma~\ref{lem:action-ram-case}, one has \linebreak
\begin{math}
\mathrm{mul}(\alpha\otimes\alpha')(r_{ij}) =
\alpha^{\mathrm{odd}}_{r_{ij}}(r_{ij}) \cdot (\alpha')^{\mathrm{odd}}_{r_{ij}}(r_{ij})
    = \phi(0)\phi'(0),
\end{math}
which is independent of $j=1, \ldots, 6$.
\end{proof}

Lemma \ref{lem:action-ram-case}  has consequences for the geometry of the Abel-Prym-Tyurin canonical curve
$$\varphi_{(-5)}=\varphi_{H^0(C,\omega_C)^{(-5)}}:C\rightarrow \mathbb P^5.$$ In stark contrast with the case of ordinary Prym-canonical curves,
the map $\varphi_{(-5)}$ is far from being an embedding.

\begin{proposition}\label{prymcan}
For an $E_6$-cover $\pi:C\rightarrow \bP^1$, we have $\varphi_{(-5)}(r_{i1})=\cdots=\varphi_{(-5)}(r_{i6})$, for each $i=1, \ldots, 24$.
\end{proposition}

\begin{proof} This is a consequence of Lemma \ref{lem:action-ram-case}: the condition that $\alpha\in H^0(C,\omega_C)^{(-5)}$ vanishes along the divisor $r_{i1}+\cdots+r_{i6}$ is expressed by a
\emph{single} condition $\phi(0)=0$, therefore  \\
$\mbox{dim } \bigl|H^0(C,\omega_C)^{(-5)}(-r_{i1}-\cdots-r_{i6})\bigr|=4$.
\end{proof}

\vskip 4pt

Finally we are in a position to describe the ramification divisor of the map $PT$. Like in the classical Prym case, it turns out that the infinitesimal study of the Prym-Tyurin map can be reduced to the projective geometry of he Abel-Prym-Tyurin curve:

\vskip 4pt

\noindent \emph{Proof of Theorem \ref{rampt}.}
Using Lemma \ref{lem:ker-of-trace} and Corollary \ref{cor:phis-equal}, it follows that the map $PT$ is ramified at a point $[C,\pi]\in\Hur$, if and only if there
exists $0\ne\beta\in\Sym^2 H^0(C, \omega_C)^{(-5)}$ such that
$$\mathrm{mul}(\beta)\in
H^0(C, 2K_C-R)=H^0(C, K_C-2L)=\mbox{Ker}(\mu(L)),$$
where the last equality follows from the Base Point Free Pencil Trick applied to the Petri map $\mu(L)$. If now  $\mu(L)$ is injective, it follows that $\mathrm{mul}(\beta)=0$, which finishes the proof. \hfill $\Box$

\section{A Petri theorem on $\hhu$}\label{sec:Petri}

We now prove the  Petri-like Theorem \ref{petri},  using a degeneration similar to the one used to establish the dominance of the map $PT$. We start with a cover $\pi_t:C_t\to \bP^1$ ramified in $24$ points such that the local monodromy elements are reflections $w_i$ in $12$ pairs of
roots $r_1, \ldots, r_{12}$ generating the lattice $E_6$. We
consider a degeneration in which the $12$ pairs of roots with the same
label come together. The degenerate cover $\pi:C\to
\bP^1$ is branched at $12$ points $q_1, \ldots, q_{12}\in \mathbb P^1$. Over each point $q_i$ there are
$6$ simple ramification points. The curve $C$ is nodal
with $12\times 6=72$ ordinary double points. We record the following fact:

\begin{lemma}
The curve $C$ has $27$ irreducible components isomorphic to $\bP^1$ and
the restriction of $\pi$ to each of them is an isomorphism.
\end{lemma}

Note that the cover $\pi$ is not \emph{admissible} in the sense of Section 5. The corresponding admissible cover is
obtained by replacing each point $q_i\in\bP^1$ by an inserted
$\bP^1$ with two additional marked points $p_i$, $p_{i+12}$, and
modifying the curve $C$ accordingly.

\vskip 3pt

\begin{num}
The $27$ irreducible components $\{X_s\simeq \bP^1\}_{s=1}^{27}$ of $C$ are in bijection with the lines
$\{\ell_s\}_{s=1}^{27}$ on a cubic surface. Let $\Gamma$ be the dual graph of $C$.  For each root $r_i$ with $i=1, \ldots, 12$, there are $6$ pairs
of lines $(a_{ij},b_{ij})$ such that $r_i\cdot a_{ij}=1$, $b_{ij} = a_{ij}+r_i$, hence
$r_i\cdot b_{ij}=-1$. To each pair we associate an edge
$(a_{ij}, b_{ij})$ of $\Gamma$ directed from the vertex $a_{ij}$ to
the vertex $b_{ij}$.  We also fix  $12$ ramification points $q_i\in \bP^1\setminus \{0, \infty\}=\mathbb C^*$ and denote by $\{p_{si}\}_{i=1}^{n_s}$ the nodes of $C$ lying on $X_s$.
Clearly, $\pi(p_{si})\in \{q_1, \ldots, q_{12}\}$, for all $s$ and $i$.
\end{num}

\begin{lemma}
The space $H^0(C, \omega_C)$ is naturally identified with
\begin{displaymath}
H_1(\Gamma,\mathbb C) =\mathrm{Ker} \Bigl\{\bigoplus_{i=1}^{12}\bigoplus_{j=1}^6
    \mathbb C(a_{ij},b_{ij}) \to \bigoplus_{s=1}^{27} \mathbb C \ell_s \Bigr\}.
\end{displaymath}
To an edge $(a_{ij},b_{ij})$ over a root $r_i$ one associates a
differential form $\omega_{ij}$ equal to $\frac{dz}{z-q_i}$ on $X_{a_{ij}}$,
to $-\frac{dz}{z-q_i}$ on $X_{b_{ij}}$ and $0$ on $X_s$, for $s\ne a_{ij},b_{ij}$.
Then $H^0(C, \omega_C)$ is the subspace of
\begin{displaymath}
\bigoplus_{s=1}^{27} H^0\Bigl(X_s, K_{X_s}\bigl(\sum_{i=1}^{n_s} p_{si}\bigr)\Bigr)
\end{displaymath}
of the forms $\omega = \sum c_{ij}\omega_{ij}$, such that for
$1\le s\le 27$ the
sum of residues of $\omega$ on $X_s$ is zero.
Equivalently, for each $1\le s\le 27$, one considers a space of forms
\begin{displaymath}
\omega_s = \sum_{s\in \{a_{ij}, b_{ij}\}} c_i \frac{dz}{z-q_i},
\quad \text{such that } \sum_i c_i = 0.
\end{displaymath}
Then a form $\omega\in H^0(C, \omega_C)$ is equivalent to a collection
of log forms $\{\omega_s\}_{s=1}^{27}$ satisfying the $72$ conditions
$\mathrm{Res}_{q_i}(\omega_{a_{ij}}) + \mathrm{Res}_{q_i}(\omega_{b_{ij}}) = 0$, for each edge
$(a_{ij}, b_{ij})$ of $\Gamma$.
\end{lemma}
\begin{proof}
This follows by putting together two well known facts:

\noindent (1) Let $C$ be a nodal curve with normalization $\nu\colon \wt C\to
C$ and nodes $p_i\in C$ such that $\nu\inv(p_i)= \{ p_i^+,
p_i^-\}$. Then $H^0(C, \omega_C)$ is identified with the space of sections
$$\tilde\omega\in H^0\Bigl(\widetilde{C}, K_{\widetilde{C}}(\sum (p_i^+ +  p_i^-))\Bigr) \
\mbox{ satisfying } \mbox{Res}_{p_i^+}(\tilde\omega) + \mbox{Res}_{p_i^-}(\tilde\omega)=0.$$

\noindent (2) A section of $H^0\bigl(\bP^1, K_{\mathbb P^1}(\sum_i q_i)\bigr)$ is a
linear combination $\sum_i c_i \frac{dz}{z-q_i}$, with $\sum_i
c_i=0$.
\end{proof}

In practice, assume that $H^0(X_s, \omega_{C|X_s})$ is
identified with the space of fractions
\begin{displaymath}
\frac{P_s(x)}{\prod_{i=1}^{n_s}(x-\pi(p_{si}))} \, dx,
\end{displaymath}
where $P_s(x)$ is a polynomial of degree $n_s-2$. Then $H^0(C, \omega_C)\subseteq \bigoplus_{s=1}^{27} H^0(X_s, \omega_{C|X_s})$ is characterized  by the condition that for every
node of $p_{sj}=p_{s'j'}\in C$ joining components $X_s$ and $X_{s'}$, the \emph{sum} of the
residues $$\mathrm{Res}_s:=\frac{P_s(\pi(p_{sj}))}{\prod_{i\ne j}\bigl(\pi(p_{sj})-\pi(p_{si})\bigr)}$$ and $\mathrm{Res}_{s'}$ respectively is $0$.

\vskip 5pt

\begin{num}
We wish to show the injectivity of the Petri map $$\mu(L):H^0(C,L)\otimes H^0(C, \omega_C\otimes L^{\vee})\rightarrow H^0(C, \omega_C)$$
 for a $72$-nodal curve $C$ corresponding to a cover $\pi:C\rightarrow \mathbb P^1$ as above.

\begin{lemma}
Let $H^0(\bP^1, \cO(1) ) = \langle x_0, x_1\rangle\subset H^0(C,L)$. Then the
subspace
\begin{displaymath}
H^0(C, \omega_C\otimes L^{\vee}) \otimes \langle x_0 \rangle \subset H^0(C, \omega_C\otimes L^{\vee}) \otimes H^0(C, L)
\subset H^0(C, \omega_C)
\end{displaymath}
consists of elements $\{\omega_s\}_{s=1}^{27}$ as above, satisfying for
each $1\le s\le 27$ the additional condition $\sum c_i q_i
=0$. Similarly, the subspace $ H^0(C, \omega_C\otimes L^{\vee}) \otimes \langle x_1 \rangle \subseteq H^0(C, \omega_C)$ consists of elements $\{\omega_s\}_{s=1}^{27}$ as above,
  satisfying for $1\le s\le 27$ an additional condition $\sum \frac{c_i}{q_i}=0$.
\end{lemma}
\begin{proof}
We identify $H^0(C, \omega_C\otimes L^{\vee})$ with $H^0(C, \omega_C(-\pi^*\infty))$, hence
$H^0(C, \omega_C\otimes L^{\vee}) \otimes \langle x_0\rangle $ is the space of forms
$\{\omega_s\}_{s=1}^{27}$ such that for each $s=1, \ldots, 27$, they satisfy the equality
\begin{displaymath}
0=\mbox{Res}_{\infty} \big( \omega_s \cdot \frac{x_1}{x_0} \big) =
    - \sum \mbox{Res}_{q_i} \big( \omega_s \cdot \frac{x_1}{x_0} \big) =
    - \sum c_i q_i.
\end{displaymath}
\end{proof}
\end{num}

\vskip  4pt

\noindent To prove Theorem \ref{petri}, it is sufficient to find one
degeneration $\pi:C\to \bP^1$ such that
\begin{enumerate}
\item $h^0(C,L)=2$.
\item The linear subspaces $H^0(C, \omega_C\otimes L^{\vee}) \otimes \langle x_0\rangle $, $H^0(C, \omega_C\otimes L^{\vee})
  \otimes \langle x_1\rangle $ and $H^0(C, \omega_C)^{(-5)}$ generate the vector space $H^0(C, \omega_C)$.
\end{enumerate}

\vskip 4pt

The initial input consists of $12$ points $q_i\in\bP^1\setminus\{0, \infty\}=\mathbb C^*$,  and
 $12$ roots $r_i$ generating the lattice $E_6$. We obtain a system of linear equations in the $72$ variables
$$x_{ij}=\mbox{Res}_{q_i}(\omega_{a_{ij}})=-\mbox{Res}_{q_i}(\omega_{b_{ij}}), \ \mbox{ for } \   i=1,\ldots, 12 \mbox{ and }\  j=1,\ldots, 6.$$  For
each of the spaces $H^0(C_0, \omega_C\otimes L^{\vee}) \otimes \langle x_0\rangle $, respectively  $H^0(C, \omega_C \otimes L^{\vee})\otimes \langle x_1\rangle $, we get $2\times 27$ equations.  By Lemma \ref{lem:action-ram-case}, $H^0(C,\omega_C)^{(-5)}$ is the subspace of $H^0(C,\omega_C)$
of forms $\{\omega_s\}_{s=1}^{27}$ satisfying
$x_{ij}=x_{ij'}$ for all $1\le j,j' \le 6$ and $i=1,\ldots, 12$. This gives a system of
$27+12\times 5$ equations.

\begin{lemma}
The above conditions are satisfied for the following choices of
roots and ramification points:
\begin{enumerate}
\item $r_1=\alpha_{135}$, $r_2=\alpha_{12}$, $r_3=\alpha_{23}$,
$r_4=\alpha_{34}$, $r_5=\alpha_{45}$, $r_6=\alpha_{56}$,
$r_7=\alpha_{\max}$, $r_8=\alpha_{124}$, $r_9=\alpha_{234}$,
$r_{10}=\alpha_{35}$, $r_{11}=\alpha_{13}$, $r_{12}=\alpha_{36}$.
\item $q_i = i$, for $i=1, \ldots, 12$.
\end{enumerate}
\end{lemma}
\begin{proof}
This is now a straightforward linear algebra computation, which we
performed in Mathematica. It can be found at
\cite{supporting-computer-computations}.
\end{proof}

\renewcommand{\MR}[1]{}
\bibliographystyle{amsalpha}

\begin{thebibliography}{ACGH85}

\bibitem[ACV03]{ACV} Dan Abramovich, Alessio Corti and Angelo Vistoli, \emph{Twisted bundles and admissible covers}, Comm. Algebra  \textbf{31} (2003), 3547--3618.


\bibitem[Ale02]{alexeev2002complete-moduli}
Valery Alexeev, \emph{{C}omplete moduli in the presence of semiabelian group
  action}, Annals of Math.  \textbf{155} (2002),  611--708.
  \MR{2003g:14059}


\bibitem[Ale04]{alexeev2004compactified-jacobians}
Valery Alexeev, \emph{{C}ompactified {J}acobians and {T}orelli map}, Publ. Res.
  Inst. Math. Sci. \textbf{40} (2004), 1241--1265. \MR{MR2105707
  (2006a:14016)}

\bibitem[AB12]{alexeev2011extending-torelli}
Valery Alexeev and Adrian Brunyate, \emph{Extending the {T}orelli map to
toroidal compactifications of {S}iegel space}, Invent. Math. \textbf{188}
(2012), 175--196. \MR{2897696}

\bibitem[ABH02]{alexeev2002degenerations-of-prym}
Valery Alexeev, Christina Birkenhake, and Klaus Hulek, \emph{{D}egenerations of
{P}rym varieties}, J. Reine Angew. Math. \textbf{553} (2002), 73--116. \MR{MR1944808 (2003k:14033)}

\bibitem[ACGH85]{ACGH} Enrico Arbarello, Maurizio Cornalba, Philip Griffiths and Joseph Harris, {\em Geometry of Algebraic Curves, Volume I}, Springer, Berlin, 1985.

\bibitem[AV12]{AV} Maksym Arap and Robert Varley, {\em{On algebraic equivalences among the 27 Abel-Prym curves on a generic abelian $5$-fold}}, arXiv:1206.6517, to appear in International Math. Research Notices.

\bibitem[Bea77]{BeauvilleSchottkyPrym} Arnaud Beauville, {\em Prym varieties and the {S}chottky problem}, Invent. Math. \textbf{41} (1977), 149--196.

\bibitem[BdS49]{BS} Armand Borel and J. de Siebethal,
{\em Les sous-groupes ferm\'es de rang maximum des groupes de Lie clos},
Comment. Math. Helv. 23, (1949), 200--221.

\bibitem[Bor72]{borel} Armand Borel,
  {\em Some metric properties of arithmetic quotients of symmetric
    spaces and an extension theorem}, J. Diff. Geom. \textbf{6}
  (1972), 543--560.


\bibitem[CF05]{CF05} Sebastian Casalaina-Martin and Robert Friedman, \emph{Cubic threefolds and abelian varieties of dimension five}, J. Algebraic Geom. \textbf{14} (2005), 295--326.

\bibitem[CG72]{ClemensGriffiths} Herbert Clemens and Phillip Griffiths, {\em The intermediate {J}acobian of the cubic threefold}, Annals of Math. \textbf{95} (1972), 281--356.

\bibitem[Dol12]{dolgachev2012classical-algebraic}
Igor~V. Dolgachev, \emph{Classical algebraic geometry: A modern view}, Cambridge University
  Press, Cambridge, 2012. \MR{2964027}

\bibitem[DS81]{donagi1981the-structure-of-the-prym}
Ron Donagi and Roy  Smith, \emph{The structure of the {P}rym map}, Acta
  Math. \textbf{146} (1981), 25--102. \MR{594627 (82k:14030b)}

\bibitem[Don84]{Don} Ron Donagi, {\em{The unirationality of $\cA_5$}}, Annals of Math. \textbf{119} (1984), 269--307.

\bibitem[Don92]{donagi1992fibres-prym-map} Ron Donagi,  {\em{ The fibers of the Prym map, }}
Curves, Jacobians, and abelian varieties (Amherst, MA, 1990)	Contemp. Math., vol. 136, Amer. Math. Soc., Providence, RI, 1992,  55--125.


\bibitem[DL11]{DyerLehrer2011} M. J. Dyer and G. I. Lehrer, {\em Reflection subgroups of finite and affine Weyl groups}, Transactions American Math. Soc. \textbf{363} (2011), 5971--6005.

\bibitem[Dyn52]{dy} E. Dynkin,
 {\em {Semisimple subalgebras of semisimple Lie algebras}},
Mat. Sbornik N.S. \textbf{30} (1952), 349--462.

\bibitem[Est96]{E96} Eduardo Esteves, {\em{Wronski algebra systems on families of singular curves}}, Ann. Scient. \'Ec. Norm. Sup. \textbf{29} (1996), 107--134.


\bibitem[Far09]{F09} Gavril Farkas, {\em{Koszul divisors on moduli spaces of
curves}}, American J. Math. \textbf{131} (2009),
819--869.

\bibitem[FG03]{FG} Gavril Farkas and Angela Gibney, {\em{The Mori cones of moduli spaces
of pointed curves of small genus}}, Transactions American Math. Soc. \textbf{355} (2003), 1183--1199.

\bibitem[FL10]{FL} Gavril Farkas and Katharina Ludwig, {\em{The Kodaira dimension of the moduli
space of Prym varieties}}, J. European Math. Society \textbf{12} (2010), 755--795.

\bibitem[FGSMV]{FGSMV} Gavril Farkas, Sam Grushevsky, Riccardo Salavati Manni and Alessandro Verra, \emph{Singularities of theta divisors and the geometry of $\mathcal{A}_5$}, J. European Math. Society \textbf{16} (2014), 1817--1848.


\bibitem[FV16]{FV14} Gavril Farkas and Alessandro Verra, {\em{The universal abelian variety over $\cA_5$}},  Ann. Scient. \'Ec. Norm. Sup. \textbf{49} (2016), 521--543.

\bibitem[Fra01]{Franzsen-thesis2001} William N. Franzsen, {\em Automorphisms of Coxeter groups}, PhD thesis, University of Sydney, 2001. \url{http://www.maths.usyd.edu.au/u/PG/Theses/franzsen.pdf}

\bibitem[vdGK12]{GK} Gerard van der Geer and Alexis Kouvidakis, {\em{The Hodge bundle on Hurwitz spaces}}, Pure Appl. Math. Quarterly \textbf{7} (2011), 1297--1308.

\bibitem[HM82]{HM} Joe Harris and David Mumford, {\em{On the Kodaira
dimension of $\mm_g$}}, Invent. Math. \textbf{67} (1982), 23--88.

\bibitem[IL12]{IL}
Elham Izadi and Herbert Lange, \emph{Counter-examples of high Clifford index to Prym-Torelli}, J. Algebraic Geom. \textbf{21} (2012), no. 4, 769--787.

\bibitem[ILS09a]{IzLaSt}
Elham Izadi, Herbert Lange, and Volker Strehl, \emph{Correspondences with split polynomial equations}, J. Reine Angew. Math. \textbf{627} (2009), 183--212.

\bibitem[ITW16]{ITW}
Elham Izadi, Csilla Tamas and Jie Wang, \emph{The primitive cohomology of the theta divisor of an abelian fivefold}, J. Algebraic Geom., to appear.

\bibitem[IvS95]{IvS}
Elham Izadi and Duco van Straten, \emph{The intermediate Jacobians of the theta divisors of four-dimensional principally polarized abelian varieties}, J. Algebraic Geom. \textbf{4} (1995), no. 3, 557--590.


\bibitem[Kan87]{Kan87} Vassil Kanev, \emph{Principal polarizations on Prym-Tyurin varieties}, Compositio Math. \textbf{64} (1987), 243--270.

\bibitem[Kan89a]{Kan1} Vassil Kanev, \emph{Intermediate Jacobians and Chow groups of threefolds with a pencil of del Pezzo surfaces}, Annali Mat. Pura ed Applicata \textbf{154} (1989), 13--48.

\bibitem[Kan89b]{kanev1989spectral-curves}
Vassil Kanev, \emph{Spectral curves, simple {L}ie algebras, and {P}rym-{T}jurin
  varieties}, Theta functions---{B}owdoin 1987, Proc. Sympos. Pure Math., vol.~49, Amer. Math. Soc.
  1989, 627--645. \MR{1013158 (91b:14028)}

\bibitem[Kan06]{kanev2006hurwitz-spaces}
Vassil Kanev, \emph{Hurwitz spaces of {G}alois coverings of {${\mathbb P}^1$}, whose
  {G}alois groups are {W}eyl groups}, J. Algebra \textbf{305} (2006),
  442--456. \MR{2264138 (2007g:14032)}

\bibitem[KM13]{KM} Sean Keel and James McKernan, {\em{Contractible extremal rays on $\mm_{0, n}$}}, in Handbook of Moduli Vol. 2, 115--130, International Press 2013.

\bibitem[KKZ11]{KKZ} Alexey Kokotov, Dmitry Korotkin and Peter Zograf, {\em{Isomonodromic tau function on the space of admissible covers}}, Advances Math. \textbf{227} (2011), 586--600.

\bibitem[Kol90]{Kol} J\'anos Koll\'ar, {\em{Projectivity of complete moduli}}, J. Differential Geometry \textbf{32} (1990), 235--268.


\bibitem[LR08]{lange2008a-galois-theoretic-approach}
Herbert Lange and Anita~M. Rojas, \emph{A {G}alois-theoretic approach to
  {K}anev's correspondence}, Manuscripta Math. \textbf{125} (2008),
  225--240. \MR{2373083 (2008m:14085)}

\bibitem[Mor98]{Mo98} Atsushi Moriwaki, {\em{Relative Bogomolov inequality and the cone of positive divisors on the moduli space of stable curves}},
J. Amer. Math. Soc.  \textbf{11} (1998), 569--600.

\bibitem[MM83]{MM83} Shigeyumi Mori and Shigeru Mukai, {\em{The uniruledness of the moduli space of curves of genus $11$}}, Springer Lecture Notes in Mathematics  \textbf{1016} (1983),  334--353.

\bibitem[Mum83]{mumford1983on-the-kodaira-dimension}
David Mumford, \emph{{O}n the {K}odaira dimension of the {S}iegel modular
  variety}, Algebraic geometry---open problems ({R}avello, 1982),
  Lecture Notes in Math., vol. 997, Springer, 1983, 348--375. \MR{714757
  (85d:14061)}

\bibitem[M74]{M74} David Mumford, \emph{Prym varieties I},  Contributions to Analysis (a collection
of papers dedicated to Lipman Bers), 325–350, New York, Academic Press 1974,

\bibitem[Nam73]{namikawa1973on-the-canonical-holomorphic}
Yukihiko Namikawa, \emph{On the canonical holomorphic map from the moduli space
  of stable curves to the {I}gusa monoidal transform}, Nagoya Math. J.
  \textbf{52} (1973), 197--259. \MR{0337981 (49 \#2750)}

\bibitem[Nam76]{namikawa1976a-new-compactification-of-the-siegel2}
Yukihiko Namikawa, \emph{{A} new compactification of the {S}iegel space and degeneration
  of {A}belian varieties {II}}, Math. Ann. \textbf{221} (1976), no.~3,
  201--241. \MR{MR0480538 (58 \#697b)}

\bibitem[Osh06]{oshima2006a-classification-of-subsystems}
Toshio Oshima, \emph{A classification of subsystems of a root system},
  arXiv:math/0611904 (2006).



\bibitem[Ver84]{V84} Alessandro Verra, {\em{A short proof of the unirationality of $\cA_5$}}, Indagationes Math. \textbf{46} (1984), 339--355.

\bibitem[Web15]{supporting-computer-computations}
\emph{Website with supporting computer computations},
\url{http://alpha.math.uga.edu/~valery/a6}.

\bibitem[Wir95]{wirtinger1895untersuchungen-uber}
Wilhelm Wirtinger, \emph{Untersuchungen \"uber Thetafunktionen}, Teubner,
  Leipzig, 1895.

\end{thebibliography}

\def\cprime{$'$}
\providecommand{\bysame}{\leavevmode\hbox to3em{\hrulefill}\thinspace}
\providecommand{\MR}{\relax\ifhmode\unskip\space\fi MR }
\providecommand{\MRhref}[2]{%
  \href{http://www.ams.org/mathscinet-getitem?mr=#1}{#2}
}
\providecommand{\href}[2]{#2}

\end{document}